\documentclass[10pt,a4paper,twoside,english]{article}
\usepackage[dvips]{graphicx}
\usepackage{amsfonts}
\usepackage{listings}          
\usepackage[english,francais]{babel}
\usepackage[applemac]{inputenc}    
\usepackage[T1]{fontenc}
\usepackage{fancyhdr}
\usepackage{indentfirst}
\usepackage{array} 
\usepackage{newlfont}
\usepackage{microtype} 
\usepackage{amscd}

\usepackage[fleqn,tbtags]{amsmath}
\usepackage{amssymb,mathrsfs}       
\usepackage{dsfont}                            
\usepackage{mathtools}
\numberwithin{equation}{section}             

\usepackage{latexsym}
\usepackage{amsthm}
\usepackage{fancyhdr}

\theoremstyle{definition}                          
\newtheorem{thm}{Theorem}[section]     
\newtheorem{prop}[thm]{Proposition}      
\newtheorem{cor}[thm]{Corollary}            
\newtheorem{lem}[thm]{Lemma}             
\theoremstyle{definition}               
\newtheorem{dfn}[thm]{Definition}
\newtheorem{ese}[thm]{Example}        
\newtheorem{nota}[thm]{Notation}

\newtheorem{rem}[thm]{Remark}          

\newtheoremstyle{prova}
{3pt}
{3pt}
{}
{}
{\textbf}
{.}
{\newline}
{}
\theoremstyle{prova}

\newcommand{\cvd}{$\hfill \sqcap \hskip-6.5pt \sqcup$} 
\def\R{\mathbb R}
\def\N{\mathbb N}
\def\Z{\mathbb Z}

\def\E{\mathbb E}
\def\P{\mathbb P}
\def\X{\mathbb X}
\def\Y{\mathbb Y}
\def\Z{\mathbb Z}
\def\F{\mathcal{F}}

\def\shd{{\cal D}}
\def\shf{{\cal F}}

\def\shm{{\cal M}}

\def\1{\mathds{1}}

\def\be{\begin{equation}}
\def\ee{\end{equation}}
\def\bea{\begin{eqnarray}}
\def\eea{\end{eqnarray}}
\def\ba{\begin{array}}
\def\ea{\end{array}}
\def\bs{\begin{split}}
\def\es{\end{split}}
\def\apt{\left(}
\def\cpt{\right)}
\def\apq{\left[}
\def\cpq{\right]}
\def\lra{\longrightarrow}

\def\e{\epsilon}

\begin{document}
\selectlanguage{english}

\title {Generalized covariation and extended Fukushima decompositions
 for Banach valued processes.\\
Application to  windows of Dirichlet processes.}

\author{{\sc Cristina DI GIROLAMI}\footnote{LUISS Guido Carli - Libera
    Universit\`a Internazionale degli Studi Sociali Guido Carli di Roma (Italy).}
  {\sc,}\  
\ {\sc and}\ {\sc Francesco RUSSO }\footnote{ENSTA ParisTech,
Unit\'e de Math\'ematiques appliqu\'ees,
32, Boulevard Victor, F-75739 Paris Cedex 15 (France).}\footnote{INRIA Rocquencourt 
and Cermics Ecole des Ponts, Projet MATHFI. Domaine de Voluceau, BP 105
F-78153 Le Chesnay Cedex (France) .} 
}

\date{May 21th, 2011}
\maketitle

\begin{abstract}
This paper concerns a class of Banach valued processes which 
have finite quadratic variation. The notion introduced
here generalizes the classical one,  of M\'etivier and 
Pellaumail which is quite restrictive. 
We make use of the notion 
of $\chi$-covariation which is a generalized notion of covariation 
for processes with values in two Banach spaces $B_{1}$ and $B_{2}$.
 $\chi$ refers to a suitable subspace of the dual of the projective 
tensor product of $B_{1}$ and $B_{2}$.
We investigate some $C^{1}$ type transformations for various classes of 
stochastic processes admitting a $\chi$-quadratic variation and 
related properties. If $\X^1$ and $\X^2$ admit a $\chi$-covariation,
$F^i: B_i \rightarrow \R$, $i = 1, 2$ are of class 
$C^1$ with some supplementary assumptions then the covariation of
the real processes
$F^1(\X^1)$ and $F^2(\X^2)$ exist. \\
A detailed analysis will be devoted
to  the so-called window processes. 
Let $X$ be a real continuous process; 
the $C([-\tau,0])$-valued  process $X(\cdot)$ defined by
$X_t(y) = X_{t+y}$, where $y \in [-\tau,0]$, is called {\it window} process. 
 Special attention is given to transformations of
window processes  associated with 
Dirichlet and weak Dirichlet processes. In fact we aim to generalize
the following properties valid for $B=\R$. If
 $\X=X$ is a real valued Dirichlet process
and $F:B \rightarrow \R$ 
 of class $C^{1}(B)$ 
then $F(\X)$
is still a Dirichlet process. 
If $\X=X$ is a weak Dirichlet process with finite
quadratic variation, 
and $F: C^{0,1}([0,T]\times B)$ is of
class $C^{0,1}$, then $\apt F(t, \X_t) \cpt $ is a weak Dirichlet process.
We specify corresponding  results when $B=C([-\tau,0])$ and $\X=X(\cdot)$.
This will consitute a significant Fukushima decomposition
for functionals of windows of (weak) Dirichlet processes.
As applications, we give a new technique for representing 
path-dependent random variables. 
%
%
\end{abstract}

[\textbf{2010 Math Subject Classification}: \ ] \  
{ 60H05, 60H07, 60H10, 60H30, 91G80}

\medskip


\bigskip

{\bf Key words and phrases} Covariation and Quadratic variation,
Calculus via regularization, Infinite dimensional analysis,
 Tensor analysis, Dirichlet processes,
Representation of path dependent random variables, Malliavin calculus,
Generalized Fukushima decomposition.
\section{Introduction}
The notion of covariation is historically defined for two real valued $(\mathcal{F}_{t})$-semimartingales $X$ and $Y$ and it is denoted by $[X,Y]$.
This notion was extended to the case of general processes by mean of discretization techniques, by \cite{fo}, or via regularization, 
see for instance \cite{rv2,Rus05}. In this paper we will follow the language of regularization; for simplicity we suppose that either $X$ or $Y$ are continuous. 
We propose here a slight different approach than \cite{rv2}.
\begin{dfn}			\label{def cov}
Let $X$ and $Y$ be two real processes such that $X$ is continuous. For $\epsilon >0$, we denote 
\begin{equation}
[X,Y]^{\epsilon}_{t}=\int_{0}^{t}\frac{ \left(X_{s+\epsilon}-X_{s} \right)\left(Y_{s+\epsilon}-Y_{s} \right)}{\epsilon} ds  \; , \quad t> 0 \; .
\end{equation}
We say that $X$ and $Y$ admit a \textbf{covariation} if 
\begin{description}
\item [i)] $\lim_{\epsilon\rightarrow 0} [X,Y]^{\epsilon}_{t}$ exists in probability for every $t>0$ and 
\item [ii)] the limiting process in \textbf{i)} admits
 a continuous modification that will be denoted by $[X,Y]$.
\end{description}
\vspace{0.3cm}
If $[X,X]$ exists,
 we say that $X$ is a {\bf finite quadratic variation process} 
(or it has finite quadratic variation) and 
it is also denoted  by $[X]$.
If $[X] = 0$, $X$ is called {\bf zero quadratic variation process}.\\
We say that $(X,Y)$ admits its (or $X$ and $Y$ admit their)
 {\bf mutual covariations}  if $[X]$, $[Y]$ and $[X,Y]$ exist.
\end{dfn}
\begin{rem}		\label{rem R}
\begin{enumerate}
\item Lemma \ref{lem CPUCP} below allows to show that, whenever
 $[X,X]$ exists, then $[X,X]^\varepsilon$ also converges in the ucp sense 
as intended for instance in 
 the \cite{rv2,Rus05} sense. 
The basic results established there are still valid here, see the
 following items.
\item If $X$ and $Y$ are $(\F_{t})$-local semimartingales, then $[X,Y]$ coincides with the classical covariation, see Corollaries 2 and 3 in \cite{Rus05}. 
\item If $X$ (resp. $A$) is a finite (resp. zero) quadratic variation
 process,
 we have $[A,X]=0$. 
\end{enumerate}
\end{rem}
We recall two useful tools related to the notion of covariation 
for real valued processes, see  Lemma 3.1 from \cite{rv4} and
  Propostion 1.2 in \cite{rv2}. 

%
%
\begin{lem}    \label{lem CPUCP}		
Let $(Z^{\epsilon})_{\epsilon>0}$ be a family of continuous processes
indexed by $[0,T]$. We suppose the following.

1) $\forall \epsilon> 0$, $t\longrightarrow Z^{\epsilon}_{t}$ is increasing.

2) There is a continuous process $(Z_{t})_{t\in [0,T]}$ such that $Z_{t}^{\epsilon}\rightarrow Z_{t}$ in probability for any $t\in [0,T]$ when $\epsilon$ goes to zero.

Then $Z^{\varepsilon}$ converges to $Z$ ucp, where ucp stands for the uniform convergence in probability on each compact.
\end{lem}

\begin{prop}	\label{prFD}
Let $X$ and $Y$ be two continuous 
processes admitting their mutual covariations. 
Then for every cadlag process $H$ on $\R^{+}$ we have 
\[
\int_0^\cdot H_s \frac{\apt X_{s+\e}-X_s\cpt \apt Y_{s+\e}-Y_{s}\cpt}{\e}ds \xrightarrow[\e\lra 0]{ucp} \int_{0}^\cdot H_s d[X,Y]_s  \ .
\]
\end{prop}

\begin{dfn}
\begin{enumerate}
\item An $(\F_t)$-adapted real process $A$ is called {\bf $(\mathcal{F}_t)$-martingale orthogonal process} if $[A,N]=0$ for any
continuous $(\mathcal{F}_{t})$-local martingale $N$.
\item A real process $A$ is called {\bf $(\mathcal{F}_{t})$-strongly predictable} if there is $\rho>0$ such that 
$X_{\cdot+\rho}$ is $(\mathcal{F}_t)$-adapted.
\end{enumerate}
\end{dfn}
Previous notion was introduced in \cite{crwdp}; the proposition
below was the object of Corollary 3.11 in \cite{crwdp}.
\begin{prop}		\label{prop STRP}
Let $A$ be an $(\F_t)$-strongly predictable process and $M$ be an $(\F_t)$-local martingale. 
Then $[A,M]=0$. In particular an $(\F_t)$-strongly predictable process is an $(\F_t)$-martingale orthogonal process.
\end{prop}

Important subclasses of finite quadratic variation processes are
Dirichlet processes. Probably the best denomination should
be F\"ollmer-Dirichlet processes, since a very similar notion
was introduced by \cite{FolDir} in the discretization
framework.

\begin{dfn} 		\label{dfn FtDir}
A real continuous process $X$ is a called \textbf{$(\mathcal{F}_{t})$-Dirichlet
  process} if $X$ admits a decomposition $X=M+A$ where $M$ is an
$(\mathcal{F}_{t})$-local martingale and $A$ is a zero quadratic
variation process. For convenience, we suppose $A_{0}=0$.
\end{dfn}
The decomposition is unique if for
instance $A_0 = 0$, see Proposition 16 in \cite{Rus05}. 
An $(\shf_{t})$-Dirichlet process has in particular finite quadratic variation. 
An $(\mathcal{F}_{t})$-semimartingale is also an $(\mathcal{F}_{t})$-Dirichlet process, 
a locally bounded variation process is in fact a zero quadratic variation process.\\
The concept of $(\shf_{t})$-Dirichlet process can be weakened. An extension of such processes are the so-called 
$(\shf_{t})$-weak Dirichlet processes, which were first introduced and discussed in \cite{er} and \cite{rg2}, but they appeared 
implicitly even in \cite{er2}. Recent developments concerning the subject appear in \cite{cjms,crwdp,st}. 
$(\mathcal{F}_{t})$-weak Dirichlet processes are generally not 
$(\mathcal{F}_{t})$-Dirichlet processes but they still maintain a 
decomposition property.
\begin{dfn}	\label{dfn FtWD} 
A real continuous process $Y$ is called \textbf{$(\mathcal{F}_{t})$-weak Dirichlet} if $Y$
admits a decomposition $Y=M+A$ where $M$ is an
$(\mathcal{F}_{t})$-local 
martingale and $A$ is an $(\mathcal{F}_t)$-martingale orthogonal process. 
For convenience, we will always suppose $A_{0}=0$.
\end{dfn} 
The decomposition is unique, see for instance Remark 3.5 in
\cite{rg2} or again Proposition 16 in \cite{Rus05}. 
Corollary 3.15 in 
\cite{crwdp} makes the following observation. If the
underlying filtration 
$(\shf_{t})$ is the natural 
filtration associated with a Brownian motion $W$, then 
any $(\shf_{t})$-adapted process $A$ is an $(\mathcal{F}_t)$-martingale orthogonal process
if and only if $[A,W]=0$. 
An $(\mathcal{F}_{t})$-Dirichlet process is also an $(\mathcal{F}_{t})$-weak Dirichlet process, a zero quadratic variation process is in fact also an 
$(\shf_{t})$-martingale orthogonal process. 
An $(\mathcal{F}_{t})$-weak Dirichlet process is not necessarily a
finite quadratic variation process; on the other hand, there are 
$(\mathcal{F}_{t})$-weak Dirichlet processes with finite quadratic
variation that are not Dirichlet, see for instance \cite{er2}.
Let $Y$ be an $(\mathcal{F}_{t})$-weak Dirichlet process with
decomposition
 $Y=W+A$, $W$ 
being a $(\mathcal{F}_{t})$-Brownian motion and the process
 $A$ an  $(\mathcal{F}_{t})$-martingale orthogonal process;
if $A$ has with finite quadratic variation, then $Y$ is also a finite
 quadratic variation process and $[Y]=[W]+[A]$.  
In Theorem \ref{thm STWDPR} we will provide another class of examples of $(\shf_{t})$-weak Dirichlet processes 
with finite quadratic variation which are not $(\shf_{t})$-Dirichlet.\\
An important property in stochastic calculus  concerns the conservation 
of the {\it semimartingale} or {\it Dirichlet process} features 
through some real transformations.
Here are some classical  results. 
\begin{description}
\item{a)} The class of real semimartingales with respect to
 a given filtration is known 
to be stable with respect to $C^{2}(\R)$ transformations, i.e. if 
$f\in C^{2}(\R)$ or difference of convex functions, $X$ is an $(\mathcal{F}_{t})$-semimartingale, then $f(X)$ is still an $(\mathcal{F}_{t})$-semimartingale. 
\item{b)} Finite quadratic variation processes are stable under 
$C^{1}(\R)$ transformations.
\item{c)} Also Dirichlet processes are stable
 with respect to $C^{1}(\R)$ transformations.
 If $f\in C^{1}(\R)$ and $X=M+A$ is a real $(\shf_{t})$-Dirichlet
 process with $M$ the 
$(\shf_{t})$-local martingale and $A$ the zero quadratic variation
 process, 
then $f(X)$ is still an $(\shf_{t})$-
Dirichlet process whose decomposition is 
$f(X)=\tilde{M}+\tilde{A}$, where $\tilde{M}_{t}=f(X_{0})+ 
\int_{0}^{t}f'(X_{s})dM_{s}$ and $\tilde{A}_{t}=f(X_{t})-\tilde{M}_{t}$; 
see \cite{ber} and \cite{rvw} for details.
\item{d)}  In some applications, in particular to control theory 
(as illustrated in \cite{rg1}), one often needs to know the nature of
 process $(f(t,X_{t}))$  where $f\in C^{0,1}(\R^{+}\times \R)$ and 
$X$ is a real continuous $(\shf_{t})$-weak Dirichlet process with
 finite quadratic variation. 
It was shown in \cite{rg2}, Proposition 3.10, 
that $\big( f(t, X_{t})\big)$ is 
an $(\shf_{t})$-weak Dirichlet process. 
Obviously, $(f(t,X_{t}))$ does not need to be of finite quadratic variation. 
Consider, as an example, $f$ only depending on time, deterministic, with infinite quadratic variation. 
\end{description}

Let $B_{1}$, $B_{2}$ be two general Banach spaces. 
If $\X$ (resp. $\Y$) is a $B_{1}$ (resp. $B_{2}$) valued stochastic process it is not obvious to define an exploitable 
notion of covariation of $\X$ and $\Y$ even if they are $H$-valued martingales and $B_{1}=B_{2}=H$ is a separable Hilbert space. 
In Definition \ref{def CHICOV} we recall the notion of $\chi$-covariation (resp. $\chi$-quadratic
variation) introduced in \cite{DGR1} in reference to a subspace $\chi$ of  the dual 
of $B_{1} \hat{\otimes}_{\pi} B_{2}$, where $\X$ is $B_{1}$-valued and $\Y$ is $B_{2}$-valued. 
When $\chi$ equals the whole space $(B_{1}\hat{\otimes}_{\pi}B_{2})^{\ast}$, we say that $\X$ and $\Y$ admit a global covariation. 
In \cite{mp,dincuvisi} one introduces two historical concepts of quadratic variations related to a Banach valued process $\X$,
 the \textbf{real and tensor quadratic variations}. 
In Definition 1.3, Propositions 1.5, 1.6 and Corollary 1.7 of \cite{DGR1} we recover in our regularization language those notions. 
In Proposition 3.16 of \cite{DGR1} we show that whenever $\X$ has a real and tensor quadratic variation then it 
has the global quadratic variation. 
Many Banach space valued processes do not admit a global quadratic 
variation, even though 
they admit a $\chi$-quadratic variation for some suitable $\chi$, 
see \cite{DGR1}, Section 4 for several examples.
In this paper, given different classes of stochastic processes $\X$
 with values in some 
Banach space $B$ and a functional $F: B \rightarrow \R$
with some Fr\'echet regularity,
we are interested in finding natural sufficient conditions
so that $F(\X)$ is a real finite quadratic variation process,
a Dirichlet or a weak Dirichlet process.
\begin{itemize}
\item In Theorem \ref{thm SCQVA} we show that if $\X$ is a $B$-valued
 process
 with $\chi$-quadratic variation and 
$F:B\rightarrow \R$ is of class $C^{1}$ Fr\'echet with some supplementary 
properties on $DF$, then $F(\X)$ is a real finite quadratic variation 
process.
This constitutes a natural generalization of previous
item b) concerning real valued processes.
\end{itemize}
A typical Banach space which justifies the introduction of the notion 
of $\chi$-quadratic variation is
 $B=C([-\tau,0])$ for some $\tau>0$. 
If $X$ is a real continuous process, 
the $C([-\tau,0])$-valued  process $X(\cdot)$ defined by
$X_t(y) = X_{t+y}$, where $y \in [-\tau,0]$, is called {\bf window
process} (associated with $X$). 
If $X$ is an $(\shf_t)$-Dirichlet (resp.  $(\shf_t)$-weak Dirichlet),
the process $X(\cdot)$ is called {\bf window $(\shf_t)$-Dirichlet} (resp.
{\bf$(\shf_t)$-weak Dirichlet) process}.
For window processes, we obtain more specific results.
We introduce here a notation which will be re-defined in Section 
\ref{sec: pre}.
Let $a$ be the vector $(a_{N},a_{N-1}, \ldots, a_{1},0)$ 
which identifies $N+1$ fixed points on $[-\tau,0]$, $-\tau=a_{N} < a_{N-1}< \ldots a_{1}< a_{0}=0$.
 Space $\mathcal{D}_{a}([-\tau,0])$ denotes the Hilbert space of
 measures $\mu$ on $[-\tau,0]$ which can be written as a sum of
 Dirac's measures
 concentrated on  points  $a_i$, i.e.
 $\mu(dx)=\sum_{i=0}^{N}\lambda_i \delta_{a_i}(dx)$, 
$\lambda_i\in \R$. $\mathcal{D}_{0}([-\tau,0])$ denotes
 the space $\mathcal{D}_{a}([-\tau,0])$ when 
 $a=(0)$, i.e. the linear space of multiples of Dirac's measure concentrated
 in $0$.
The following items are generalizations  of properties  c), d) 
valid for  real processes. We set $B=C([-\tau,0])$.
\begin{itemize} 
\item Let $X$ be an $(\shf_t)$-Dirichlet process, with associated 
window process 
$\X = X(\cdot)$ and again $F:B\longrightarrow \R$ of class $C^{1}$ Fr\'echet. 
Theorem \ref{thm STWDPR} gives conditions so that $F(\X)$ is a real $(\shf_t)$-weak Dirichlet process. 
Under a stronger condition, Theorem \ref{thm STWDIRP} shows that $F(\X)$ is a real $(\shf_t)$-Dirichlet process.
More precisely Theorem \ref{thm STWDPR} (resp. Theorem \ref{thm STWDIRP}) states the following.
Let $F:  B\longrightarrow \R$ be of class $C^{1}\big(B \big)$ in the Fr\'echet sense such that the 
first derivative $DF(\eta)$ at each point $\eta \in B$, belongs to $\mathcal{D}_{a}([-\tau,0])\oplus L^{2}([-\tau,0])$ (resp. $\mathcal{D}_{0}([-\tau,0])\oplus L^{2}([-\tau,0])$). 
We suppose moreover that $DF$, with values in the mentioned space,
is continuous. Then $F(\mathbb{X})$ is a real $(\shf_t)$-weak Dirichlet process (resp. Dirichlet process).
\item 
 Previous item is extended to the case 
when $X$ is an $(\shf_t)$- weak Dirichlet process
with finite quadratic variation in Theorem \ref{thm: thm 1}.
Let $F: [0,T] \times B \longrightarrow \R$
as time dependent of class $C^{0,1}$. Similarly to the case when $B$ is finite dimensional, \cite{rg2}, 
we cannot expect $\big(F(t,X_{t}(\cdot))\big)$ to be a Dirichlet process. 
In general it will not even be a finite quadratic variation process. 
In Theorem \ref{thm: thm 1} we state the following. Suppose that the first derivative $DF(t,\eta)$, 
at each point $(t, \eta) \in [0,T] \times C([-\tau,0])$, belongs to $\mathcal{D}_{a}([-\tau,0])\oplus L^{2}([-\tau,0])$. 
We suppose again that $DF$, with values in the mentioned space, is continuous. Then 
$\big(F(t,X_{t}(\cdot))\big)$ is at least a weak Dirichlet process. 
\item If $D F$ does not necessarily live in  
$\mathcal{D}_{a}([-\tau,0]^{2})\oplus L^{2}([-\tau,0])$,
and in some cases even if $t \mapsto  DF(t,\eta)$  for fixed $\eta$ is only stepwise continuous  
but it fulfills a technical condition called the {\it 
support predictability condition} (see Definition \ref{del SPC}),
it is possible to recover the conclusion of previous statement, see Theorems \ref{thm:thm1B} and \ref{thm ter}.
\end{itemize}
One of the consequences  of the paper is that, under some modest conditions on a functional
$F: [0,T] \times B \rightarrow \R$ and on a $B$-valued  process $\X$
which is a window of a semimartingale (with $B= C([-\tau,0])$), it is possible to 
characterize  $F(t, \X_t)$ through a Fukushima type decomposition,
which is unique, and it plays the role of It\^o type formula under
weak conditions. 
The Fukushima decomposition given in Theorems \ref{thm: thm 1}, \ref{thm:thm1B} and \ref{thm ter} 
is innovating at the level of stochastic analysis. In fact it does not concern the decomposition of a functional of an infinite 
dimensional Dirichlet process (or maybe weak Dirichlet); 
in fact, even the window of a semimartingale
is generally not a $B$-valued semimartingale, see Proposition 4.7 
in \cite{DGR1}.\\
In Section \ref{sec: last}, we consider a diffusion  $X$ such that 
$X_t = X_0 + \int_0^t\sigma(r,X_r)dW_r +\int_{0}^t b(r,X_r) dr $
and $\sigma, b: [0,T] \times \R \rightarrow \R$ of class $C^{0,1}$ 
whose partial derivative in the second variable is bounded.
Even if $\sigma$ is possibly degenerate, we give representations
of a class of path dependent random variables $h$ depending on the 
whole history of $X$ via a functional $C^{1}$ Fr\'echet.
A first representation result is Proposition \ref{pr PCS} which is
based on Theorem 
\ref{thm META}. When the process $X$ is a standard Brownian motion we allow the functional not be smooth, see Section \ref{esemSINGOLARE}.
In Section \ref{secPDE} we consider $h$ of the form $f\apt X_{t_1}, \ldots, X_{t_N} \cpt$ $0< t_1< \ldots < t_N=T$ and $f\in C^{2}(\R^{N})$ with polynomial growth. 
In this case the representation can be associated with $N$ PDEs, each-one stated when the time $t$ varies in the subinterval $(t_{i-1},t_{i})$, for $1\leq i\leq N$.
\par
The paper is organized as follows. After this introduction, 
Section \ref{sec: pre} contains general notations and some preliminaries. 
Section \ref{sec: chiqv} will be devoted to the definition of $\chi$-covariation and $\chi$-quadratic variation and some related results. 
In that section we will remind the evaluation of $\chi$-covariation 
and  $\chi$-quadratic variation
 for different classes of processes. 
In Section \ref{sec: stab}, we discuss how a $B$-valued process having some $\chi$-quadratic variation
transforms. In Section \ref{sec:Dirichlet} we concentrate on the case $B= C([-\tau,0]$ and on generalized 
Fukushima decomposition of windows of Dirichlet or weak Dirichlet processes. 
At Section \ref{sec: last} we provide an application to the problem of recovering quasi-explicit
representation formulae for square integrable random variables.

\section{Preliminaries}		\label{sec: pre}

In this section we recall some definitions and notations concerning the
whole  paper.
Let $A$ and $B$ be two general sets such that $A\subset B$;
 $1_{A}: B\rightarrow\{ 0,1\}$ will denote the indicator function of
 the set $A$, so $1_{A}(x)=1$ if $x\in A$ and $1_{A}(x)=0$ if $x\notin A$. 
We also write
 $1_{A}(x)=1_{\{ x\in A\}}$.
Throughout this paper we will denote by 
$(\Omega,\mathcal{F},\P)$ a fixed probability space, equipped with
a given filtration $\mathbb{F}=(\mathcal{F}_{t})_{t\geq 0}$ 
fulfilling the usual conditions. 
Let $E$ and $F$ be Banach spaces over the scalar field
 $\mathbb{R}$. We shall denote by $L(E;F)$ the Banach space of $F$-valued bounded linear maps on $E$ with the norm given by 
 $\|\phi\| =\sup\{ \| \phi(e)\|_F\,:\, \|e\|_{E}\leq 1\}$. 
 When $F=\R$, the topological dual space of $E$ will be denoted simply by $E^{\ast}$. 
If $\phi$ is a linear functional on $E$, we shall denote the value 
of $\phi$ at an element $e\in E$ either by $\phi(e)$ or $\langle \phi, e \rangle$ or even 
$\prescript{}{E^{\ast}}{\langle} \phi, e \rangle_{E}$.
Throughout the paper the symbols $\langle\cdot,\cdot\rangle$ will denote always some type of duality 
that will change depending on the context.
We shall denote the space of $\R$-valued bounded bilinear forms on the product $E\times F$ by $\mathcal{B}(E\times F)$ with the norm given by 
 $\|\phi\|_{\mathcal{B}}
 =\sup\{ | \phi(e,f) |\,:\, \|e\|_{E}\leq 1; \|f\|_{F} \leq 1 \}$. 
 If $a < b$ are two real numbers, $C([a,b])$ will denote the Banach linear
 space of real continuous functions equipped with the uniform norm denoted by $\|\cdot\|_{\infty}$. 
 If $K$ is a compact subset
of $\R^n$, $\mathcal{M}(K)$ will denote the dual space $C(K)^{\ast}$, i.e. the so-called set of finite signed measures on $K$. 
Our principal references about functional analysis and Banach spaces topologies are \cite{ds, bre}.\\
The capital letters $\X,\Y,\Z$ (resp. $X,Y,Z$) will generally denote Banach valued (resp. real valued) processes indexed by the time variable $t\in [0,T]$ with $T>0$. 
A stochastic process $\X$ will  also be denoted by $(\X_{t})_{t\in[0,T]}$.
A $B$-valued (resp. $\R$-valued) stochastic process
 $\X:\Omega\times [0,T]\rightarrow B$ (resp. $\X:\Omega\times [0,T]\rightarrow \R$) 
is said to be measurable if $\X:\Omega\times [0,T]\longrightarrow B$
(resp. $\X:\Omega\times [0,T]\rightarrow \R$) is measurable with
respect to the $\sigma$-algebras $\mathcal{F}\otimes
\mathcal{B}or([0,T])$ and $\mathcal{B}or(B)$
(resp. $\mathcal{B}or(\R)$), $\mathcal{B}or$ denoting the
corresponding Borel $\sigma$-algebra. 
We recall that 
$\mathbb{X}: 
\Omega\times [0,T]\longrightarrow B$ 
(resp. $\R$) 
is said to be \textit{strongly measurable} 
(or \textit{measurable in the Bochner sense}) 
if it is the limit of measurable countable valued functions. 
If $\X$ is measurable and cadlag with $B$ separable then $\X$ is strongly measurable. 
If $B$ is finite dimensional then a measurable process $\X$ is also strongly measurable. 
If nothing else is mentioned, all the processes indexed by $[0,T]$ will be naturally prolonged by continuity setting $\X_{t}=\X_{0}$ for
$t\leq 0$ and $\X_{t}=\X_{T}$ for $t\geq T$.
A similar convention is done for deterministic functions.
 A sequence $(\X^{n})_{n \in \N}$  of continuous $B$-valued processes
indexed by $[0,T]$, will be said to converge \textit{ucp}
(\textit{uniformly convergence in probability}) to a process $\X$ if
$\sup_{0\leq t\leq T}\|\X^{n}_t-\X_t\|_{B}$ converges to zero in probability
when $n \rightarrow \infty$. 
The Fr\'echet space $\mathscr{C}([0,T])$ 
will denote the linear space of continuous real processes 
equipped with the ucp topology and the metric 
$d(\X,\Y)=\mathbb{E}\left[ \sup_{t\in [0,T]}|\X_{t}-\Y_{t}|\wedge 1\right]$. 
%
%
%
We go on with other notations.\\
The direct sum of two Banach spaces $E_{1}$ and $E_{2}$ will be denoted by $E:=E_{1}\oplus E_{2}$. $E$ is still a Banach space under the $2$-norm defined by 
$\|e_{1}+e_{2}\|_{E}:=(\|e_{1}\|^{2}_{E_{1}} + \|e_{2}\|^{2}_{E_{2}} )^{1/2} $. 
If each of the spaces $E_{i}$ is a Hilbert space then $E$ coincides with the uniquely determined Hilbert space with scalar product $\langle e,f\rangle_{E}=
\langle e_{1}+e_{2},f_{1}+f_{2}\rangle_{E}=\sum_{i=1}^{2}\langle e_{i},f_{i}\rangle_{i}$, where $\langle \cdot ,\cdot \rangle_{i}$ 
is the scalar product in $E_{i}$.\\
We recall now some basic concepts and results about
 tensor products of two Banach spaces $E$ and $F$. For details and a more
 complete description of these arguments, 
 the reader may refer to \cite{rr}, the case with $E$ and $F$ Hilbert spaces being particularly
  exhaustive in \cite{nevpg}. 
  If $E$ and $F$ are Banach spaces, 
 the Banach space $E\hat{\otimes }_{\pi}F$ (resp. $E\hat{\otimes }_{h}F$) denotes 
 the \textit{projective (resp. Hilbert) tensor product} of the Banach spaces $E$ and $F$. 
If $E$ and $F$ are Hilbert spaces the Hilbert tensor product $E\hat{\otimes }_{h}F$ is a Hilbert space. 
We recall that $E\hat{\otimes }_{\pi}F$ is obtained by a completion of the algebraic tensor product $E\otimes F$ equipped with the projective 
norm $\pi$. Let $\{ x_i\}_{1\leq i \leq n}\subset E$ and $\{ y_i\}_{1\leq i\leq n} \subset F$, for a general element $u=\sum_{i=1}^{n}x_{i}\otimes y_{i} $ in $E\otimes F$, 
$\pi(u)=\inf \left\{  \sum_{i=1}^{n}\|x_{i}\| \, \| y_{i} \|:\, u=\sum_{i=1}^{n}x_{i}\otimes y_{i} \right\}$.
Let $e\in E$ and $f\in F$, 
symbol $e\otimes f$ (resp. $e\otimes ^{2}$) will denote a basic element of the algebraic tensor product $E\otimes F$ (resp. $E\otimes E$).
The space $(E\hat{\otimes}_{\pi}F)^{\ast}$ denotes, as usual, the topological dual of the projective tensor product.
There is an isometric isomorphism between the dual space of the projective tensor product and the space of bounded bilinear forms equipped with the usual norm: 
\begin{equation}				\label{eq IDBILPIDUAL}
(E\hat{\otimes}_{\pi} F)^{\ast}\cong \mathcal{B}(E\times F)\cong L(E;F^{\ast}) \ .
\end{equation}
Through relation 
\begin{equation}
\prescript{}{(E\hat{\otimes}_{\pi}F)^{\ast}}{\langle} T, \sum_{i=1}^{n}x_{i}\otimes y_{i}\rangle_{E\hat{\otimes}_{\pi}F} = 
T\left(\sum_{i=1}^{n}x_{i}\otimes y_{i} \right)=\sum_{i=1}^{n}\tilde{T}(x_{i},y_{i})=\sum _{i=1}^{n}\bar{T}(x_{i})(y_{i})
\end{equation}
we associate a bounded bilinear form $\tilde{T}\in  \mathcal{B}(E\times F)$, a bounded linear functional $T$ on $E\hat{\otimes}_{\pi}F$ and an element $\bar{T}\in L(E;F^{\ast})$. 
In the sequel that identification will be often used without explicit mention.\\
The importance of tensor product spaces and their duals 
is justified  first of all from identification 
\eqref{eq IDBILPIDUAL}. 
In fact the second order Fr\'echet derivative of a real function defined on a Banach space $E$ belongs to $\mathcal{B}(E\times E)$. 
We recall another important property.
\begin{equation}		\label{eq 1.15bis}
\mathcal{M}([-\tau,0]^{2})\subset
\left(C([-\tau,0])\hat{\otimes}_{\pi}C([-\tau,0])\right)^{\ast} \, .
\end{equation}
Let $\eta_{1}$, $\eta_{2}$ be two elements in $C([-\tau,0])$. The element $\eta_{1}\otimes \eta_{2}$ in the 
algebraic tensor product $C([-\tau,0])\otimes^{2}$ will be identified with the element $\eta$ in $C([-\tau,0]^{2})$ 
defined by $\eta(x,y)=\eta_{1}(x)\eta_{2}(y)$ for all $x$, $y$ in $[-\tau,0]$. So if $\mu$ is a measure on
 $\mathcal{M}([-\tau,0]^{2})$, the pairing duality $\prescript{}{\mathcal{M}([-\tau,0]^{2})}{\langle} \mu, \eta_{1}\otimes \eta_{2}\rangle_{C([-\tau,0]^{2})} $ has to be understood as the following pairing duality: 
\begin{equation}\label{eqai}	
\prescript{}{\mathcal{M}([-\tau,0]^{2})}{\langle} \mu, \eta\rangle_{C([-\tau,0]^{2})} =\int_{[-\tau,0]^{2}}\eta(x,y)\mu(dx,dy)=\int_{[-\tau,0]^{2}}\eta_{1}(x)\eta_{2}(y)\mu(dx,dy)  \; .
\end{equation}

Along the paper, spaces $\mathcal{M}([-\tau,0])$ and $\mathcal{M}([-\tau,0]^{2})$ and their subsets will play a central role. 
We will introduce some other notations that will be used in the sequel.
Let $-\tau=a_{N} < a_{N-1}< \ldots a_{1}< a_{0}=0$ be $N+1$ fixed points in $[-\tau,0]$. 
Symbol $a$ 
will refer to the vector $(a_{N},a_{N-1}, \ldots, a_{1},0)$ which 
identifies $N+1$ points on $[-\tau,0]$. 
\begin{itemize}
\item
Symbol $\shd_{i}([-\tau,0])$ (shortly $\shd_{i}$), will denote the 			
one dimensional Hilbert space of multiples of 
Dirac's measure concentrated at $a_{i}\in [-\tau,0]$ , i.e. 
$
\mathcal{D}_{i} ([-\tau,0]):= 
\{\mu \in \mathcal{M}([-\tau,0]);\,s.t. \mu(dx)=\lambda \,\delta_{a_{i}}(dx) \textrm{ with } \lambda\in \mathbb{R} \}\; ;
$
the space $\shd_{0}$ will be the space of multiples of Dirac measure concentrated at $0$.
\item Symbol $\shd_{a}([-\tau,0])$ (shortly $\shd_{a}$), will denote the $(N+1)$-dimensional Hilbert space of 
multiples of Dirac's measure concentrated at $a_{i}\in [-\tau,0]$, $0\leq i\leq N$ , i.e. 
\begin{equation}			\label{eq-def Da}
\mathcal{D}_{a} ([-\tau,0]):= 
\{\mu \in \mathcal{M}([-\tau,0]);\,s.t. \mu(dx)=\sum_{i=0}^{N} \lambda_i \,\delta_{a_{i}}(dx) \textrm{ with } \lambda_i \in \mathbb{R} \} =\bigoplus_{i=0}^{N}\shd_{i}\;.
\end{equation}

\item
Symbol $\shd_{i,j}([-\tau,0]^{2})$ (shortly $\shd_{i,j}$), will denote the one dimensional Hilbert space of the
 multiples of Dirac measure concentrated at 
$(a_{i},a_{j})\in [-\tau,0]^{2}$, i.e.
$
\mathcal{D}_{i,j}([-\tau,0]^{2}):= \{ \mu \in \mathcal{M}([-\tau,0]^{2});\; s.t. \mu(dx,dy)=\lambda \,\delta_{a_{i}}(dx)\delta_{a_{j}}(dy) \textrm{ with } \lambda \in \mathbb{R} \}
\cong \shd_{i}\hat{\otimes}_{h} \shd_{j}\; .
$
The space $\shd_{0,0}$ will be the space of Dirac's measures concentrated at $(0,0)$.
\item
$L^{2}([-\tau,0])$ is a Hilbert subspace of $\shm([-\tau,0])$, as well as $L^{2}([-\tau,0]^{2})\cong L^{2}([-\tau,0])\hat{\otimes}_{h}^{2}$ is a Hilbert subspace of 
$\shm([-\tau,0]^{2})$, both equipped with the norm derived from the usual scalar product.
\item
$\shd_{i}([-\tau,0]) \oplus  L^{2}([-\tau,0]) $ is a Hilbert subspace of $\shm([-\tau,0])$. 
The particular case when $i=0$, the space $\shd_{0}([-\tau,0]) \oplus  L^{2}([-\tau,0]) $, shortly $\shd_{0} \oplus  L^{2}$, 
will be often recalled in the paper.
\item
$\shd_{i}([-\tau,0])\hat{\otimes}_{h} L^{2}([-\tau,0])$ is a Hilbert subspace of $\shm([-\tau,0]^{2})$. 
%
%
%
\item
$\chi^{2}([-\tau,0]^{2})$, $\chi^{2}$ shortly, the Hilbert space defined as follows.
\begin{equation}   \label{eq-def chi2}
L^{2}([-\tau,0]^{2}) \oplus 
\bigoplus_{i=0}^{N} \left(  L^{2}([-\tau,0]) \hat{\otimes}_{h}\mathcal{D}_{i}([-\tau,0]) \right) \oplus
\bigoplus_{i=0}^{N} \left( \mathcal{D}_{i} ([-\tau,0]) \hat{\otimes}_{h} L^{2}([-\tau,0])\right)  \oplus 
\bigoplus_{i,j=0}^{N}\shd_{i,j}([-\tau,0]^{2}) \; .
\end{equation}
\item
As a particular case of $\chi^{2}([-\tau,0]^{2})$ we will denote $\chi^{0}([-\tau,0]^{2})$, $\chi^{0}$ shortly, 
the subspace of measures defined as 
\begin{equation}   \label{eq-def chi0}
\chi^{0} ([-\tau,0]^{2})=
L^{2}([-\tau,0]^{2}) \oplus
L^{2}([-\tau,0])\hat{\otimes}_{h}\mathcal{D}_{0}([-\tau,0]) \oplus
\mathcal{D}_{0}([-\tau,0])\hat{\otimes}_{h} L^{2}([-\tau,0])\oplus
\mathcal{D}_{0,0}([-\tau,0]^{2})  \; .
\end{equation}
\end{itemize}
Let $B$ be a Banach space. A function $F:[0,T] \times B \longrightarrow 
\mathbb{R} $, is said to be $C^{1,2}([0,T]\times B)$ (Fr\'echet), 
or $C^{1,2}$ (Fr\'echet), if the following properties are fulfilled.
\begin{itemize}
\item $F$ is once continuously differentiable; the partial derivative
 with  respect to $t$ will be denoted by $\partial_{t} F :[0,T]\times B 
\longrightarrow \mathbb{R}$;
\item for any $t \in [0,T]$,  $x \mapsto DF(t,x)$ is of class $C^1$
where $DF:[0,T]\times B \longrightarrow B^*$ denotes the derivative with respect to the second argument; 
\item the second order derivative with respect to the 
second argument $D^2F: [0,T] \times B \rightarrow (B\hat{\otimes}_\pi B)^\ast $
is  continuous.
\end{itemize}
If $B=C([-\tau,0])$, we remark that $DF$ defined on $[0,T]\times B$ takes values in $B^* \cong \mathcal{M}([-\tau,0])$.
For all $(t,\eta)\in [0,T]\times C([-\tau,0])$, we will denote by $D_{dx} F(t,\eta )$ the measure such that 
\begin{equation}   \label{eq duality deriv prima}
\prescript{}{\mathcal{M}([-\tau,0])}{\langle} DF(t,\eta), h \rangle_{C([-\tau,0])}=DF(t,\eta)(h)=\int_{[-\tau,0]} h(x)D_{dx} F(t,\eta )  \hspace{1cm} \forall \; h\in C([-\tau,0]).
\end{equation}
Recalling \eqref{eq 1.15bis}, if $D^{2}F\,(t, \eta)\in \mathcal{M}([-\tau,0]^{2})$ 
for all $(t,\eta) \in [0,T]\times C([-\tau,0])$ (which will happen in most of
 the treated cases) 
we will denote with $D^{2}_{dx\,dy} F(t, \eta)$, or $D_{dx}D_{dy}F(t,\eta)$, the measure on $[-\tau,0]^{2}$ such that 
\begin{equation}		\label{eq duality deriv seconda}
\prescript{}{\mathcal{M}([-\tau,0]^{2})}{\langle} D^{2}F(t,\eta), g \rangle_{C([-\tau,0]^{2})} =
D^{2}F(t,\eta)(g)  = \int_{[-\tau,0]^{2}} g(x,y)\,D^{2}_{dx\,dy} F(t,\eta) \hspace{1cm} \forall \; g\in C([-\tau,0]^{2}).\\
\end{equation}
A useful notation that will be used along all the paper is the following.
\begin{nota}			\label{nota MEASURE}
Let $F:[0,T]\times C([-\tau,0])\longrightarrow \R$ be a Fr\'echet differentiable function, with Fr\'echet derivative 
$DF:[0,T]\times C([-\tau,0])\longrightarrow \mathcal{M}([-\tau,0])$.
 For any given $(t,\eta)\in [0,T]\times C([-\tau,0])$ and $a \in [-\tau,0]$, 
we denote by $D^{ac}F\;(t,\eta)$ the absolutely continuous part of measure 
$DF\,(t,\eta)$, and by 
$D^{\delta_{a}}F \,(t,\eta) := DF \,(t,\eta)(\{a\})$.
For every $\eta \in C([-\tau,0])$, we observe that 
$t \mapsto D^{\delta_{a}}F(t,\eta)$
 is a real valued function. \\
We denote $D^{\perp}F\, (t,\eta)= DF\, (t,\eta)-DF\, (t,\eta)
(\{ 0\})\delta_{0}$.
\end{nota}
\begin{ese}{If for example $DF(t,\eta)\in \shd_{0}\oplus L^{2}([-\tau,0])$
 for every $(t,\eta)\in [0,T]\times C([-\tau,0])$, then we will often 
write}
\begin{equation}		\label{eq notatMEASURES}
D_{dx} F\,(t,\eta)=D^{\delta_{0}}F\,(t,\eta)\delta_{0}(dx)  +D^{ac}_{x}F\,(t,\eta)dx \; .
\end{equation}
\end{ese}
%
%
\section{Notions of $\chi$-covariation between Banach valued processes}		\label{sec: chiqv}
Let $B_{1}$, $B_{2}$ be two Banach spaces. Whenever $B_1=B_2$ we will denote it simply by $B$. 
%
\begin{dfn}		\label{DefChiCOV}
A Banach subspace $\left( \chi, \|\cdot\|_{\chi} \right)$ continuously injected into $(B_{1}\hat{\otimes}_{\pi}B_{2})^{\ast}$ 
will be called a {\bf Chi-subspace} (of $(B_{1}\hat{\otimes}_{\pi}B_{2})^{\ast}$).
\end{dfn}
 \begin{rem}\label{RPairing}
Obviously the pairing between  $(B_{1}\hat{\otimes}_{\pi}B_{2})^{\ast}$ 
and  $(B_{1}\hat{\otimes}_{\pi}B_{2})^{\ast \ast}$
is compatible with the paring between
$\chi$ and $\chi^\ast$. 
\end{rem}
\begin{ese} When $B=C([-\tau, 0])$, typical examples of Chi-subspace of $(B\hat{\otimes}_{\pi}B)^{\ast}$ 
are $\mathcal{M}([-\tau,0]^{2})$ equipped with the total variation
norm 
and all Hilbert closed subspaces of $\mathcal{M}([-\tau,0]^{2})$. 
For instance $L^{2}([-\tau,0]^{2})$, $\shd_{i}([-\tau,0])\hat{\otimes}_{h} L^{2}([-\tau,0])$, $\shd_{i, j}([-\tau,0]^{2})$, for $1\leq i, j \leq N$, $\chi^{2}([-\tau,0]^{2})$ and $\chi^{0}([-\tau,0]^{2})$.
\end{ese}

We recall now the notion of $\chi$-covariation between a $B_{1}$-valued stochastic process $\X$ and a $B_{2}$-valued stochastic process $\Y$. 
We suppose $\mathbb{X}$ to be a continuous $B_{1}$-valued stochastic process and $\mathbb{Y}$ to be a strongly measurable 
$B_{2}$-valued stochastic process such that $\int_{0}^{T}\|\mathbb{Y}_{s}\|_{B^{\ast}} ds <+\infty$ a.s.
We remind that $\mathscr{C}([0,T])$ denotes the space of continuous processes equipped with the ucp topology.\\
Let $\chi$ be a Chi-subspace of $(B_{1}\hat{\otimes}_{\pi}B_{2})^{\ast}$ and $\epsilon > 0$. 
We denote by $[\X,\Y]^{\epsilon}$,
the following application
\begin{equation}		\label{eq Xepsilon}
[\X,\Y]^{\epsilon}:\chi\longrightarrow \mathscr{C}([0,T])
\quad \textrm{defined by}\quad
\phi
\mapsto
\left( \int_{0}^{t} \prescript{}{\chi}{\langle} \phi,
\frac{J\left(  \left(\X_{s+\epsilon}-\X_{s}\right)\otimes \left(\Y_{s+\epsilon}-\Y_{s}\right)  \right)}{\epsilon} 
\rangle_{\chi^{\ast}} \,ds 
\right)_{t\in [0,T]}
\end{equation}
where $ J: B_{1}\hat{\otimes}_{\pi}B_{2} \longrightarrow (B_{1}\hat{\otimes}_{\pi}B_{2})^{\ast\ast}$ is the canonical injection between a space and its bidual. 
With application $[\X,\Y]^{\epsilon}$ it is possible to associate
 another one, 
denoted by $\widetilde{[\X,\Y]}^{\epsilon}$, defined by
\[
\widetilde{[\X,\Y]}^{\epsilon}(\omega,\cdot):[0,T]\longrightarrow \chi^{\ast}
\quad \textrm{given by}\quad
t\mapsto \left(\phi\mapsto
\int_{0}^{t}  \prescript{}{\chi}{\langle} \phi,\frac{J\left(  
  \left(\X_{s+\epsilon}(\omega)-\X_{s}(\omega)\right)\otimes 
\left(\Y_{s+\epsilon}(\omega)-\Y_{s}(\omega)\right) \right)}{\epsilon}
 \rangle_{\chi^{\ast}} \,ds\right)  \, .
\]
\begin{dfn}		\label{def CHICOV} 
Let $B_{1}$, $B_{2}$ be two Banach spaces and $\chi$ be a Chi-subspace of 
$(B_{1}\hat{\otimes}_{\pi} B_{2})^{\ast}$. 
Let $\X$ (resp. $\Y$) be a continuous $B_{1}$ (resp. strongly measurable $B_{2}$) valued stochastic process such that $\int_{0}^{T}\|\mathbb{Y}_{s}\|_{B^{\ast}} ds <+\infty$ a.s.. 
We say that {\bf $\X$ and $\Y$ admit a $\chi$-covariation} if 
\begin{description}
\item[H1] For all $(\epsilon_{n})$ there exists a subsequence $(\epsilon_{n_{k}})$ such that 
\begin{equation} 
\begin{split}
&\sup_{k}\int_{0}^{T} \sup_{\|\phi\|_{\chi}\leq 1}\left| \langle \phi,\frac{(\X_{s+\epsilon_{n_{k}}}-\X_{s})\otimes(\Y_{s+\epsilon_{n_{k}}}-\Y_{s})}{\epsilon_{n_{k}}}\rangle \right|ds\\
&=\sup_{k}\int_{0}^{T} \frac{\left\| (\X_{s+\epsilon_{n_{k}}}-\X_{s})\otimes(\Y_{s+\epsilon_{n_{k}}}-\Y_{s})\right\|_{\chi^{\ast}} }{\epsilon_{n_{k}}} ds
\;< \infty\; a.s.
\end{split}
\end{equation}
\item[H2]
\begin{description}
\item{(i)} There exists an application $\chi\longrightarrow
  \mathscr{C}([0,T])$, denoted by $[\X,\Y]$, such that
\begin{equation}		\label{H2 cONDucp}
[\X,\Y]^{\epsilon}(\phi)\xrightarrow[\epsilon\longrightarrow 0_{+}]{ucp} [\X,\Y](\phi) 
\end{equation} 
for every $\phi \in \chi\subset
(B_{1}\hat{\otimes}_{\pi}B_{2})^{\ast}$.
\item{(ii)} 
There is a measurable process $\widetilde{[\X,\Y]}:\Omega\times [0,T]\longrightarrow \chi^{\ast}$, 
such that
\begin{itemize}
\item for almost all $\omega \in \Omega$, $\widetilde{[\X,\Y]}(\omega,\cdot)$ is a (cadlag) bounded variation process, 
\item 
$\widetilde{[\X,\Y]}(\cdot,t)(\phi)=[\X,\Y](\phi)(\cdot,t)$ a.s. for all $\phi\in \chi$.
\end{itemize}
\end{description}
\end{description}
If $\X$ and $\Y$ admit a $\chi$-covariation we will call
 $\chi$-\textbf{covariation} of $\X$ and $\Y$ the $\chi^{\ast}$-valued
process  $(\widetilde{[\X,\Y]})_{0\leq t\leq T}$ defined for every
$\omega\in \Omega $ and $t\in [0,T]$ by $\phi \mapsto
\widetilde{[\X,\Y]}(\omega,t)(\phi)=[\X,\Y](\phi)(\omega,t) $. 
By abuse of notation, $[\X,\Y]$ will also be often called $\chi$-covariation
and it will be confused with $\widetilde{[\X,\Y]}$.\\
\end{dfn}
\begin{dfn}			\label{DChiQV}   
Let $\X=\Y$ be a $B$-valued stochastic process and $\chi$ be a Chi-subspace of $(B\hat{\otimes}_{\pi}B)^{\ast}$. 
The $\chi$-covariation $[\X,\X]$ (or $\widetilde{[\X,\X]}$) will also be denoted 
by $[\X]$ (or $\widetilde{[\X]}$), it will be called \textbf{$\chi$-quadratic variation of $\X$} and we will say that $\X$ has a $\chi$-quadratic variation.
\end{dfn}
\begin{dfn}		\label{DQVWS}
If the $\chi$-covariation exists for $\chi=(B_{1}\hat{\otimes}_{\pi} B_{2})^{\ast}$, we say that $\X$ and $\Y$ admit a \textbf{global covariation}. 
Analogously if $\X$ is $B$-valued and the $\chi$-quadratic variation exists for $\chi=(B\hat{\otimes}_{\pi} B)^{\ast}$, we say that $\X$ admits a \textbf{global quadratic variation}.
\end{dfn}
We recall Corollary 3.2 from \cite{DGR1}, which generalizes Proposition \ref{prFD} in the Banach spaces framework. 
\begin{prop}       \label{pr CONVCCOV}
Let $B_{1}$, $B_{2}$ be two Banach spaces and $\chi$ be a Chi-subspace of 
$(B_{1}\hat{\otimes}_{\pi} B_{2})^{\ast}$. 
Let $\X$ and $\Y$ be two stochastic processes with values in 
$B_{1}$ and $B_{2}$ admitting a $\chi$-covariation; 
let $\mathbb{H}$ be a continuous measurable process $\mathbb{H}:\Omega\times [0,T] \longrightarrow \mathcal{V}$ where $\mathcal{V}$ is a 
closed separable subspace of $\chi$. 
Then for every $t\in [0,T]$
\begin{equation}   		\label{eq SDFR}  
\int_{0}^{t} \prescript{}{\chi}{\langle} \mathbb{H}(\cdot,s),d\widetilde{[\X,\Y]}^{\epsilon}(\cdot,s)\rangle_{\chi^{\ast}}
\xrightarrow[\epsilon \longrightarrow 0]{} 
\int_{0}^{t} \prescript{}{\chi}{\langle} \mathbb{H}(\cdot,s),d\widetilde{[\X,\Y]}(\cdot,s)\rangle_{\chi^{\ast}}
\end{equation}
in probability.
\end{prop}

We recall some evaluations of $\chi$-covariations and $\chi$-quadratic variations for window processes given in Section 4 of \cite{DGR1}, in particular we refer to 
Proposition 4.9 and Corollary 4.10.
\begin{prop}   \label{pr QV123}
Let $0 <\tau \leq T$ and we make the same conventions about
vector $a = (a_N=-\tau, \ldots, a_0=0)$ as those introduced after \eqref{eqai}.
Let $X$ and $Y$ be two real continuous processes with finite quadratic
variation. 
\begin{enumerate}
\item [1)] $X(\cdot)$ and $Y(\cdot)$ admit a zero $\chi$-covariation, where $\chi=L^{2}([-\tau,0]^{2})$.
\item [2)] $X(\cdot)$ and $Y(\cdot)$ admit zero $\chi$-covariation for every given $i\in \{0,\ldots, N\}$, where $\chi=L^{2}([-\tau,0])\hat{\otimes}_{h}  \shd_{i}([-\tau,0])$ and 
$\shd_{i}([-\tau,0])\hat{\otimes}_{h}   L^{2}([-\tau,0])$.
\end{enumerate}
If moreover the covariation $[X_{\cdot+a_{i}},Y_{\cdot+a_{j}}]$ 
exists for a given $i, j\in \{ 0,\ldots, N\}$, 
 the following statements hold.

\begin{enumerate}
\item [3)]
 $X(\cdot)$ and $Y(\cdot)$ admit a $\chi$-covariation, where $\chi=\shd_{i,j}([-\tau,0]^{2})$ and it equals 
\begin{equation}		\label{eq QV Dij}
[X(\cdot), Y(\cdot)] (\mu)=\mu(\{a_{i},a_{j}\})[X_{\cdot+a_{i}},Y_{\cdot+a_{j}}], \quad \forall \mu \in \shd_{i,j}([-\tau,0]^{2}).
\end{equation}
\item [4)] In the case $i=j=0$, i.e. $X$ and $Y$ admit a covariation $[X,Y]$, then 
$X(\cdot) $ and $Y(\cdot)$ admit $\chi^{0}([-\tau,0]^{2})$-covariation which equals 
\begin{equation}		\label{eq QVCHI0}
[X(\cdot), Y(\cdot)](\mu)=\mu(\{0,0\}) [X,Y], \quad \forall \mu \in \chi^0.
\end{equation}
\end{enumerate}
If $[X_{\cdot+a_{i}},Y_{\cdot+a_{j}}]$ exists
for all $i,j=0,\ldots,N $, then
\begin{enumerate}
\item [5)] 
 $X(\cdot)$ and $Y(\cdot)$ admit a $\chi^{2}([-\tau,0]^{2})$-covariation which equals 
\begin{equation}				\label{eq CHI0-QUADR}
[X(\cdot),Y(\cdot)](\mu)=\sum_{i,j=0}^{N}\mu(\{a_{i},a_{j}\})[X_{\cdot+a_{i}},Y_{\cdot+a_{j}}],\quad  \forall \mu \in \chi^{2}([-\tau,0]^{2}) \, . 
\end{equation}
\end{enumerate}
\end{prop} 
As application of Proposition \ref{pr QV123}  we obtain the following.
\begin{cor} \label{corCDWMbis} 
Let $X$ be a real $(\shf_{t})$-weak Dirichlet process with finite 
quadratic variation and decomposition $X=M+A$, $M$ being its 
$(\shf_{t})$-local martingale component. 
Let $N$ be a real $(\shf_{t})$-martingale. 
We set $\chi = \shd_{0,0} \oplus \chi_2$ with
\begin{equation} \label{DecChi2}
\chi_2 = \oplus_{i=1}^{N}\shd_{i, 0}\oplus \apt
 L^{2}([-\tau,0])\hat{\otimes}_h \shd_{0}\cpt.
\end{equation}
We have the following.
\begin{enumerate}
\item  $X(\cdot)$ and $N(\cdot)$ admit a $\shd_{0,0}$-covariation
  given, for $\mu\in \shd_{0,0}$, by 
\be \label{fr}
[X(\cdot),N(\cdot)](\mu)=\mu(\{0,0\})[M,N] \ .
\ee
\item
  $X(\cdot)$ and $N(\cdot)$ admit a zero  
$\chi_2$-covariation.
\item 
  $X(\cdot)$ and $N(\cdot)$ admit a $\chi$-covariation where 
for any $\mu \in \chi$, \eqref{fr} holds.  
\item $X(\cdot)$ and $N(\cdot)$ admit a $\chi^{0}$-covariation given,
 for $\mu\in \chi^{0}$, by \eqref{fr}.
\end{enumerate}
\end{cor}
\begin{cor} 		\label{corCDWM}
Let $X$ be a real $(\shf_{t})$-Dirichlet process 
with  decomposition $X=M+A$, $M$ being its 
$(\shf_{t})$-local martingale component.
 Let $N$ be a real $(\shf_{t})$-martingale. Then we have the following.
\begin{enumerate}
\item $X(\cdot)$ admits a 
$\chi^{2}$-quadratic variation given by 
\[
[X(\cdot)](\mu)=\sum_{i=0}^{N}\mu(\{a_{i},a_{i}\}) [M]_{\cdot+a_{i}}  \ .
\]
\item $X(\cdot)$ and $N(\cdot)$ admit a 
$\chi^{2}$-covariation given by 
\[ 
[X(\cdot),N(\cdot)](\mu)
=\sum_{i=0}^{N}\mu(\{a_{i},a_{i}\}) [M,N]_{\cdot+a_{i}} \ .
\]
\end{enumerate}
\end{cor}

\begin{rem}
More details about Dirichlet processes and their properties will be given in section \ref{sec:Dirichlet}. Examples of finite quadratic variation 
weak Dirichlet processes are provided in Section 2 of \cite{er2}. For an $(\mathcal{F}_{t})$-weak Dirichlet process $X$ the 
covariations $[X_{\cdot+a_{i}},X_{\cdot+a_{j}}]$ are not a priori determined. 
\end{rem}
\begin{proof}[Proof of Corollary \ref{corCDWMbis}]
\begin{enumerate}
\item Using Proposition \ref{prop STRP}  $[X,Y] = [M,N]$. So this point follows by item 3) of Proposition \ref{pr QV123}. 
\item We keep in mind the direct sum decomposition of $\chi_2$
given in \eqref{DecChi2}.
We compute the $\oplus_{i=1}^{N}\shd_{0, i}$-covariation. 
Covariations $[X_{\cdot+a_{i}},N]=0$ because it is the sum of
 $[M_{\cdot+a_{i}},N]$ and $[A_{\cdot+a_{i}},N]$ which are zero 
by Proposition \ref{prop STRP} for $i=1,\ldots, N$. 
Using item 3) of Proposition \ref{pr QV123}, $X(\cdot)$ and 
$N(\cdot)$ have zero $\oplus_{i=1}^{N}\shd_{i, 0}$-covariation. 
By item 2) of Proposition \ref{pr QV123}, $X(\cdot)$ and $N(\cdot)$
 have zero $L^{2}([-\tau,0] \hat{\otimes}_{h} \shd_{i}$-covariation
for every $i$. 
Proposition 3.18 in \cite{DGR1} concludes the proof of item 2, since
 it allows to express the $\chi$-covariation in a sum of $\chi$-covariation whenever $\chi$ is a direct sum of Chi-subspaces.
\item It follows by 1., 2. and again by Proposition 3.18 in \cite{DGR1}.
\item We know that  that $[X,N]=[M,N]$. So this point follows by item 4) of Proposition \ref{pr QV123}.
\end{enumerate}
\end{proof}

\begin{proof}[Proof of Corollary \ref{corCDWM}]
\begin{enumerate}
\item  If $i \neq j$, 
by Proposition \ref{prop STRP} and Remark \ref{rem R}, it follows 
that  $[X_{\cdot+a_{i}},X_{\cdot+a_{j}}]=0$.
If $i = j$, by Remark \ref{rem R} and by definition of 
quadratic variation we get
$[X_{\cdot+a_{i}}] = [M]_{\cdot + a_i}$.
The result follows by item 5) of Proposition \ref{pr QV123}.
\item  Similarly as in the proof of item 1. we have that 
$[X_{\cdot + a_{i}}, N_{\cdot+a_{j}}]=0$ if $i\neq j$ and $[X_{\cdot +
    a_{i}}, N_{\cdot+a_{i}}]
=[M,N]_{\cdot+a_{i}}$. The result follows as a 
consequence of item 5) of Proposition \ref{pr QV123}.
\end{enumerate}
\end{proof}
Other interesting results about $\chi$-covariation and $\chi$-quadratic variation for a window of a finite quadratic variation process are given in 
Proposition 6.4 in \cite{DGRnote} and with more details in \cite{DGR1}, Propositions 4.16 and 4.18.

\section{Transformation of $\chi$-quadratic variation and of $\chi$-covariation}		\label{sec: stab}
Let $X$ be a real finite quadratic variation process and $f\in C^{1}(\R)$. We recall that 
$f(X)$ is again a finite quadratic variation process. 
We will illustrate some natural generalizations to the infinite
dimensional 
framework.
In this section, we analyze how transform Banach valued processes having a $\chi$-covariation through $C^{1}$ Fr\'echet differentiable functions. 
We first recall the finite dimensional case framework, see 
\cite{flru1} Remark 3.

%
\begin{prop}		\label{pr STMLT} 
 Let $X=(X^{1},\ldots,X^{n})$ be a $\R^{n}$-valued process having all its mutual covariations $[X^{i},X^{j}]_{t}$ and 
 $F$, $G \in C^{1}(\mathbb{R}^{n})$. 
 Then the covariation $ \left[F(X),G(X) \right]$ exists and is given by
 \begin{equation}			\label{eq STMLT}
 \left[F(X),G(X) \right]_{\cdot}=
 \sum_{i,j=1}^{n} \int_{0}^{\cdot} \partial_{i} F(X)\partial_{j}G(X)d[X^{i},X^{j}]
 \end{equation}
 \end{prop}
This includes the case of Proposition 2.1 in \cite{rv2}, setting $n=2$, $F(x,y)=f(x)$, $G(x,y)=g(y)$, $f, g\in C^{1}(\R)$.\\

When the value space  is a  general Banach space,
we need to recall some other preliminary results.
\begin{prop}		\label{prop RRCBL}
Let $E$ be a Banach space, $S,T:E \longrightarrow \R$ be linear continuous forms. There is a unique linear 
continuous forms from $E\hat{\otimes}_{\pi}E$ to $\mathbb{R}\hat{\otimes}_{\pi}\mathbb{R} \cong\R$, denoted by 
$S\otimes T$, such that $S\otimes T(e_{1}\otimes e_{2})=S(e_{1})\cdot T(e_{2})$ and $\|S\otimes T\|=\|S\|\,\|T\|$.
\end{prop}
\begin{proof}
See Proposition 2.3 in \cite{rr}.
\end{proof}
\begin{rem} 		\label{rem 7.2}
\begin{enumerate}
\item If $T=S$, we will denote $S\otimes S=S\otimes^{2}$.
\item Let $B$ be a Banach space and $F\, ,G:E\longrightarrow \mathbb{R}$ of class $C^{1}(E)$ in the 
Fr\'echet sense. If $x$ and $y$ are fixed, $DF(x)$ and $DF(y)$ 
are linear continuous form from 
$E$ to $\R$. 
We remark that the symbol $DF(x)\otimes DF(y)$ is defined according to Proposition \ref{prop RRCBL}, we 
insist on the fact that ``a priori'' $DF(x)\otimes DF(y)$ does not denote an element of some tensor product $E^{\ast}\otimes E^{\ast}$.
\end{enumerate}
\end{rem}
When $E$ is a Hilbert space, the application
 $S\otimes T$ of Proposition \ref{prop RRCBL} can be further specified.
\begin{prop}		\label{prop RRCBLbis}
Let $E$ be a Hilbert space, $S$, $T\in E^{\ast}$ and $\mathcal{S}$, $\mathcal{T}$ the associated elements in $E$ via Riesz identification. 
$S\otimes T$ can be characterized as the continuous bilinear form
\begin{equation}		\label{eq EEEE}
S\otimes T (x\otimes y)= \langle \mathcal{S},\mathcal{T}\rangle _{E}\cdot 
\langle x , y \rangle_{E}= \langle \mathcal{S}\otimes \mathcal{T}, x\otimes y\rangle_{E\hat{\otimes}_{h}E}, \quad  \forall x,y \in E.
\end{equation}
In particular the linear form $S\otimes T$ belongs to $(E\hat{\otimes}_{h}E)^{\ast}$ and via Riesz it is identified with the tensor product $\mathcal{S}\otimes \mathcal{T}$.
That Riesz identification will be omitted in the sequel.
\end{prop}
\begin{proof}
The application $\phi$ defined in the right-side of \eqref{eq EEEE} belongs to $(E\hat{\otimes}_{h}E)^{\ast}$ by construction.
Since $(E\hat{\otimes}_{h}E)^{\ast} \subset  
(E\hat{\otimes}_{\pi}E)^{\ast}$, it also belongs
to $(E\hat{\otimes}_{\pi}E)^{\ast}$.
Moreover we have 
\[
\left\|\phi\right\|_{\mathcal{B}}=\sup_{\|f\|_{E}\leq 1,\|g\|_{E}\leq 1} 
\left| \phi(f,g)\right|=
\sup_{\|f\|_{E}\leq 1} \vert \langle \mathcal{S}, f \rangle \vert
    \sup_{\|g\|_{E}\leq 1} \vert \langle \mathcal{T}, g \rangle \vert =
\left\|S\right\|_{E^{\ast}}\left\|
 T\right\|_{E^{\ast}}\; .
\]
By uniqueness in Proposition \ref{prop RRCBL}, $\phi$ must coincide with $S\otimes T$.
\end{proof}
As  application of Proposition \ref{prop RRCBLbis},
 setting the Hilbert space $E=\mathcal{D}_{a}\oplus L^{2}([-\tau,0])$, 
we state the following useful result that will be often used in Section \ref{sec:Dirichlet} devoted to $C([-\tau,0])$-valued window processes.
\begin{ese}   \label{pr STCHI2} 
Let $F^{1}$ and $F^{2}$ be two functions from $C([-\tau,0])$ to $\mathcal{D}_{a}\oplus L^{2}([-\tau,0])$ such that $\eta\mapsto F^{j}(\eta)=\sum_{i=0,\ldots N}\lambda^{j}_{i}(\eta)\delta_{a_{i}}+g^{j}(\eta)$ with $\eta\in C([-T,0])$, $\lambda^{j}_{i}: C([-\tau,0])\longrightarrow \mathbb{R}$ and $g^{j}: C([-\tau,0])\longrightarrow L^{2}([-T,0])$ continuous for $j=1,2$.
Then for any $\eta_{1},\eta_{2} \in C([-\tau,0])$,
 $(F^{1}\otimes F^{2})(\eta_{1},\eta_{2})$ will be identified with the true tensor product 
$F^{1}(\eta_{1})\otimes F^{2}(\eta_{2})$ which belongs to $\chi^{2}([-\tau,0]^{2})$. In fact 
we have 
\begin{equation}  \label{prod tensoriale applicazioni}
\begin{split}
F^{1}(\eta_{1})\otimes F^{2}(\eta_{2})
&
=\sum_{i,j=0,\ldots, N}\lambda^{1}_{i}(\eta_{1})\lambda^{2}_{j}(\eta_{2})\delta_{a_{i}}\otimes\delta_{a_{j}} + g^{1}(\eta_{1})\otimes  \sum_{i=0,\ldots,N} \lambda^{2}_{i}(\eta_{2})\delta_{a_{i}}+\\
&\\
& +  \sum_{i=0,\ldots,N} \lambda^{1}_{i}(\eta_{1})\delta_{a_{i}}\otimes g^{2}(\eta_{2})+ g^{1}(\eta_{1}) \otimes g^{2}(\eta_{2})
\end{split}
\end{equation}
\end{ese}

We now state a result related to the generalization of Proposition \ref{pr STMLT} to functions of processes admitting a $\chi$-covariation.  
\begin{thm}				\label{thm SCQVA}
Let $B$ be a separable Banach space, $\chi$ a Chi-subspace of $(B\hat{\otimes}_{\pi} B)^{\ast}$ and 
$\X^{1}$, $\X^{2}$ two $B$-valued continuous stochastic processes admitting a $\chi$-covariation. 
Let $F^{1},F^{2}: B\longrightarrow \mathbb{R}$ be two functions of class $C^{1}$ in the Fr\'echet sense. 
We suppose moreover that the applications 
\begin{displaymath}
\begin{split}
D F^{i}(\cdot)\otimes D F^{j}(\cdot): B\times B & \longrightarrow \chi\subset (B \hat{\otimes}_{\pi} B)^{\ast}\\
(x,y) & \mapsto  D F^{i}(x)\otimes DF^{j}(y)
\end{split}
\end{displaymath}
are continuous for $i,j=1,2$.\\ 
Then, for  every $i,j \in \{1,2\}$,
 the covariation between $F^{i}(\X^{i})$ and $F^{j}(\X^{j})$ exists and is 
given by
\begin{equation} 		\label{eq SCHICOV}
[F^{i}(\X^{i}),F^{j}(\X^{j})] =\int_{0}^{\cdot} 
\langle DF^{i}(\X^{i}_{s})\otimes DF^{j}(\X^{j}_{s}),d
 \widetilde{[\X^{i},\X^{j}]}_{s} \rangle.
\end{equation}
\end{thm}
\begin{rem}			\label{rem JUSTCR}
In view of an application of Proposition \ref{pr CONVCCOV} in the proof of Theorem \ref{thm SCQVA}, we observe 
the following. 
 Since $B$ is separable and  $D F^{i}(\cdot)\otimes D F^{j}(\cdot):
 B\times B  \longrightarrow \chi$ is continuous,  the process 
$H_{t}= DF^{i}(\X^{i}_{t})\otimes DF^{j}(\X^{j}_{t})$ takes values in a separable closed subspace $\mathcal{V}$ of $\chi$.
\end{rem}
\begin{cor}  
Let us formulate the same assumptions as in Theorem  \ref{thm SCQVA}.
If there is a $\chi^{\ast}$-valued stochastic process $\mathbb{H}^{i,j}$ 
 such that $\widetilde{[\X^{i},\X^{j}]}_{s}=\int_{0}^{s}\mathbb{H}^{i,j}_{u}\,du$
in the Bochner sense then
\begin{equation} \label{eq4.5}
[F^{i}(\X^{i}),F^{j}(\X^{j})]_t = \int_{0}^{t} 
\langle DF^{i}(\X^{i}_{s})\otimes DF^{j}(\X^{j}_{s}),
\mathbb{H}^{i,j}_{s}\rangle \,ds, \quad t \in [0,T].
\end{equation}
\end{cor}
\begin{proof} [Proof of Theorem \ref{thm SCQVA}]
We make use in an essential manner of Proposition \ref{pr CONVCCOV}. 
Without restriction of generality we only consider the case
 $F^{1}=F^{2}=F$ and $\X^{1}=\X^{2}=\X$.\\
Let $t \in [0,T]$. By definition of the quadratic variation of a
 real process in Definition
 \ref{def cov},
it will be enough to show that 
 the quantity
$$
\int_{0}^t
\frac{\left(F(\X_{s+\epsilon})-F(\X_{s})\right)^{2}}{\epsilon}ds  \; .
$$
converges in probability to the right-hand side of \eqref{eq4.5}.
Using Taylor's expansion we have 
\begin{displaymath}
\begin{split}
\frac{1}{\epsilon}\int_{0}^{t}
\left(F(\X_{s+\epsilon})-F(\X_{s})\right)^{2}ds 
&
= \frac{1}{\epsilon}\int_{0}^{t} \Big(
\langle DF(\X_{s}),\X_{s+\epsilon}-\X_{s} \rangle +  \\
& \hspace{1.5cm} 
  +\int_{0}^{1} \langle DF\left((1-\alpha)\X_{s}+\alpha \X_{s+\epsilon}\right)-DF(\X_{s}),\X_{s+\epsilon}-\X_{s}\rangle \,d\alpha
 \Big)^{2}  ds=\\
&
=A_{1}(\epsilon)+A_{2}(\epsilon)+A_{3}(\epsilon),
\end{split}
\end{displaymath}
where
\begin{displaymath}
\begin{split}
A_{1}(\epsilon)& = \frac{1}{\epsilon}\int_{0}^{t} 
\langle DF(\X_{s}),\X_{s+\epsilon}-\X_{s}   \rangle  ^{2}ds = \\
&=
\int_{0}^{t}\langle DF(\X_{s})\otimes DF(\X_{s}), \frac{(\X_{s+\epsilon}-\X_{s})\otimes^{2}}{\epsilon}\rangle ds\\
A_{2}(\epsilon) 
&
=\frac{2}{\epsilon} \int_{0}^{t} \langle DF(\X_{s}),\X_{s+\epsilon}-\X_{s}\rangle\cdot\\
&\hspace{1.5cm}
\cdot \int_{0}^{1}\langle DF\left((1-\alpha)\X_{s}+\alpha \X_{s+\epsilon}\right)-DF(\X_{s}),\X_{s+\epsilon}-\X_{s} \rangle d\alpha\, ds =\\
& 
= 2\, \int_{0}^{t} \int_{0}^{1}\langle DF(\X_{s})\otimes \left( DF\left((1-\alpha)\X_{s}+\alpha \X_{s+\epsilon}\right)-DF(\X_{s})  \right), \frac{(\X_{s+\epsilon}-\X_{s})\otimes^{2}}{\epsilon} \rangle d\alpha\,ds \\
A_{3}(\epsilon) & =
 \frac{1}{\epsilon} \int_{0}^{t} \Big(  \int_{0}^{1} \langle DF\left(   (1-\alpha)\X_{s}+\alpha \X_{s+\epsilon}   \right)-DF(\X_{s}),\X_{s+\epsilon}-\X_{s} \rangle d\alpha  \Big)^{2} ds
\leq \\
&
\leq
\frac{1}{\epsilon} \int_{0}^{t}\int_{0}^{1} \langle  DF\left((1-\alpha)\X_{s}+\alpha \X_{s+\epsilon}\right)-DF(\X_{s}),\X_{s+\epsilon}-\X_{s} \rangle^{2} d\alpha \, ds=\\
& 
=
\int_{0}^{t} \int_{0}^{1}  \langle \left( DF\left((1-\alpha)\X_{s}+\alpha \X_{s+\epsilon}\right)-DF(\X_{s})  \right)\otimes^{2}  ,
\frac{(\X_{s+\epsilon}-\X_{s})\otimes^{2}}{\epsilon}\rangle d\alpha\, ds		\; .
\end{split}
\end{displaymath}
According to  Remark \ref{rem JUSTCR} and Proposition \ref{pr CONVCCOV} with $X=Y$,  it follows 
$$
A_{1}(\epsilon)\xrightarrow {\mathbb{P}} \int_{0}^{t}\langle DF(\X_{s}) \otimes DF(\X_{s}), d\widetilde{[\X]}_{s}\rangle \; .
$$
It remains to show the convergence in probability of $A_{2}(\epsilon)$ and $A_{3}(\epsilon)$ to zero.\\ 
About $A_{2}(\epsilon)$ 
 the following decomposition holds:
\begin{equation}  \label{scomposizione per termine 2}
DF(\X_{s})\otimes \left( DF\left((1-\alpha)\X_{s}+\alpha \X_{s+\epsilon}\right)-
DF(\X_{s})  \right) = 
DF(\X_{s})\otimes DF\left((1-\alpha)\X_{s}+\alpha \X_{s+\epsilon}\right) - DF(\X_{s})\otimes DF(\X_{s});
\end{equation}
 concerning $A_{3}(\epsilon)$ we get
\begin{equation}  \label{scomposizione per termine 3}
\begin{split}
\left( DF \left( (1-\alpha)\X_{s}+\alpha \X_{s+\epsilon}\right)-DF(\X_{s})  \right)\otimes^{2}  
&
=
DF\left((1-\alpha)\X_{s}+\alpha \X_{s+\epsilon}\right)  \otimes^{2}+\\
&
- DF\left((1-\alpha)\X_{s}+\alpha \X_{s+\epsilon}\right)  \otimes DF(\X_{s})+\\
&
+ 
DF(\X_{s}) \otimes  DF(\X_{s})+\\
&
- DF(\X_{s}) \otimes  DF\left((1-\alpha)\X_{s}+\alpha \X_{s+\epsilon}\right)  \; . \\
\end{split}
\end{equation}
Using \eqref{scomposizione per termine 2}, we obtain
\begin{equation}  \label{eq 8.4}
\begin{split}
\left| A_{2}(\epsilon)  \right| & 
\leq 2 \int_{0}^{t} \int_{0}^{1}
\left| \langle DF(\X_{s})\otimes \left( DF\left((1-\alpha)\X_{s}+\alpha \X_{s+\epsilon}\right)-DF(\X_{s})  \right), \frac{(\X_{s+\epsilon}-\X_{s})
\otimes^{2}}{\epsilon} \rangle \right| d\alpha\,ds \leq \\
&
\leq \int_{0}^{t} \int_{0}^{1}
\left\| DF(\X_{s})\otimes DF\left((1-\alpha)\X_{s}+\alpha \X_{s+\epsilon}\right)- DF(\X_{s})\otimes DF(\X_{s})  \right\|_{\chi}
\left\| \frac{(\X_{s+\epsilon}-\X_{s})\otimes^{2}}{\epsilon} \right\|_{\chi^{\ast}} d\alpha\,ds \; . 
\end{split}
\end{equation}
For fixed $\omega \in \Omega$ we denote 
by $\mathcal{V}(\omega):=\{ \X_{t}(\omega);\, t\in [0,T]\}$ and 
\begin{equation}		\label{def U}
\mathcal{U}=\mathcal{U}(\omega)=\overline{conv(\mathcal{V}(\omega))},
\end{equation} 
i.e. 
the set $\mathcal{U}$ is the closed convex hull of the compact subset $\mathcal{V}(\omega)$ of $B$. 
From \eqref {eq 8.4} we deduce 
\begin{displaymath}
\begin{split}
\left| A_{2}(\epsilon)  \right| 
&
\leq \varpi^{\mathcal{U}\times \mathcal{U}}_{DF\otimes DF}\left( \varpi_{\X}( \epsilon)  \right)\int_{0}^{t} \left\| \frac{(\X_{s+\epsilon}-\X_{s})\otimes^{2}}{\epsilon} \right\|_{\chi^{\ast}} ds,\\
\end{split}
\end{displaymath}
where $\varpi^{\mathcal{U}\times \mathcal{U}}_{DF\otimes DF}$ 
is the continuity modulus of the application $DF(\cdot)\otimes DF(\cdot):B\times B \longrightarrow \chi$ restricted to $\mathcal{U}\times \mathcal{U}$ and $\varpi_{\X}$ is the continuity 
modulus of the continuous process $\X$. 
We recall that 
\[
\varpi^{\mathcal{U}\times \mathcal{U}}_{DF\otimes DF}(\delta)=
\sup_{\|(x_{1},y_{1})-(x_{2},y_{2})\|_{B\times B}\leq\delta  }  
\left\| DF(x_{1})\otimes DF(y_{1})-DF(x_{2})\otimes DF(y_{2})  \right\|_{\chi}
\] where the space $B\times B $ is equipped with the norm obtained summing the norms of the two components.\\
%
According to Theorem 5.35 in \cite{InfDimAn}, $\mathcal{U}(\omega)$ 
is compact, 
so the function $DF(\cdot)\otimes DF(\cdot)$ on $\mathcal{U}(\omega)\times \mathcal{U}(\omega)$ is uniformly continuous and 
$ \varpi^{\mathcal{U}\times \mathcal{U}}_{DF\otimes DF}$ is a positive, increasing function on $\mathbb{R}^{+}$ converging to 
$0$ when the argument converges to zero.\\
%
%
%
%
%
Let $(\epsilon_{n})$ converging to zero; Condition {\bf H1} in the definition of $\chi$-quadratic variation, implies the existence of a subsequence 
$(\epsilon_{n_{k}})$ such that $A_{2}(\epsilon_{n_{k}})$ converges to zero a.s. This implies that $A_{2}(\epsilon)\rightarrow 0$ in probability.\\
With similar arguments, using \eqref{scomposizione per termine 3}, we can show that $A_{3}(\epsilon)\rightarrow 0$ in probability. We observe in fact 
\begin{displaymath}
\begin{split}
\left| A_{3}(\epsilon)  \right| & \leq 
 \int_{0}^{t} \int_{0}^{1}  \left\|  DF\left((1-\alpha)\X_{s}+\alpha \X_{s+\epsilon}\right) \otimes^{2} - DF(\X_{s}) \otimes DF\left((1-\alpha)\X_{s}+\alpha \X_{s+\epsilon}\right)  \right\|_{\chi}\cdot\\
&\hspace{10cm}\cdot
 \left\|  \frac{(\X_{s+\epsilon}-\X_{s})\otimes^{2}}{\epsilon} \right\|_{\chi^{\ast}}  d\alpha\, ds+\\
 &
 +
 \int_{0}^{t} \int_{0}^{1}  \left\|  DF\left((1-\alpha)\X_{s}+\alpha \X_{s+\epsilon}\right) \otimes  DF(\X_{s}) - DF(\X_{s}) \otimes ^{2}  \right\|_{\chi}
 \left\|  \frac{(\X_{s+\epsilon}-\X_{s})\otimes^{2}}{\epsilon} \right\|_{\chi^{\ast}}  d\alpha\, ds \leq\\
&
\leq 2 \varpi^{\mathcal{U}\times \mathcal{U} }_{DF\otimes DF}\left( \varpi_{\X}( \epsilon)  \right)\int_{0}^{t} \left\| \frac{(\X_{s+\epsilon}-\X_{s})\otimes^{2}}{\epsilon} \right\|_{\chi^{\ast}} ds  \;.
\end{split}
\end{displaymath}
The result is now established.
\end{proof}
\begin{cor}		\label{cor B0}  
Let $B$ be a separable Banach space and  $B_{0}$ be 
a Banach space such that  $B_{0}\supset B$
continuously. Let 
 $\chi=(B_{0}\hat{\otimes}_{\pi}B_{0})^{\ast}$ 
and  $\X$ a continuous $B$-valued stochastic process admitting a 
$\chi$-quadratic variation. 
Let $F^{1},F^{2}: B\longrightarrow \mathbb{R}$ be functions of class $C^{1}$ Fr\'echet such that $DF^{i}$, $i=1,2$ are continuous as applications 
from $B$ to $B^{\ast}_{0}$.\\
Then the covariation of $F^{i}(\X)$ and $F^{j}(\X)$ exists and it is given by
\begin{equation} 		\label{eq STCCOVB0}
[F^{i}(\X),F^{j}(\X)]_{\cdot}=\int_{0}^{\cdot} \langle DF^{i}(\X_{s})\otimes DF^{j}(\X_{s}),d \widetilde{[\X]}_{s} \rangle  \; .
\end{equation}
\end{cor}
\begin{proof}  \
It  is clear that  $\chi$ is a Chi-subspace of 
$(B\hat{\otimes}_{\pi}B)^{\ast}$.
For any given $x,y\in B$, $i,j=1,2$, by the 
characterization of $DF^{i}(x)\otimes DF^{j}(y) $ given in Proposition \ref{prop RRCBL} and Remark \ref{rem 7.2}, the following applications 
\[
DF^{i}(x)\otimes DF^{j}(y):B_{0}\hat{\otimes}_{\pi}B_{0} \longrightarrow \mathbb{R}
\]
are continuous for $i,j \in \{1,2\}$. The result follows by
 Theorem \ref{thm SCQVA}.
\end{proof}
\begin{rem}			\label{rem B0H}
Under the same assumptions as Corollary \ref{cor B0} we suppose moreover that $B_{0}$ is a Hilbert space.
For any $x, y\in B$, $DF(x)\otimes DG(y)$ belongs to $\left(B_{0}\hat{\otimes}_{h}B_{0} \right)^{\ast}$ because of Proposition \ref{prop RRCBLbis} 
and it will be associated to a true tensor product in the sense explained in the same proposition.
\end{rem}
We discuss rapidly the finite dimensional framework. A detailed analysis was performed in Paragraph 1, Chapter 6 of \cite{DGR}.
We also recall that 
$\mathbb{R}^{n} \hat{\otimes}_{\pi}\mathbb{R}^{n}$ can be identified with the space of matrices $\mathbb{M}_{n\times n}(\mathbb{R})$. 
Since $\mathbb{R}^{n} \hat{\otimes}_{\pi}\mathbb{R}^{n}$ is finite dimensional and all the topologies are equivalent, it is enough to show the 
identification on $\R^{n}\otimes \R^{n}$.
In fact let $u\in \R^{n}\otimes \R^{n}$ in the form  $u=\sum_{1\leq i,j\leq n} u_{i,j}\; e_{i}\otimes e_{j}$ 
where $(e_{i})_{1\leq i\leq n }$ is the canonical basis for $\mathbb{R}^{n}$. 
To $u$ is possible to associate a unique matrix $U=(u_{i,j})_{1\leq i, j\leq n}$, $U\in \mathbb{M}_{n\times n} (\mathbb{R})$. 
Conversely given a matrix $U\in \mathbb{M}_{n\times n} (\mathbb{R})$ of the form $U=(U_{i,j})_{1\leq i, j\leq n}$, we ssociate
the unique element $u\in  \R^{n}\otimes \R^{n}$ in the form $u=\sum_{1\leq i,j \leq n} U_{i,j}\; e_{i}\otimes e_{j}$.
Concerning the dual space we have 
 $(\R^{n}\otimes \R^{n})^{\ast}\cong L(\mathbb{R}^{n};
L(\mathbb{R}^{n}))$ which is naturally identified with $\mathbb{M}_{n\times n} (\mathbb{R})$. 
So a matrix $T\in\mathbb{M}_{n\times n} (\mathbb{R})$ of the form 
$T=(T_{i,j})_{1\leq i \leq m,1\leq j \leq n}$ is associated with the linear form $t:\R^{n}\otimes \R^{n}\longrightarrow \R$ such that 
$t(x\otimes y)=\prescript{}{\R^{n}}{\langle} T x\; , y\rangle_{ \R^{n}}$.\\ 
Moreover the duality pairing between an element
 $t\in (\mathbb{R}^{n}\hat{\otimes}_{\pi}\mathbb{R}^{n})^{\ast}$ 
 and an element $u\in (\mathbb{R}^{n}\hat{\otimes}_{\pi}\mathbb{R}^{n})$
(or simply $(\mathbb{R}^{n}\otimes \mathbb{R}^{m}$), denoted 
by $\langle t,u\rangle$, coincides with the trace $Tr(TU)$, whenever $U$ (resp. $T$) is the $\mathbb{M}_{n\times n}
 (\mathbb{R})$ matrix
 associated with $u$ (resp. $t$). 
\begin{ese} \label{EFD}
Let $\X=(X^{1}, \cdots, X^{n})$ be a $\mathbb{R}^{n}$-valued stochastic process admitting all its mutual covariations, and 
$F,G: \mathbb{R}^{n}\longrightarrow \mathbb{R}$ $\in C^{1}(\mathbb{R}^{n})$. 
We recall that $\X$ admits a global quadratic variation 
$\widetilde{[\X]}$ 
which coincides with the tensor element associated with the matrix $([\X^{\ast},\X])_{1\leq i,j\leq n}:=([X^{i},X^{j}])_{i,j}$, see Proposition 6.2 in \cite{DGR}. \\
The application of Theorem \ref{thm SCQVA} to this context provides a new proof of Proposition \ref{pr STMLT}.\\ 
According to Proposition 6.2 item 2.(b) in \cite{DGR}, the right-hand side of \eqref{eq SCHICOV} equals
\[
\int_{0}^{\cdot} Tr \left(DF(\X_{s})\otimes DG(\X_{s}) \cdot d[\X^{\ast}, \X]_{s} \right)
\]
which coincides with the right-hand side of \eqref{eq STMLT}.
\end{ese}

\section{Transformation of window Dirichlet processes and window weak Dirichlet processes}  	\label{sec:Dirichlet}  
\subsection{Some preliminary result on measure theory}
We set now $B=C([-\tau,0])$ and we formulate now some related 
Fukushima type decomposition involving $B$-valued window Dirichlet and window weak Dirichlet processes. 
First we need a preliminary result on measure theory.\\ 
We start with some notations appearing for instance in \cite{dincuvisi}, Chapter 1, Section D, Definition 18. 
Let $E$ be a Banach space and $g:[0,T]\longrightarrow E^*$ be a bounded variation function. 
Then the real function $\|g\|_{E^*}:[0,T]\longrightarrow \R$ has also bounded variation. If $f:[0,T]\longrightarrow E$ is a Bochner measurable, then the Bochner integral 
$\int_{0}^T \prescript{}{E}{\langle} f(s),dg(s)\rangle_{E^*}$ is well-defined provided that 
\begin{equation}	\label{EQBI}
\int_{0}^T \| f(s)\|_{E} d\| g\| _{E^*}(s)< +\infty \; .
\end{equation}
We denote by $L^{1}_{E}(g)$ the linear space of functions $f$ verifying \eqref{EQBI}.
\begin{lem}	\label{lemma misura}  
Let $E$ be a topological  direct sum $E_{1}\oplus E_{2}$ where $E_{1}$, $E_{2}$ 
are Banach spaces equipped with norms $\|\cdot\|_{E_{i}}$.
We denote by $P_{i}$ the projectors $P_{i}:E \rightarrow E_{i}$, $i\in 1,2$.\\ 
Let $\tilde{g}:[0,T]\rightarrow E^{\ast}$
and we define     
$\tilde{g}_{i}:[0,T]\rightarrow E_{i}^{\ast}$ setting 
$\tilde{g}_{i}(t)(\eta):=\tilde{g}(t)(\eta)$ for all $\eta\in E_{i}$, i.e. the restriction of $\tilde{g}(t)$ to $E_{i}^{\ast}$. 
We suppose $\tilde{g}_{i}$ continuous with bounded variation, $i=1,2$.\\ 
Let $f:[0,T]\rightarrow E$ measurable with projections $f_{i}:=P_{i}(f)$ defined from $[0,T]$ to $E_{i}$.\\
Then the following statements hold.
\begin{enumerate}
\item $f$ in $L^{1}_{E}(\tilde{g})$ 
if and only if $f_{i}$ in $L^{1}_{E_{i}}(\tilde{g}_{i}), i = 1,2 $  and yields
\begin{equation}		\label{eq E456}
\int_{0}^{t}\prescript{}{E}{\langle}  f(s), d\tilde{g}(s)\rangle_{E^{\ast}}=
\int_{0}^{t}\prescript{}{E_{1}}{\langle}  f_{1}(s), d\tilde{g}_{1}(s)\rangle_{E_{1}^{\ast}}+
\int_{0}^{t}\prescript{}{E_{2}}{\langle}  f_{2}(s), d\tilde{g}_{2}(s)\rangle_{E_{2}^{\ast}}			\; .
\end{equation}
\item
If $\tilde{g}_{2}(t)\equiv 0$ and 
 $f_{1}$ in $L^{1}_{E_{i}}(\tilde{g}_{1})$
  then
\begin{equation} \label{eq: 20}
\int_{0}^{t}\prescript{}{E}{\langle}  f(s), d\tilde{g}(s)\rangle_{E^{\ast}}=
\int_{0}^{t}\prescript{}{E_{1}}{\langle}  f_{1}(s),
 d\tilde{g}_{1}(s)\rangle_{E_{1}^{\ast}}		\; .
\end{equation}
\end{enumerate}
\end{lem}
\begin{proof}\
\begin{enumerate}
\item 
By the hypothesis on $\tilde{g}_{i}$ we deduce that $\tilde{g}:[0,T]\rightarrow E^{\ast}$ has bounded variation. 
If $f:[0,T]\rightarrow E$ belongs to $L^{1}_{E}$, then $f_{i}=P_{i}(f):[0,T]\rightarrow E_{i}$, 
$i=1,2$ belong to $L^{1}_{E_{i}}$ by the property $\|P_{i}f\|_{E_{i}}\leq \|f\|_{E}$.\\
We prove \eqref{eq E456} for a step function $f:[0,T]\rightarrow E$ defined by 
$f(s)=\sum_{j=1}^{N}\phi_{A_{j}}(s) f_{j}$ with $\phi_{A_{j}}$ indicator functions of the subsets 
$A_{j}$ of $[0,T]$ and $f_{j}\in E$. We have $f_{j}=f_{1j}+f_{2j}$ with $f_{ij}=P_{i}f_{j}$, $i=1,2$, so
\begin{displaymath}
\begin{split}
\int_{0}^{T} \prescript{}{E}{\langle}  f(s), d\tilde{g}(s) \rangle_{E^{\ast}}  &
=
\sum_{j=1}^{N} \int_{A_{j}}   \prescript{}{E}{\langle} f_{j},  d\tilde{g}(s)\rangle_{E^{\ast}} =
\sum_{j=1}^{N} \prescript{}{E}{\langle} f_{j}, \int_{A_{j}} d\tilde{g}( s) \rangle_{E^{\ast}}
=
\sum_{j=1}^{N} \prescript{}{E}{\langle} f_{j}, d\tilde{g}(A_{j})\rangle_{E^{\ast}}=\\
&
=
\sum_{j=1}^{N}\prescript{}{E_{1}}{\langle} f_{1j}, d\tilde{g}_{1}(A_{j})\rangle_{E_{1}^{\ast}} +\sum_{j=1}^{N}\prescript{}{E_{2}}{\langle} f_{2j}, d\tilde{g}_{2}(A_{j})\rangle_{E_{2}^{\ast}} =\\
&=
 \int_{0}^{T} \prescript{}{E_{1}}{\langle} f_{1}(s) , d\tilde{g}_{1}(s)\rangle_{E_{1}^{\ast}} 
 +\int_{0}^{T} \prescript{}{E_{2}}{\langle} f_{2}(s) , d\tilde{g}_{2}(s)\rangle_{E_{2}^{\ast}} \; .
\end{split} 
\end{displaymath}
A general function $f$ in $L^{1}_{E}(\tilde{g})$ is a sum of 
$f_{1}+f_{2}$, $f_{i}\in L^{1}_{E_{i}}(\tilde{g}_{i})$ for $i=1,2 $. Both $f_{1}$ and $f_{2}$ can be approximated 
by step functions. Vector integration $L^{1}_{E}(\tilde{g})$, as well as on $L^{1}_{E_{i}}(\tilde{g}_{i})$, is defined by density on step functions. 
The result follows by an approximation argument.
\item It follows directly by 1.
\end{enumerate}
\end{proof}

A useful consequence  of Lemma \ref{lemma misura} is the following.
\begin{prop} \label{lemma misura1}
Let $E_{1}=\mathcal{D}_{i,j}([-\tau,0]^{2})$ and $E_{2}$ be a 
Banach subspace of $\mathcal{M}([-\tau,0]^{2})$
such that $E_1 \cap E_2 = \{0\}$.
\begin{itemize}
\item Let $\tilde g:[0,T] \rightarrow E^\ast$ 
 such that 
$ \tilde g(t)_{\vert E_2} \equiv 0$.
\item We set 
 $g_{1}:[0,T]\rightarrow \R$ by $g_{1}(t)=
 \prescript{}{E_{1}}{\langle} \delta_{(a_{i},a_{j})},
\tilde{g}_{1}(t)\rangle_{E^{\ast}_{1}}$,
supposed continuous with bounded variation.
\item Let  $f: [0,T] \rightarrow  E$ such that 
$t \rightarrow f(t)(\{(a_i,a_j\}) \in L^1(d\vert g_1\vert).$
\end{itemize}
Then
\begin{equation}		\label{eq 5.15}
\int_{0}^{t}\prescript{}{E}{\langle}  f(s), d\tilde{g}(s)
\rangle_{E^{\ast}}=
\int_{0}^{t} f(s)(\{ a_{i}, a_{j}\}) dg_{1}(s)  \; .
\end{equation}
\end{prop}

\begin{rem} \label{lemma misura R} 
Let $g_1$ be the  real function defined in the second item
of the hypotheses.\\
Defining $\tilde g_1 : [0,T] \rightarrow E_1^\ast$
by $\tilde g_1(t) = g_1(t) \; \delta_{(a_{i}, a_{j})}$,
 by construction it follows $\tilde g_1(t)(f) =  \tilde g(t)(f)$
for every $f \in E_1, t \in [0,T]$.
Since for $a,b\in [0,T]$, with $a< b$, we have 
\[  \left\| \tilde{g}(b)-\tilde{g}(a)\right\|_{E^{\ast}} =
 \left\| \tilde{g}_{1}(b)-\tilde{g}_{1}(a)\right\|_{E_{1}^{\ast}}
= \left|  g_{1}(b)-g_{1}(a)\right|  \ ;
\]
then $g_1$ is continuous with bounded  variation
if and only if $\tilde g$ is continuous with bounded variation. 
\end{rem}
\begin{proof} [Proof of Proposition \ref{lemma misura1}]
We apply Lemma  \ref{lemma misura}.2. 
Clearly we have $P_1 (f) = f(\{a_i,a_j\})\delta_{(a_{i}, a_{j})} $.
It follows that
$$
\int_{0}^{t}\prescript{}{E}{\langle}  f(s), d\tilde{g}(s)\rangle_{E^{\ast}}=
\int_{0}^{t}  \prescript{}{E_1}
 {\langle} f(s)(\{ a_{i}, a_{j}\}) \delta_{(a_{i}, a_{j})},
 d\tilde g_{1}(s)\rangle_{E_1^\ast}. 
$$
Since  $g_{1}(t)= \prescript{}{E_{1}}{\langle}\delta_{(a_{i},a_{j})},
 \tilde{g}_{1}(t)  \rangle_{E_{1}^{\ast}}$
and because of Theorem 30 in Chapter 1, paragraph 2 
of \cite{dincuvisi},
previous expression equals the right-hand side  of 
 \eqref{eq 5.15}.
\end{proof}

\begin{rem}  
Let $E$ be a Banach subspace of $\mathcal{M}([-\tau,0]^{2})$ containing $\mathcal{D}_{i,j}([-\tau,0]^{2})$. 
A typical example of application of Proposition \ref{lemma misura1} is given by $E_{1}=\mathcal{D}_{i,j}([-\tau,0]^{2})$ and 
$E_{2}=\left\{ \mu\in E \; |\;  \mu(\{a_{i},a_{j}\})=0 \right\}$. Any $\mu\in E$ can be decomposed into $\mu_{1}+\mu_{2}$, 
where $\mu_{1}=\mu(\{a_{i},a_{j}\})\delta_{(a_{i},a_{j})}$, which belongs to $E_{1}$, and $\mu_{2}\in E_{2}$. 
\end{rem}
In the proof of item 3. in proposition below we will use Proposition \ref{lemma misura1} 
considering $\tilde{g}$ as the $\chi$-covariation of two processes $X(\cdot)$ and $Y(\cdot)$.
\begin{prop}			\label{pr PLM}
Let $i,j\in \{0,\ldots,N\}$ and let $\chi_{2}$ be a Banach subspace of $\mathcal{M}([-\tau,0]^{2})$ such that $\mu(\{a_{i},a_{j}\})=0$ 
for every $\mu\in \chi_{2}$. We set $\chi=\mathcal{D}_{i,j}([-\tau,0]^{2})\oplus \chi_{2}$.\\ 
Let $X$, $Y$ be two real continuous processes such that $(X_{\cdot+a_{i}},Y_{\cdot+a_{j}})$ admits their mutual covariations and 
such that $X(\cdot)$ and $Y(\cdot)$ admit a zero $\chi_{2}$-covariation. 
Then following properties hold.
\begin{enumerate}
\item $\chi$ is a Chi-subspace of $(B\hat{\otimes}_{\pi}B)^{\ast}$,
 with $B=C([-\tau,0])$.
\item $X(\cdot)$ and $Y(\cdot)$ admit a $\chi$-covariation of the type 
\[
[X(\cdot),Y(\cdot)]:\chi \longrightarrow \mathscr{C}([0,T]) \ ,
\hspace{2cm}
[X(\cdot),Y(\cdot)](\mu)=\mu(\{a_{i},a_{j}\})[X_{\cdot+a_{i}},Y_{\cdot+a_{j}}]			\; .
\]
\item For every $\chi$-valued process $\mathbb{Z}$ with locally bounded paths (for instance cadlag) we have
\begin{equation}		\label{eq prop 4}
\int_{0}^{\cdot} \langle \mathbb{Z}_{s}, d\widetilde{[X(\cdot),Y(\cdot)]}_{s}\rangle =
\int_{0}^{\cdot} \mathbb{Z}_{s}(\{a_{i},a_{j}\})d[X_{\cdot+a_{i}},Y_{\cdot+a_{j}}]_{s}		\; .
\end{equation} 
\end{enumerate}
\end{prop}
\begin{proof}\
\begin{enumerate}
\item  By Proposition 3.4 in \cite{DGR1}, $\chi$ is a closed subspace 
of $\mathcal{M}([-\tau,0]^{2})$.
 The claim follows by Proposition 3.3 in \cite{DGR1}.
\item We denote here $\chi_1 = \mathcal{D}_{i,j}([-\tau,0]^{2})$; 
$\chi_1$ and $\chi_{2}$ are closed subspaces of $\mathcal{M}([-\tau,0]^{2})$. 
By Proposition \ref{pr QV123}, item 3) 
$X(\cdot)$ and $Y(\cdot)$ admit a 
$\chi_1$-covariation. Proposition 3.18 in \cite{DGR1} 
 implies that 
$X(\cdot)$ and $Y(\cdot)$ admit a $\chi$-covariation 
which can be determined from the 
$\chi_1$-covariation and the 
$\chi_{2}$-covariation. 
More precisely, for $\mu$ in $\chi$ with decomposition $\mu_{1}+\mu_{2}$,
 $\mu_{1}\in \chi_1$
 and $\mu_{2}\in \chi_{2}$, with a slight  
abuse of  notations, we have
\begin{eqnarray*}
[X(\cdot),Y(\cdot)](\mu) & =&
[X(\cdot),Y(\cdot)](\mu_{1})+[X(\cdot),Y(\cdot)](\mu_{2})=
[X(\cdot),Y(\cdot)](\mu_{1})\\
&=&
\mu_{1}(\{a_{i},a_{j}\}) [X_{\cdot+a_{i}},Y_{\cdot+a_{j}}]
=
\mu(\{a_{i},a_{j}\}) [X_{\cdot+a_{i}},Y_{\cdot+a_{j}}]  \; .
\end{eqnarray*} 
\item Since both sides of \eqref{eq prop 4} are continuous processes, it is enough to show that they are equal a.s. for every fixed $t\in [0,T]$. 
This follows for almost all $\omega\in \Omega$ using Proposition 
\ref{lemma misura1}
 where $f=\mathbb{Z}(\omega)$ and $\tilde{g}=\widetilde{[X(\cdot),Y(\cdot)]}(\omega)$.
We remark that here $\tilde{g}_{1}=\widetilde{[X_{\cdot+a_{i}}(\cdot),Y_{\cdot+a_{j}}(\cdot)]}(\omega)$ and $g_1 = [X_{\cdot+a_{i}},Y_{\cdot+a_{j}}](\omega)$.
\end{enumerate}
\end{proof}
\begin{rem} \label{RPLM1}
Proposition \ref{pr PLM} will be used in the sequel especially 
in the case $a_{i}=a_{j}=0$. 
\end{rem}
\begin{rem}\label{RPLM}
Under the same assumptions as Proposition \ref{pr PLM}, 
if $\Z$ takes values in $\shd_{i,j}$, then 
\begin{equation} \label{ERRR}
\int_{0}^{\cdot} \langle \mathbb{Z}_{s},
 d\widetilde{[X(\cdot),Y(\cdot)]}_{s}\rangle =
\int_{0}^{\cdot} \mathbb{Z}_{s}(\{a_{i},a_{j}\})
d[X_{\cdot+a_{i}},Y_{\cdot+a_{j}}]_{s}		\; .
\end{equation}
In fact, the left-hand side equals
$$\int_{0}^{\cdot}  \mathbb{Z}_{s}({a_i,a_j}) \langle \delta_{a_i,a_j},
 d\widetilde{[X(\cdot),Y(\cdot)]}_{s}\rangle.$$
That expression equals the right-hand side of \eqref{ERRR} because
of item 3) in Proposition \ref{pr QV123} and Theorem 30, Chapter 1,
par 2 of \cite{dincuvisi}.
\end{rem}
\subsection{On some generalized Fukushima decomposition}
We are ready now to show some decomposition results.
\begin{thm}  		\label{thm STWDIRP}		
Let $X$ be a real continuous $(\shf_{t})$-Dirichlet process with decomposition 
$X=M+A$, where $M$ is the $(\shf_{t})$-local martingale and $A$ is a zero quadratic variation process with $A_{0}=0$. 
Let $F:C([-\tau,0])\longrightarrow \mathbb{R}$ be a Fr\'echet differentiable function such that the range of $DF$ is $\mathcal{D}_{0}([-\tau,0])\oplus L^{2}([-\tau,0])$. 
Moreover we suppose that $DF:C([-\tau,0])\longrightarrow \mathcal{D}_{0}([-\tau,0])\oplus L^{2}([-\tau,0])$ is continuous.\\
Then $F(X(\cdot))$ is an $(\shf_{t})$-Dirichlet process with local martingale component equal to
\begin{equation} 	\label{FL18}
\bar{M}_{\cdot}=F\big(X_{0}(\cdot)\big) + \int_{0}^{\cdot} D^{\delta_{0}}F \big(X_{s}(\cdot)\big)dM_{s}  \; ,
\end{equation}
where from Notation \ref{nota MEASURE} we recall that $D^{\delta_{0}}F (\eta)=DF(\eta)(\{0\})$.
\end{thm}
\begin{rem}
The It\^o integral in \eqref{FL18} makes sense because $\apt D^{\delta_0} F\apt X_t(\cdot) \cpt \cpt$ is 
$(\F_t)$-adapted.
\end{rem}
\begin{proof} \
We need to show that $[\bar{A}]=0$ where $\bar{A}:=F(X(\cdot))-\bar{M}$. 
For simplicity of notations,  in this proof we will 
denote $\alpha_{0}(\eta)=D^{\delta_{0}}F(\eta)$.
By the linearity of the covariation of real processes, we have
$[\bar{A}]=A_{1}+A_{2}-2 A_{3}$
where 
\begin{displaymath}
\begin{split}
A_{1}
&
=\left[  F(X_{\cdot}(\cdot)) \right] \\
A_{2}
&
=\left[  \int_{0}^{\cdot}\alpha_{0} \big(X_{s}(\cdot)\big)dM_{s}\right]\\
A_{3}
&
=
\left[   F(X(\cdot)) , \int_{0}^{\cdot}\alpha_{0} \big(X_{s}(\cdot)\big)dM_{s} \right]  \; .
\end{split}
\end{displaymath}
Since $X$ is a finite quadratic variation process, by Proposition 
\ref{pr QV123} 4), its window process $X(\cdot)$ 
admits a $\chi^{0}([-\tau,0]^{2})$-quadratic variation.
Moreover by Example \ref{pr STCHI2} and Remark \ref{rem B0H} the map $DF\otimes DF:C([-\tau,0])\times C([-\tau,0])\longrightarrow \chi^{0}([-\tau,0]^{2})$ 
is a continuous application. 
Applying Theorem \ref{thm SCQVA} and \eqref{eq prop 4} of Proposition \ref{pr PLM} we obtain
\begin{displaymath}
\begin{split}
A_{1} &
= 
\int_{0}^{\cdot}\langle DF(X_{s}(\cdot))\otimes DF(X_{s}(\cdot)), d \widetilde{[X_{\cdot}(\cdot)]}_{s}\rangle=\\
&
=
\int_{0}^{\cdot} \alpha^{2}_{0}(X_{s}(\cdot) )d[X]_{s}=
\int_{0}^{\cdot} \alpha^{2}_{0}(X_{s}(\cdot) )d[M]_{s}		\; .
\end{split}
\end{displaymath}
The term $A_{2}$ is the quadratic variation of a local martingale; by Remark \ref{rem R} item 2. we get
\[
A_{2}=
\int_{0}^{\cdot}\alpha^{2}_{0} (X_{s}(\cdot))d[M]_{s}   \; .
\]
It remains to prove that $A_{3}=
\int_{0}^{\cdot}\alpha^{2}_{0} (X_{s}(\cdot))d[M]_{s} $. 
We define $G:C([-\tau,0])\longrightarrow \mathbb{R}$ by $G(\eta)=\eta(0)$. We observe that $\bar{M}=G(\bar{M}(\cdot))$ where 
$\bar{M}(\cdot)$ denotes as usual the window process associated to $\bar{M}$.
$G$ is Fr\'echet differentiable and $DG(\eta)=\delta_{0}$, therefore $DG$ is continuous from $C([-\tau,0])$ to $\mathcal{D}_{0}([-\tau,0])\oplus L^{2}([-\tau,0])$. 
Moreover by Example \ref{pr STCHI2} we know that $DF\otimes
DG:C([-\tau,0])\times C([-\tau,0])\longrightarrow
\chi^{0}([-\tau,0]^{2})$
 is continuous. Corollary \ref{corCDWMbis} item 2. says that 
the $\chi^{0}([-\tau,0]^{2})$-covariation between $X(\cdot)$ 
and $\bar{M}(\cdot)$ exists and it is given by 
\begin{equation}		\label{eq YUI}
[X(\cdot),\bar{M}(\cdot)](\mu)=\mu(\{0,0\})[X,\bar{M}]  \; .
\end{equation} 
We have $[X,\bar{M}]=[M,\bar{M}]+[A,\bar{M}]=[M,\bar{M}]$. 
By Remark \ref{rem R} item 2. and 
the usual properties of stochastic calculus we have 
\be \label{415b}
[X,\bar{M}]=
\left[ M,\int_{0}^{\cdot}\alpha_{0} \big(X_{s}(\cdot)\big)dM_{s} \right]=
\int_{0}^{\cdot}\alpha_{0} \big(X_{s}(\cdot)\big) d[M]_{s} \ . 
\ee
Finally,  applying again Theorem \ref{thm SCQVA}, relation \eqref{eq prop 4} in 
Proposition \ref{pr PLM} and \eqref{415b} we obtain
\begin{eqnarray*}
A_{3} & =&[F(X(\cdot)), G(\tilde{M}(\cdot))]=
\int_{0}^{\cdot}\langle DF(X_{s}(\cdot))\otimes
DG(\bar{M}_{s}(\cdot)),
 d \widetilde{[X(\cdot),\bar{M}(\cdot)]}_{s}\rangle
\\
&=&
\int_{0}^{\cdot} \alpha_{0}(X_{s}(\cdot) )d[X, \bar{M}]_{s}=
\int_{0}^{\cdot} \alpha^{2}_{0}(X_{s}(\cdot) )d[M]_{s}  \; .
\end{eqnarray*}
The result is now established.
\end{proof}
Theorem \ref{thm STWDIRP} admits a slight generalization, in which
 will intervene  the space $\mathcal{D}_{a}$ as defined at equation
 \eqref{eq-def Da} but the final process is no longer a Dirichlet process, only a weak Dirichlet.
\begin{thm}		\label{thm STWDPR}
Let $X$ be a real continuous $(\shf_{t})$-Dirichlet process with decomposition $X=M+A$, $M$ being a local martingale and $A$ a zero quadratic variation process with $A_{0}=0$. 
Let $F:C([-\tau,0])\longrightarrow \mathbb{R}$ be a Fr\'echet
differentiable
 function such that $DF:C([-\tau,0]) \longrightarrow 
\mathcal{D}_{a}([-\tau,0])\oplus L^{2}([-\tau,0])$ is 
continuous. We have the following.
\begin{enumerate}
\item $F(X(\cdot))$ is an $(\shf_{t})$-weak Dirichlet process with decomposition 
$F(X(\cdot))=\bar{M}+\bar{A}$, where $\bar{M}$ is the local martingale defined by 
$$ \bar{M}_{\cdot}:=F\left(X_{0}(\cdot)\right)+\int_{0}^{\cdot}D^{\delta_{0}}F(X_{s}(\cdot))dM_{s}$$ 
and $\bar{A}$ is the $(\shf_{t})$-martingale orthogonal process, 
corresponding to  Definition \ref{dfn FtWD}.
\item $F(X(\cdot))$ is a finite quadratic variation process and
\begin{equation}		\label{eq: Estar}
\left[F(X(\cdot)) \right]=\sum_{i=0,\ldots,N}\int_{0}^{t}
 \apt  D^{\delta_{a_{i} } }F( X_{s}(\cdot))\cpt^{2}
d[M]_{s+a_{i}}
\end{equation} 
\item Process $\bar{A}$ is a finite quadratic variation process and
\begin{equation} 		\label{eq: var quad A}
[\bar{A}]_{t}=\sum_{i=1,\ldots,N}\int_{0}^{t}\apt D^{\delta_{a_{i}}}F(X_{s}(\cdot))\cpt^{2}
d[M]_{s+a_{i}}
\end{equation}
\item In particular
$\{ F(X_{t}(\cdot)); t\in [0,-a_{1}]\}$ is a Dirichlet process with local martingale component $\bar{M}$.
\end{enumerate}
\end{thm}
\begin{proof} \  In this proof $\alpha_{i}(\eta)$ will denote 
$D^{\delta_{a_{i}}}F(\eta)=DF(\eta)(\{a_{i}\})$ if $\eta \in C([-\tau,0])$. 
\begin{enumerate}
\item 
To show that $F(X(\cdot))$ is an $(\shf_{t})$-weak Dirichlet process we need to show that 
$[F(X(\cdot))-\int_{0}^{\cdot}\alpha_{0} \big(X_{s}(\cdot)\big)dM_{s}, N]$ is zero for every $(\mathcal{F}_{t})$-continuous local martingale $N$. 
Again we set $G:C([-\tau,0])\longrightarrow \mathbb{R}$ by $G(\eta)=\eta(0)$. It holds $N_{t}=G(N_{t}(\cdot))$. 
We remark that function $G$ is Fr\'echet differentiable with
$DG:C([-\tau,0])\longrightarrow \mathcal{D}_{0}([-\tau,0])$ continuous
and 
$DG(\eta) \equiv \delta_{0}$.\\
In view of applying Corollary   \ref{corCDWMbis}, we set
 $\chi:=\shd_{0,0}\oplus \chi_2$ where $\chi_2=
\oplus_{i=1}^{N}\shd_{i, 0}
\oplus \apt L^{2}([-\tau,0])\hat{\otimes}_h \shd_{0}\cpt$. 
In particular for every $\mu \in \chi_2$ we have
$\mu(\{0,0\}) = 0$. 
 $X(\cdot)$ and $N(\cdot)$ admit a  $\chi$-covariation
 by Corollary \ref{corCDWMbis} 3.
On the other hand
$DF\otimes DG:C([-\tau,0])\times C([-\tau,0]) \longrightarrow \chi$
and it is a continuous map. 
By Theorem \ref{thm SCQVA}   we have
\begin{equation}	\label{eq G2bis}
\left[F(X(\cdot)),N\right]_{t}=\left[ F(X(\cdot)),
  G(N(\cdot))\right]_{t}
=\int_{0}^{t}\langle DF(X_{s}(\cdot))\otimes \delta_0, 
d\widetilde{[X(\cdot),N(\cdot)]}_{s}\rangle \ .
\ee
By \eqref{eq prop 4} in Proposition \ref{pr PLM} 
it follows that
\begin{eqnarray}	\label{eq G2}
\left[F(X(\cdot)),N\right]_{t} &=&
\int_{0}^{t} (D^{\delta_{0}}F(s,X_{s}(\cdot))\otimes \delta_0)
(\{0,0\})  d[X,N]_{s} \nonumber
\\
&&\\
 &=& 
\int_{0}^{t}D^{\delta_{0}}F(s,X_{s}(\cdot))d[M,N]_{s}  \ . \nonumber
\end{eqnarray}
By Remark \ref{rem R} item 2. and 
usual properties of stochastic calculus, it yields
\[
\left[ \int_{0}^{\cdot}\alpha_{0} \big(X_{s}(\cdot)\big)dM_{s}, N \right]_{t}=\int_{0}^{t}\alpha_{0} \big(X_{s}(\cdot)\big)d[M,N]_{s}
\]
and the result follows.\\
\item By Example \ref{pr STCHI2} we know that 
$DF\otimes DF:C([-\tau,0])\times C([-\tau,0])\longrightarrow 
\chi^{2}([-\tau,0]^{2})$ and 
it is a linear continuous map. 
We decompose
$ DF(\eta) = \sum_{i=1}^N \alpha_i(\eta) \delta_{a_i} + g(\eta)$
where $g: C([-\tau,0] \rightarrow L^2([-\tau,0])$ 
so that 
\begin{equation*}
\begin{split}
 DF(\eta) \otimes DF(\eta) &= \sum_{i,j=0}^N \alpha_i(\eta)
\alpha_j(\eta) \delta_{a_i}\otimes \delta_{ a_j}
+   \sum_{i=0}^N \alpha_i(\eta)   \delta_{a_i} \otimes g(\eta) 
+  \sum_{j=0}^N \alpha_j(\eta)  g(\eta) \otimes  \delta_{a_j}\\
&+ g(\eta) \otimes g(\eta). 
\end{split}
\end{equation*}
Applying Theorem \ref{thm SCQVA}, relation \eqref{eq prop 4} in 
Proposition \ref{pr PLM} and obvious bilinearity arguments,
 we obtain 
\be \label{EPLM}
\begin{split}
[F(X(\cdot))]_{t} & =\int_{0}^{t}\langle DF(X_{s}(\cdot))
\otimes DF(X_{s}(\cdot)), d\widetilde{[X_{s}(\cdot)]}\rangle \\
&
=
\sum_{i,j=0}^N \int_0^t \langle  \Z^{i,j}_s , d\widetilde{[X_{s}(\cdot)]}\rangle 
+ \int_0^t \langle \Z_s , d\widetilde{[X_{s}(\cdot)]}\rangle 
\end{split}
\ee
where 
\be  
\begin{split}
  \Z_s^{i,j} &=\alpha_i(X_s(\cdot))  \alpha_j(X_s(\cdot)) \ \delta_{a_i}\otimes \delta_{a_j} \\
\Z_s &= DF(X_s(\cdot)) \otimes DF(X_s(\cdot)) - \sum_{i,j=0}^N \Z_s^{i,j}.
\end{split}
\ee
A wise application of Proposition \ref{pr PLM} and Remark \ref{RPLM}
show that \eqref{EPLM} equals
$$ 
 \sum_{i,j=0\ldots,N} \int_0^t  \alpha_{i}(X_{s}(\cdot))
\alpha_{j}(X_{s}(\cdot))d[X_{\cdot+a_{i}},X_{\cdot+a_{j}}]_{s} =
\sum_{i=0,\ldots,N}\int_{0}^{t}\alpha^{2}_{i}(X_{s}(\cdot))
d[M]_{s+a_{i}}. $$
The last equality is a consequence of
Proposition \ref{prop STRP} and of the definition of 
weak Dirichlet process.
Finally \eqref{eq: Estar} is proved.
\item By bilinearity of the covariation of real processes we have 
$[\bar{A}]=[F(X(\cdot))]+[\bar{M}]-2 [F(X(\cdot)),\bar{M}]$.
 The first bracket is equal to \eqref{eq: Estar}
and the second term gives
\begin{displaymath}
\left[ \int_{0}^{\cdot}\alpha_{0} \big(X_{s}(\cdot)\big)dM_{s} \right] =\int_{0}^{t} \alpha_{0}^{2}(X_{s}(\cdot))d[M]_{s}  \; .
\end{displaymath}
Setting $N_{t}=\int_{0}^{t}\alpha_{0}(X_{s}(\cdot))dM_{s}$, \eqref{eq G2} gives
\begin{displaymath}
\left[ F(X(\cdot) ), \int_{0}^{\cdot}  \alpha_{0} \big(X_{s}(\cdot)\big)dM_{s} \right]
=\int_{0}^{t}\alpha^{2}_{0}(X_{s}(\cdot))d[M]_{s}
\end{displaymath}
and \eqref{eq: var quad A} follows.
\item It is an easy consequence of \eqref{eq: var quad A} since $(\bar{A}_{t})_{t\in [0,-a_{1}[}$ is a zero quadratic variation process.
\end{enumerate}
\end{proof}
\begin{rem}
\begin{enumerate}
\item
Theorem \ref{thm STWDPR} gives a class of examples of $(\shf_{t})$-weak Dirichlet processes with finite quadratic variation which are not necessarily 
$(\shf_{t})$-Dirichlet processes.
\item 
An example of $F:C([-\tau,0])\longrightarrow \mathbb{R}$ Fr\'echet differentiable such that $DF:C([-\tau,0]) \longrightarrow \mathcal{D}_{a}([-\tau,0])\oplus L^{2}([-\tau,0])$ 
continuously 
is, for instance, $F(\eta)=f\apt \eta(a_0),\ldots, \eta(a_N) \cpt$, with $f \in C^{1}(\R^N)$. We have $DF(\eta)=\sum_{i=0}^{N}\partial_i f \apt \eta(a_0),\ldots, \eta(a_N) \cpt \delta_{a_{i}}$.
\item Let $a\in [-\tau,0[$ and $W$ be a classical $(\shf_{t})$-Brownian motion, process $X$ defined as $X_{t}:=W_{t+a}$ is an $(\shf_{t})$-weak Dirichlet process that is not 
$(\shf_{t})$-Dirichlet.\\ 
This follows from Theorem \ref{thm STWDPR}, point 2. and 3. taking $F(\eta)=\eta(a)$. 
In particular point 3. implies that the quadratic variation of the martingale orthogonal process is $[\bar{A}]_{t}=(t+a)^+$.
This result was also proved directly in Proposition 4.11 in \cite{crnsm2}.
\end{enumerate}
\end{rem}
We now go on with a $C^1$ transformation of window of weak Dirichlet processes. 
\begin{thm}  			\label{thm: thm 1}
Let $X$ be an $(\shf_{t})$-weak Dirichlet process with finite quadratic variation where $M$ is the local 
martingale part. Let $F:[0,T]\times C([-\tau,0])\longrightarrow \R$ continuous. 
We suppose moreover that $(t,\eta)\mapsto D F(t,\eta)$ exists with values in $\shd_{a}([-\tau,0])\oplus L^{2}([-\tau,0])$ and 
$DF:[0,T]\times C([-\tau,0])\longrightarrow \shd_{a}([-\tau,0])\oplus L^{2}([-\tau,0])$ is continuous.\\
Then $\apt F\apt t,X_{t}(\cdot) \cpt \cpt_{t\in [0,T]}$ 
is an $(\shf_{t})$-weak Dirichlet process with martingale part 
\begin{equation}		\label{eq: mart part uC01}
\bar{M}^{F}_{t}:=F(0,X_{0}(\cdot))+\int_{0}^{t} D^{\delta_{0}} F(s, X_{s}(\cdot)) dM_{s}  \; .
\end{equation}
\end{thm}
\begin{proof} \
In this proof we will denote real processes $\bar{M}^{F}$ simply by $\bar{M}$ and $\chi$ will denote the following Chi-subspace
$
\chi:=\apt\mathcal{D}_{a}([-\tau,0]) \oplus L^{2}([-\tau,0]) 
\cpt  \hat{\otimes}_{h}\shd_{0}([-\tau,0])
$.
We need to show that for any $(\shf_{t})$-continuous local martingale $N$
\begin{equation}		\label{eq: da dimostrare}
\left[ F(\cdot, X(\cdot))-  \bar{M} , N \right] \equiv 0.
\end{equation}
Since the covariation of semimartingales coincides with the 
classical covariation, see Remark \ref{rem R} item 2., it follows
\begin{equation}		\label{eq: covariation}
\left[  \bar{M}, N \right]_{t}=\int_{0}^{t}  D^{\delta_{0}} F(s, X_{s}(\cdot)) d[M,N]_{s}  \; .
\end{equation}
It remains to check that, for every $t\in [0,T]$,
\[
\left[  F(\cdot, X(\cdot)), N \right]_{t}=\int_{0}^{t}  D^{\delta_{0}} F(s, X_{s}(\cdot)) d[M,N]_{s}  \; .
\]
For this, for fixed $t \in [0,T]$, we will evaluate the  limit in
 probability of 
\begin{equation}		\label{eq ZERO}
\int_{0}^{t} \Big(   F(s+\epsilon, X_{s+\epsilon}(\cdot))- F(s, X_{s}(\cdot)) \Big) \frac{N_{s+\epsilon}-N_{s}}{\epsilon} ds
\end{equation}
if it exists. \eqref{eq ZERO} can be written as the sum of the two terms
\[
\begin{split}
I_{1}(t,\epsilon)& =\int_{0}^{t} \Big(   F(s+\epsilon, X_{s+\epsilon}(\cdot))- F(s+\epsilon, X_{s}(\cdot)) \Big) \frac{N_{s+\epsilon}-N_{s}}{\epsilon} ds \, ,\\
I_{2}(t,\epsilon)&= \int_{0}^{t} \Big(   F(s+\epsilon, X_{s}(\cdot))- F(s, X_{s}(\cdot)) \Big) \frac{N_{s+\epsilon}-N_{s}}{\epsilon} ds  \; .
\end{split}
\]
First we prove that $I_{1}(t,\epsilon)$ converges to $\int_{0}^{t}  D^{\delta_{0}} F(s, X_{s}(\cdot)) d[M,N]_{s}$.\\ 
If $G:C([-\tau,0])\rightarrow \R$ is the function $G(\eta)=\eta(0)$, then 
$G$ is of class $C^{1}$ and $DG(\eta)=\delta_{0}$ for all $\eta\in C([-\tau,0])$ 
so that $DG:C([-\tau,0])\longrightarrow \shd_{0}([-\tau,0])$ is continuous. 
In particular it holds the equality $\eta(0)=G(\eta(\cdot))=\langle \delta_{0},\eta\rangle$. 
We express
\begin{equation}		\label{eq: I1}
\begin{split}
I_{1}(t,\epsilon) 
&
=
\int_{0}^{t} \langle DF(s+\epsilon, X_{s}(\cdot)), (X_{s+\epsilon}(\cdot)-X_{s}(\cdot))\rangle \frac{N_{s+\epsilon}-N_{s}}{\epsilon} ds +R_{1}(t,\epsilon)
\\
&
=\int_{0}^{t} \langle DF(s+\epsilon, X_{s}(\cdot)), (X_{s+\epsilon}(\cdot)-X_{s}(\cdot))\rangle \frac{\langle\delta_{0}, N_{s+\epsilon}(\cdot) - N_{s}(\cdot)\rangle }{\epsilon} ds +R_{1}(t,\epsilon),
\end{split}
\end{equation}
and
\be		\label{eq 4.25bis}
\begin{split}
R_{1}(t,\epsilon) 
& = \int_{0}^{t} \left[ \int_{0}^{1}   \langle DF\big(s+\epsilon,
 (1-\alpha)X_{s}(\cdot) +\alpha X_{s+\e}(\cdot) \big)-
 DF\big(s+\epsilon,X_{s}(\cdot)\big),
 \big(X_{s+\epsilon}(\cdot)-X_{s}(\cdot) \big)  \rangle d\alpha
 \right] \times
\\
& \hspace{9.5cm} \times  \frac{\langle \delta_{0}, N_{s+\epsilon}
(\cdot)-N_{s}(\cdot)\rangle}{\epsilon} ds =
\\
&
=
\int_{0}^{t}  \int_{0}^{1}  
\langle DF\big(s+\epsilon, (1-\alpha)X_{s}(\cdot) +\alpha X_{s+\e}(\cdot) \big)\otimes \delta_{0}- DF\big(s+\epsilon,X_{s}(\cdot)\big)\otimes \delta_{0} \ , \\
&  \hspace{7cm}
 \frac{   \big(X_{s+\epsilon}(\cdot)-X_{s}(\cdot) \big)  \otimes  \big(N_{s+\epsilon}(\cdot)-N_{s}(\cdot) \big)  }{\epsilon} \rangle d\alpha \, ds   \; .
\end{split}
\ee
We will show that $R_1(\cdot, \varepsilon)$ converges ucp to zero,
when $\varepsilon \rightarrow 0$.
Since $\chi$ is a Hilbert
 space, making the proper Riesz identification 
for $t\in [0,T]$, $\eta_{1},\eta_{2}\in C([-\tau,0])$ the map $DF(t,\eta_{1})\otimes DG(\eta_{2})$ coincides with 
the tensor product $DF(t,\eta_{1})\otimes \delta_{0}$, see Proposition \ref{prop RRCBLbis}. 
As in Example \ref{pr STCHI2} the map $DF\otimes \delta_{0}:[0,T]\times C([-\tau,0])$ takes values in the separable space 
$\chi^{2}([-\tau,0]^{2})$ and 
it is a continuous map. In particular it takes values in $\chi$ which is a Hilbert subspace of $\chi^{2}([-\tau,0]^{2})$.\\
We denote by $\mathcal{U}=\mathcal{U}(\omega)$ the closed convex hull of the compact subset $\mathcal{V}$ of $C([-\tau,0])$ defined, for every $\omega$, by 
\be	\label{eq 4.25ter}
\mathcal{V}=\mathcal{V}(\omega):=\{X_{t}(\omega); \,t\in [0,T] \}		\; .
\ee
According to Theorem 5.35 from \cite{InfDimAn}, $\mathcal{U}(\omega)=\overline{conv(\mathcal{V})(\omega)}$ is compact, 
so the function $DF(\cdot,\cdot)\otimes\delta_{0}$ on $[0,T]\times \mathcal{U}$ is uniformly continuous and we denote 
by $\varpi^{[0,T]\times \mathcal{U}}_{DF(\cdot,\cdot)\otimes \delta_{0}}$ the continuity modulus of the application 
$DF(\cdot,\cdot)\otimes \delta_{0}$ restricted to $[0,T]\times \mathcal{U}$ and by $ \varpi_{X}$ the continuity modulus of the continuous process $X$.  
$\varpi^{[0,T]\times \mathcal{U}}_{DF(\cdot,\cdot)\otimes \delta_{0}}$ is, as usual, 
a positive, increasing function on $\R^{+}$ converging to zero when the argument converges to zero. So we have   
\begin{equation}		\label{eq: conv R1}
\sup_{t\in[0,T]} \vert R_{1}(t,\epsilon)\vert \leq \int_{0}^{T} \varpi^{[0,T]\times \mathcal{U}}_{DF(\cdot,\cdot)\otimes \delta_{0}}\left( \varpi_{X}(\epsilon) \right) \left\| \frac{   \big(X_{s+\epsilon}(\cdot)-X_{s}(\cdot) \big)  \otimes  \big(N_{s+\epsilon}(\cdot)-N_{s}(\cdot) \big)  }{\epsilon}\right\|_{\chi}ds  \; .
\end{equation}
We recall by Corollary \ref{corCDWMbis}, item 3. that $X(\cdot)$ and $N(\cdot)$ admit a $\chi$-covariation. 
In particular using condition {\bf H1}  and \eqref{eq: conv R1} the
claim
 $R_{1}(\cdot,\epsilon)\xrightarrow [\epsilon\rightarrow 0]{ucp }0$ follows.\\
On the other hand, the first addend in \eqref{eq: I1} can be rewritten as 
\begin{equation}		\label{eq: I11}
\int_{0}^{t} \langle DF\big(s, X_{s}(\cdot) \big)\otimes \delta_{0},  \frac{   \big(X_{s+\epsilon}(\cdot)-X_{s}(\cdot) \big)  \otimes  \big(N_{s+\epsilon}(\cdot)-N_{s}(\cdot) \big)  }{\epsilon} \rangle ds+ R_{2}(t,\epsilon)
\end{equation}
where $R_{2}(\cdot,\epsilon)\xrightarrow [\epsilon\rightarrow 0]{ucp }0$ arguing similarly as for $R_{1}(t,\epsilon)$.\\
In view of the application of Proposition \ref{pr CONVCCOV}, since 
$D F\otimes \delta_{0}:[0,T]\times C([-\tau,0])\longrightarrow \chi$
is continuous, we observe that the process 
$H_{s}= DF\big(s,X_{s}(\cdot)\big)\otimes \delta_{0}$ takes
 obviously values in the separable space $\chi$ which is a
 closed subspace of $\chi^{2}([-\tau,0]^{2})$.
Using bilinearity and Proposition \ref{pr CONVCCOV}, the integral 
in \eqref{eq: I11} converges then in probability 
to
\begin{equation}  \label{eq: I111}
\int_{0}^{t} \langle DF\big(s,X_{s}(\cdot)\big)\otimes \delta_{0}, d\widetilde {[X(\cdot),N(\cdot)]}_{s}\rangle  \; .
\end{equation}
As in Theorem  
\ref{thm STWDIRP}, item 1.,   we decompose $\chi$ in the following 
direct sum $\shd_{0,0}\oplus \chi_2$ where we recall that 
$\chi_2= \oplus_{i=1}^{N}\shd_{i, 0}\oplus \apt L^{2}([-\tau,0])\hat{\otimes}_h \shd_{0}\cpt$.
 By Corollary \ref{corCDWMbis} 2.,
 $X(\cdot)$ and $N(\cdot)$ admit a zero $\chi_2$-covariation.
By \eqref{eq prop 4} in Proposition \ref{pr PLM} 
it follows that \eqref{eq: I111} equals
\begin{equation}
\int_{0}^{t} (D^{\delta_{0}}F(s,X_{s}(\cdot))\otimes \delta_0)
(\{0,0\})  d[X,N]_{s} = 
\int_{0}^{t}D^{\delta_{0}}F(s,X_{s}(\cdot))d[M,N]_{s}  \ .
\end{equation}
We will  show now that $I_{2}(\cdot,\epsilon)\xrightarrow [\epsilon\rightarrow 0]{ucp }0$.\\
By stochastic Fubini's theorem we obtain
\[
I_{2}(t,\epsilon)
=
\int_{0}^{t}\xi(\epsilon,r ) dN_{r} 
\]
where
\[
\xi(\epsilon,r)= \frac{1}{\epsilon} \int_{0\vee(r-\epsilon)}^{r} 
\left[ F(s+\epsilon, X_{s}(\cdot))- F(s, X_{s}(\cdot))  \right] ds		\; .
\]
Proposition  2.26, chapter 3 of \cite{ks} says that $I_{2}(\cdot, \epsilon)\xrightarrow [\epsilon\rightarrow 0]{ucp }0$
if 
\begin{equation}		\label{eq EPXE}
\int_{0}^{T}\xi^{2}(\epsilon,r)d[N]_{r}\xrightarrow 
[\epsilon\rightarrow 0]{ }0			
\end{equation}
in probability.
We fix $\omega\in \Omega$ and we show that 
the convergence in \eqref{eq EPXE} holds in particular a.s.
We denote by $\varpi^{[0,T]\times \mathcal{U}}_{F}$ the continuity modulus of the application 
$F$ restricted to the compact set $[0,T]\times \mathcal{U}$. 
For every $r\in [0,T]$ we have 
\[
\left| \xi(\epsilon, r)
\right |   
\leq
\sup_{r\in [0,T]} \left| F(r+\epsilon,X_{r}(\cdot))-F(r,X_{r}(\cdot)) \right| \leq \varpi^{[0,T]\times \mathcal{U}}_{F}(\epsilon) 
\]
which converges to zero for $\epsilon$ going to zero 
since function $F$ on $[0,T]\times \mathcal{U}$ is uniformly continuous on the compact set and 
$\varpi^{[0,T]\times \mathcal{U}}_{F}$ is, as usual, 
a positive, increasing function on $\R^{+}$ converging to zero when the argument converges to zero. By Lebesgue's dominated 
convergence theorem we finally obtain \eqref{eq EPXE}.
%
\end{proof}
%
If $DF$ does not necessarily leave in some 
$\shd_{a}([-T,0])\oplus L^{2}([-T,0])$ space, it is still possible to express 
a variant of Theorem \ref{thm: thm 1}. The price to pay is a new property required for $DF$ which will be called 
\emph{support predictability property}. 
It is described below.

\begin{dfn} 		\label{del SPC}
Let $0\leq a<b\leq T$. A function $F:[a,b] \times C([-\tau,0])\lra \R$ such that 
$F(t,\cdot)$ is differentiable for any $t\in [a,b]$ is said to fulfill the \textbf{support predictability property} if the following holds.
For every compact $K$ of $C([-\tau,0])$, we have 
\be		\label{SPP}
\int_{a}^{b}  \left[ \sup_{\eta\in K} \frac{1}{\e}\int_{(-\e)\vee (-\tau)}^{0} \left| D^{\perp}_{dr}F \right|(t,\eta) \right] dt=O(\e) \ , 
\ee
where we recall that $D_{dr}^{\perp}F=D_{dr}F- DF(\{0 \})\delta_0(dr)$.
\end{dfn}

\begin{rem}		\label{remGBN}
Suppose that $F(t,\cdot)$ is differentiable for any $t\in [a,b]$.
\begin{enumerate}
\item Suppose the existence of $\rho >0 $ such that $D^{\perp}F(t,\eta)$ has support in $[-\tau,-\rho]$ for any $t\in [a,b]$, $\eta\in C([-\tau,0])$. 
Then $F$ fulfills the support predictability property; in fact quantity \eqref{SPP} vanishes for $\e$ small.
\item Suppose that $D^{\perp}F(t,\eta)$ is absolutely continuous for every $t\in [a,b]$.
We denote $\apt D_r^{\perp}F(t,\eta) , r\in [-\tau, 0] \cpt$, the corresponding density. 
If for any compact $K$ of $C([-\tau,0])$ there is $\rho_1>0$ such that 
$t\mapsto \sup_{r\in [-\rho_1,0],\eta\in K}\left| D^{\perp}_{r}F(t,\eta) \right|$ belongs to $L^{1}([a,b])$, then 
$F$ fulfills the support predictability property. 
This is for instance verified if $(r,t,\eta)\mapsto D^{\perp}_{r}F(t,\eta)$ is continuous. 
\end{enumerate}
\end{rem}
As announced a variant of Theorem \ref{thm: thm 1} is given below.
%
\begin{thm}	\label{thm:thm1B}
Let $0\leq a< b\leq T$ and 
$X$ be an $(\mathcal{F}_{t})$-weak Dirichlet process with finite quadratic variation and decomposition $X=M+A$, $M$ local martingale. 
Let $F:[a,b]\times C([-\tau,0])\lra \R$ continuous such that 
\begin{itemize}
\item [i)] $F(t,\cdot)$ is differentiable for every $t\in [a,b]$,
\item [ii)] $(t,\eta)\mapsto D^{\perp}F(t,\eta)$ is bounded on each compact of $[a,b]\times C([-\tau,0])$,
\item [iii)] $F$ fulfills the support predictability property,
\item [iv)]  $(t,\eta)\mapsto D^{\delta_0}F(t,\eta)$ is continuous on $]a,b]\times C([-\tau,0])$ and it admits a continuous extension on $[a,b]\times C([-\tau,0])$.
\end{itemize}
Then $F\apt \cdot , X_{\cdot}(\cdot) \cpt$ is an $(\mathcal{F}_{t})$-weak Dirichlet process with martingale part
\be \label{F1B}
\tilde{M}_t^{F}=F\apt a,X_{a}(\cdot) \cpt +\int_{a}^{t} D^{\delta_0}F\apt s,X_{s}(\cdot) \cpt dM_s \ , \quad t\in [a,b]  \ .
\ee
\end{thm}
\begin{proof} 
Without restriction of generality we will suppose $a=0$ and $b=T$.
The proof follows from a modification of the one of Theorem \ref{thm: thm 1}. 
\eqref{eq ZERO} was expressed as the sum of $I_1(t,\e)$ and $I_2(t,\e)$. $I_1(t,\e)$ is the sum of 
$I_{11}(t,\e)$ and $I_{12}(t,\e)$ where 
\[
\begin{split}
I_{11}(t,\e) &=
\int_{0}^t D^{\delta_0}F\apt s+\e, X_s(\cdot)  \cpt \frac{ \apt X_{s+\e}-X_s \cpt \apt N_{s+\e}-N_s\cpt }{\e}ds   \, , \\
I_{12}(t,\e) &=
\int_{0}^{t} \int_{0}^{1}  \left[  D^{\delta_0}F\apt s+\e,(1-\alpha) X_s(\cdot)+ \alpha X_{s+\e}(\cdot) \cpt  -D^{\delta_0}F\apt s+\e, X_s(\cdot) \cpt  \right]d\alpha\ \frac{\apt X_{s+\e}-X_s \cpt \apt N_{s+\e}-N_s\cpt}{\e}  ds   \, , \\
I_{13}(t,\e) &=
\int_{0}^{t} 
\int_{0}^{1}
\prescript{}{\mathcal{M}([-\tau,0])}{\langle }   
D^{\perp} F \apt s+\e,(1-\alpha) X_s(\cdot)+ \alpha X_{s+\e}(\cdot) \cpt  \, , \apt X_{s+\e}(\cdot)-X_s(\cdot)\cpt \rangle_{C([-\tau,0])}  d\alpha\ \frac{ \apt N_{s+\e}-N_s\cpt}{\e}  ds  \ . \\
\end{split}
\]
We have 
\[
I_{11}(t, \e) =J_{11}(t,\e)+R_{11}(t,\e)
\]
where 
\[
J_{11}(t,\e)
=
\int_{0}^{t} 
D^{\delta_0}F\apt s, X_s(\cdot)  \cpt \frac{ \apt X_{s+\e}-X_s \cpt \apt N_{s+\e}-N_s\cpt }{\e}ds  
\]
and 
$
\sup_{t\leq T} \left| R_{11}(t,\e)  \right| \xrightarrow[\epsilon\lra 0]{} 0
$
in probability because $D^{\delta_0} F$ is continuous by item iv)  and there uniformly continuous on each compact.
In fact $(X,N)$ have all their mutual covariations, then by Proposition \ref{prFD} and the fact that $[X,N]=[M,N]$, clearly, 
\[
J_{11}(t,\e)\xrightarrow[\e\lra 0]{ucp} 
\int_{0}^{t} D^{\delta_0} F\apt s,  X_{s}(\cdot) \cpt d[M,N]_s   \ .
\] 
$I_{12}(t,\e)$ behaves similarly to $R_{1}(t,\epsilon) $ in \eqref{eq 4.25bis}, so it converges ucp to zero. 
Term $I_{13}(t,\e)$ can be rewritten as
\[
I_{13}(t,\e)=
\int_0^t
\int_0^1
\int_{-\tau}^{0} D^{\perp}_{dr}F\apt s+\e, (1-\alpha)X_{s}(\cdot)+\alpha X_{s+\e}(\cdot) \cpt \apt X_{s+r+\e}-X_{s+r}\cpt d\alpha \  \frac{N_{s+\e}-N_{s} }{\e} ds
\]
and it decomposes into $J_{13}(t,\e)+R_{13}(t,\e)$ where 
\[
J_{13}(t,\e)=
\int_{0}^{t} Z_{s}(\e)  \frac{N_{s+\e}-N_{s} }{\e} ds
\]
with 
\[
Z_{s}(\e)
=
\int_0^1
\int_{-\tau}^{-\e} D^{\perp}_{dr}F\apt s+\e, (1-\alpha)X_{s}(\cdot)+\alpha X_{s+\e}(\cdot) \cpt \apt X_{s+r+\e}-X_{s+r}\cpt d\alpha 
\]
and 
\[
\begin{split}
R_{13}(t,\e)& 
=
\int_{0}^{t} \int_0^1
\int_{-\e}^{0} D^{\perp}_{dr}F\apt s+\e, (1-\alpha)X_{s}(\cdot)+\alpha X_{s+\e}(\cdot) \cpt \apt X_{s+r+\e}-X_{s+r}\cpt d\alpha \ \frac{N_{s+\e}-N_{s} }{\e} ds\\
&
=
\int_{\e}^{t+\e} \int_0^1
\int_{-\e}^{0} D^{\perp}_{dr}F\apt s, (1-\alpha)X_{s-\e}(\cdot)+\alpha X_{s}(\cdot) \cpt \apt X_{s+r}-X_{s+r-\e}\cpt d\alpha \ \frac{N_{s}-N_{s-\e} }{\e} ds
\end{split}
\]
By stochastic Fubini's theorem we obtain
\[
J_{13}(t,\e)=\int_{0}^{t} \xi(u,\e) dN_{u}   \ , 
\quad 
\textrm{with} 
\quad
\xi(u,\e)=\frac{1}{\e}\int_{(u-\e)^{+}}^{u} Z_{s}(\e) ds  \; .
\]
Proposition  2.26, chapter 3 of \cite{ks} says that $J_{13}(\cdot, \epsilon)\xrightarrow [\epsilon\rightarrow 0]{ucp }0$ if 
\be 				\label{D4}
\int_{0}^{T} \xi^2(u,\e)d[N]_{u} 
\xrightarrow [\epsilon\rightarrow 0]{ }0
\ee
in probability. 
We have 
\[
\begin{split}
\left| 
Z_{s}(\e)
\right|
& \leq 
\int_{(u-\e)^{+}}^{u}
\frac{1}{\e} 
\int_{0}^{1}
\int_{-\tau}^{-\e} \left\|   
D^{\perp}_{dr}F\apt s+\e , (1-\alpha) X_{s}(\cdot) +\alpha X_{s+\e}(\cdot) \cpt 
\right\|
\left| 
X_{s+r+\e}-X_{s+r}
\right|
 d\alpha \ ds\\
 &
 \leq 
 \varpi_{X}(\e) \int_{0}^{1} \left\|   
D^{\perp}_{dr}F\apt s+\e , (1-\alpha) X_{s}(\cdot) +\alpha X_{s+\e}(\cdot) \cpt 
\right\|_{Var([-\tau,0])}d\alpha\\
&
\leq 
 \varpi_{X}(\e) \sup_{(t,\eta)\in [0,T]\times \mathcal{U}}
 \left\|  D^{\perp}_{dr}F \right\|_{Var([-\tau,0])}(t,\eta)
 \\
 \end{split}
 \]
 where $\mathcal{U}=\mathcal{U}(\omega)=\overline{conv(\mathcal{V})(\omega)}$ where $\mathcal{V}(\omega)$ was defined in \eqref{eq 4.25ter}. 
 Previous expression is bounded because of 
 item ii) in the assumptions and since $\mathcal{U}$ is a compact set in the infinite dimensional space $C([-\tau,0])$. 
So $\xi^{2}(u,\e)\leq \varpi_{X}^{2}(\e) \sup_{(t,\eta)\in [0,T]\times \mathcal{U}} \left\|
D^{\perp}_{dr}F
\right\|^{2}_{Var([-\tau,0])}$ since 
\[
\left| \xi (u,\e) \right|=\left|\frac{1}{\e}\int_{(u-\e)^{+}}^{u} Z_{s}(\e) ds \right| \leq  \varpi_{X}(\e) \sup_{(t,\eta)\in [0,T]\times \mathcal{U}}
 \left\|  D^{\perp}_{dr}F \right\|_{Var([-\tau,0])}(t,\eta)  \ .
\] 
Finally the left-hand side of \eqref{D4} is bounded by 
\[
 \varpi_{X}^{2}(\e) \sup_{(t,\eta)\in [0,T]\times \mathcal{U}} \left\|
D^{\perp}_{dr}F
\right\|^{2}_{Var([-\tau,0])} \int_{0}^T d[N]_{u}=\varpi_{X}^{2}(\e) \sup_{(t,\eta)\in [0,T]\times \mathcal{U}} \left\|
D^{\perp}_{dr}F
\right\|^{2}_{Var([-\tau,0])}[N]_{T}
\]
which converges to zero a.s. for $\e \lra 0$.\\
It remains to control $R_{13}(t,\e)$. This term is bounded by 
\[
\varpi_{N}(\e)\ \varpi_{X}(\e) \int_{0}^{T+\e }\frac{1}{\e} \left[ \sup_{\eta\in\mathcal{U}(\omega)} \int_{-\e}^{0} \left| D_{dr} F\right| (t,\eta) \right] dt
\]
The result follows since $F$ fulfills the support predictability property.
\end{proof}
We present a slight generalization of Theorem \ref{thm:thm1B}.
\begin{thm}		\label{thm ter}
Let $X$ be an $(\F_t)$-weak Dirichlet process with finite quadratic variation with decomposition $X=M+A$, $M$ local martingale.
Let $F: [0,T] \times C([-\tau,0])\lra \R $ continuous, fulfilling
 assumptions i), ii) and iii) of Theorem \ref{thm:thm1B} for $a=0$ and $b=T$. 
Suppose the existence of $0=a_0< a_1<\ldots<a_{N}=T$ such that $(t,\eta)\mapsto D^{\delta_0}F(t,\eta)$ is continuous 
on $]a_{i},a_{i+1}]\times C([-\tau,0])$ admitting a continuous
  extension
 on $[a_{i},a_{i+1}]\times C([-\tau,0])$, for $0\leq i\leq (N-1)$.\\ 
Then $\apt F\apt t,X_{t}(\cdot) \cpt \cpt_{t\in [0,T]}$ is an
 $(\mathcal{F}_{t})$-weak Dirichlet process with 
local martingale part
\be
\tilde{M}^{F}_{t}=F\apt 0, X_{0}(\cdot) \cpt +\int_0^t D^{\delta_{0}} F\apt s,X_{s}(\cdot)\cpt d M_s \ .
\ee
\end{thm}
\begin{proof}
Let $N$ be an $(\F_t)$-local martingale. Since 
 \be		\label{eq FC1} 
 \left[ \tilde{M}^{F}, N \right]_t=\int_0^t D^{\delta_{0}} F\apt s,X_{s}(\cdot)\cpt d [M,N]_s  \ , \quad t\in [0,T]  \ ,
 \ee
 it will be enough to show that 
 \be		\label{eq 4.36}
 \left[ (F\apt t,X_{t}(\cdot))\cpt  , N \right]_t=
 \int_0^t D^{\delta_{0}} F\apt s,X_{s}(\cdot)\cpt d [M,N]_s  \ , \quad t\in [0,T]  \ .
 \ee 
 We observe that 
$F |_{[a_i, a_{i+1}]\times C([-\tau,0])}$ verifies the assumptions of Theorem \ref{thm:thm1B} with $a=a_i$, $b=a_{i+1}$. 
Consequently for $t\in ]a_{i},a_{i+1}]$,  $i\in \{0, \ldots, (N-1) \}$, $\apt F\apt t,X_{t}(\cdot) \cpt \cpt_{t\in [a_i,a_{i+1}]}$ is an $(\mathcal{F}_{t})$-weak Dirichlet process with 
local martingale part 
\[
\tilde{M}^{i}_{t}=F\apt a_{i}, X_{a_i}(\cdot) \cpt +\int_{a_i}^t D^{\delta_{0}} F\apt s,X_{s}(\cdot)\cpt d M_s \ , \quad t\in[a_i, a_{i+1}[  \ .
\]
\eqref{eq 4.36} follows by summation.
\end{proof}
We discuss now some consequences related to the martingale representation.
\subsection{About some martingale representation}
Suppose that $X$ is an $(\F_t)$-weak Dirichlet process with finite quadratic variation with decomposition $X=M+A$, $M$ local martingale.
Let $h\in L^{1}(\Omega)$. 
We are interested in sufficient conditions so that
\begin{equation} 		\label{eq: marting represent no br}
h=h_0+\int_{0}^{T} \xi_{s}dM_{s}
\end{equation} 
where $(\xi_{s})$ is an explicit predictable process, $h_0\in \R$.\\

The two results below are a consequence respectively of
 Theorems \ref{thm: thm 1} and \ref{thm ter}. They settle the basis
 for a representation 
of integrable random variables. 
 $\shd_{a}\oplus L^{2}$ will denote here $\shd_{a}([-\tau,0])\oplus 
L^{2}([-\tau,0])$. 
%
\begin{prop}		\label{cor thm 1 weak Dirichlet}
Let $F:[0,T]\times C([-\tau,0])\longrightarrow \R$ 
continuous such that $(s,\eta)\mapsto D F(s,\eta)$ exists with
 values in $\shd_{a}\oplus L^{2}$ and 
$DF:[0,T] \times C([-\tau,0])\longrightarrow \shd_{a}\oplus L^{2}$ is
 continuous. 
If moreover 
\begin{equation}	\label{eq F445}
\mathbb{E}\left[  h\vert \F_{t}  \right]= F(t,X_{t}(\cdot)) \ \textrm{ a.s.} \quad \forall\; t\in [0,T[ 
\end{equation}
then
\begin{equation}	\label{eq F446}
h=
F(0,X_{0}(\cdot))+\int_{0}^{T} D^{\delta_{0}}F(s,X_{s}(\cdot))dM_{s}  \; .
\end{equation}
\end{prop}
\begin{rem}
We observe that $F\apt 0, X_0(\cdot) \cpt=\mathbb{E}[h\vert \mathcal{F}_{0}]$.
\end{rem}
\begin{proof}\
Since $F$ verifies the assumptions of Theorem \ref{thm: thm 1}, then 
$F(\cdot, X_{\cdot}(\cdot))$ is an $(\F_{t})$-weak Dirichlet process with martingale 
part given by 
\begin{equation}		\label{eq mrtpart}
M^{F}_{t}=F(0,X_{0}(\cdot))+\int_{0}^{t} D^{\delta_{0}} F(s, X_{s}(\cdot)) dM_{s}	\ ,
\end{equation}
according to \eqref{eq: mart part uC01}.
By \eqref{eq F445}, $F(\cdot, X_{\cdot}(\cdot))$ is obviously an
 $(\F_{t})$-martingale being a conditional expectation with respect to 
filtration $(\F_{t})$. By the uniqueness of the decomposition of $(\F_{t})$-weak Dirichlet processes, 
it follows 
\[
F(t,X_{t}(\cdot))=F(0,X_{0}(\cdot))+\int_{0}^{t}D^{\delta_{0}}F(s,X_{s}(\cdot))dM_{s}\; .
\]
In particular the $(\F_{t})$-martingale orthogonal component is zero. 
Since $h=F(T,X_{T}(\cdot))$ and $F(0,X_{0}(\cdot))=\mathbb{E}[h\vert \mathcal{F}_{0}]$ the result follows.
\end{proof}
%
%
\begin{cor}		\label{cor CCS}  
Let $X$ be an $(\F_t)$-weak Dirichlet process with finite quadratic variation with decomposition $X=M+A$, $M$ local martingale.\\
Let $0< a_1< \ldots < a_{N}=T$. 
Let $h\in L^{1}(\Omega)$. Let $F:[0,T]\times C([-\tau, 0])\lra \R$ verifying the following properties.
\begin{itemize}
\item [a)] $F(s,\cdot)$ is differentiable $\forall \ s\in [0,T]$.
\item [b)] $F$ fulfills the support predictability property.
\item [c)] $(s,\eta)\mapsto D^{\delta_{0}}F(s,\eta)$ is continuous on
 $]a_i,a_{i+1}]$, $0 \leq i\leq N-1$ and it admits a continuous extension
 on $[a_{i},a_{i+1}]$.
\item [d)] $(s,\eta)\mapsto D^{\perp}F(s,\eta)$ is bounded for each
 compact,
 with respect to the total variation norm.
\item [e)] $F\apt s,X_{s}(\cdot) \cpt =\E\left[ h | \F_s \right]$.
\end{itemize}
Then 
\be	\label{eq HED}
h=F(0,X_{0}(\cdot))+\int_{0}^{T} D^{\delta_{0}}F \apt s, X_{s}(\cdot) \cpt dM_s  \ .
\ee
\end{cor}
\begin{proof}
By Theorem \ref{thm ter} 
$\apt F(\apt t, X_{t}(\cdot)\cpt \cpt_{t\in [0,T]}$ is an $(\mathcal{F}_{t})$-weak Dirichlet process. On the other hand by construction 
$\apt F(\apt t, X_{t}(\cdot)\cpt \cpt_{t\in [0,T]}$ is an 
$(\F_{t})$-martingale. The result follows again by uniqueness of the decomposition of an $(\F_t)$-weak Dirichlet process.
\end{proof}

\begin{rem}	\label{rem5.19}
Suppose $h\in L^2(\Omega)$. Indeed previous corollary can be stated 
requiring the same assumptions a), b), c), d), e) 
but on $F$ restricted to $[0,b] \times C([-\tau,0])$
for 
 every $b\in ]a_{N-1}, T[$. 
In fact, the square integrable martingale $\E\apq h | \F_t\cpq $ can
 be decomposed according to Kunita-Watanabe's theorem,
 it can be decomposed into the sum 
\[
F(0,X_{0}(\cdot))+\int_{0}^t\xi_s dM_s +N_t 	 \ , \quad t\in [0,T]
\]
where $N$ is an $(\F_t)$-local martingale strongly orthogonal to $M$, 
i.e. $[N,M]=0$; moreover we know that 
$\int_{0}^T \xi^{2}_sd[M]_s<+\infty $ a.s. Since
 $\apt F(t, X_t(\cdot)) \cpt_{t\in [0,b]}$ is also an $(\F_t)$-weak 
Dirichlet process by Theorem \ref{thm ter}, 
the uniqueness of the decomposition implies that 
\[
\int_{0}^t D^{\delta_0} F\apt s, X_s(\cdot) \cpt dM_s =\int_{0}^t \xi_s dM_s+N_t \ ,\quad t\in [0,b]
\]
and so $N_s = 0, s\in [0,b]$ and 
$\xi_s\equiv D^{\delta_0}F(s,X_s(\cdot))  $. This implies \eqref{eq HED}.
\end{rem}
\section{Consequences on quasi explicit representations of path dependent random variables}		\label{sec: last}
\subsection{General considerations}
This section has illustrative features. We test our method on the 
representation of a random variable which depends on the path of a
 diffusion process. 
It will be a toy model for future investigations with applications in verification theorems in control theory on functional dependent equation.
Let $\apt W_t\cpt_{t\in [0,T]}$ be a standard Brownian motion with
respect 
to some usual filtration $(\F_t)$.
In this section $\eta$ (resp. $\gamma$) will denote an element in $C([-T, 0])$ (resp. $C([0,T])$). 
Let $(X_t)_{t\in [0,T]}$ be a real continuous process; with $X$ we will also denote the whole trajectory of $X$ in $C([0,T])$. 
$a$ will stand for a grid $0=a_0<a_1<\ldots <a_N=T$, not anymore in
 $[-T,0]$ as before. In this section $\tau$ will be equal to $T$.
 
 We aim at implementing Corollary \ref{cor CCS} when $X$ is solution 
of a SDE of the type 
\be		\label{eq SDE}
X_{t}=X_0+\int_0^t\sigma(r,X_r)dW_r +\int_{0}^t b(r,X_r) dr \; ,
\ee
where $\sigma,b:[0,T]\times \R\rightarrow \R$ are continuous, $\partial_x \sigma$ and $\partial_x b$ exist and they are bounded. 
We remark that $\sigma$ is possibly degenerate. 
With  $\partial_x \sigma$ (resp. $\partial_x b$) we will denote the derivative of $\sigma$ (resp. $b$) with respect to the second argument. 
In the present context $X$ is an $(\F_t)$-semimartingale, therefore  an $(\F_t)$-weak Dirichlet, with decomposition $X=M+A$, $M$ the local martingale 
given by $M_t=X_0+\int_{0}^{t}\sigma(r,X_r)dW_r$.\\
The idea is to evaluate $$h:=\Phi(X)$$ with $\Phi:C([0,T])\lra \R$ such that $h\in L^{1}(\Omega)$.\\ 
It is well known that the following flow property holds: 
\be		\label{eq 2}
X_t=Y_t^{s,X_s(\cdot)}
\ee
where for $(s,\eta) \in [0,T]\times  C([-\tau,0])$, where $Y^{s,\eta}$ is an element in $C([0,T])$, in fact $Y_t=Y_t^{s,\eta}$, $t\in [0,T]$ is defined by 
\be		\label{eq FLUSSO}
Y_t=
\left\{
\ba{ll}
\eta(t-s)		& t\in [0,s]\\
\eta(0)+\int_s^t \sigma(r,Y_r)dW_r	+\int_{s}^t b\apt r,Y_{r} \cpt dr	& t\in [s,T]   \ .
\ea
\right. 
\ee
$Y^{s,\eta}$ also symbolizes a function in $C([0,T])$.
We start with some basic estimates.
\begin{prop}		\label{PPP}
For every $q\geq 1$, there is a constant $C(q)$ such that 
\[
\sup_{s\in [0,T]}\E\left[ \left\|  Y^{s,\eta}\right\|^{q}_{\infty}\right]\leq C(q) \left\|\eta\right\|_{\infty}^{q}
\]
\end{prop}
\begin{proof}
Since $\sigma$, $b$ have linear growth, by Burkholder-Davis-Gundy inequality, for every $q\geq 1$, and by Gronwall lemma, there is a constant $C(q)$ such that 
\[
\sup_{0\leq s\leq t\leq T} \E\apq \left| Y^{s,\eta}_t \right|^{q} \cpq 
\leq C(q) \ \left| \eta(0)\right|^{q} \ .
\]
This implies the result.
\end{proof}
Process $Y = Y^{s,\eta}$ given in \eqref{eq FLUSSO} also solves 
$$Y_t=\eta(0)+\int_s^t \sigma\apt r,Y_r^{s,\eta}\cpt d\bar{W}_r+\int_{s}^t b\apt r,Y_{r}^{s,\eta} \cpt dr$$
where $\bar{W}=W_{s+\cdot}-W_s$ is independent of $\mathcal{F}_{t}=\sigma\apt Y_r, r\leq t \cpt$. 
Suppose now $\Phi$ being continuous. Taking into account \eqref{eq 2}, for $s\leq t$, we get
\be		\label{eq Ms}
M_s:= \mathbb{E}\left[ \Phi(X) | \mathcal{F}_s \right]=u\apt s,X_s(\cdot) \cpt
\ee
where $u:[0,T]\times C([-T,0])\lra \R$
\be		\label{eq U}
u\apt s,\eta \cpt =\mathbb{E}\left[ \Phi(Y^{s,\eta}) \right]
\ee
and $Y^{s,\eta}$ is given by 
\eqref{eq FLUSSO}
\begin{rem}
\begin{enumerate}
\item We remark that $Y^{s,\eta}:\Omega\lra C([0,T])$. 
\item It is possible to show that $u\in C^{0}\apt[0,T]\times C([-T,0]) \cpt$.
\end{enumerate}
\end{rem}
We would like to examine situations where $F:=u$ fulfills the conditions for the application of Corollary \ref{cor CCS}. 
In particular Theorem \ref{thm META} below verifies condition a) of that corollary. The other assumptions will be the object of Proposition \ref{pr PCS}.
In this section the natural candidate for function $u$ is given in \eqref{eq U}.
First we introduce a notation and we recall that 
$D^{\delta_0}u(t,\eta) =Du(t,\eta)\{0\}$ and $D^{\perp}_{dr}u(t,\eta)$ is singular with respect to $\delta_0$. 
\begin{nota} \label{not6.3}
If $X$ is a continuous real $(\F_t)$-semimartingale, symbol $\mathcal{E}\apt X \cpt$ 
denotes the Dol\'eans exponential of $X$, in particular 
$\mathcal{E}\apt X \cpt_t = \exp\left\{X_t-X_0-\frac{1}{2}[X]_t\right\}$.\\
\end{nota}

\begin{thm}		\label{thm META}
Let $X$ be a diffusion process of type \eqref{eq SDE} and $\Phi:C([0,T])\lra \R$ of class $C^{1}$ Fr\'echet such that $D\Phi$ has polynomial growth. 
For $s\in [0,T]$ and $\eta\in C([-T,0])$, we set $u:[0,T]\times C([-T,0])\lra \R$ by 
$u(s,\eta)=\E\left[ \Phi \apt Y^{s,\eta}\cpt \right]$ where
$Y^{s,\eta}$ belongs to $C([0,T])$ and it is defined by \eqref{eq FLUSSO}.\\
Then $u\in C^{0}([0,T]\times C([-T,0]))$. Moreover for every $s\in [0,T]$, $u(s,\cdot)$ belongs to $C^{1}\apt C([-T,0])\cpt$ and 
$Du(s,\eta)$, with $s\in [0,T]$ and $\eta\in C([-T,0])$, is characterized by 
\be		\label{eq 4.1}
D_{dr}u(s,\eta)=D^{\perp}_{dr}u(s,\eta)+D^{\delta_0}u(s,\eta)\delta_0(dr)
\ee
with 
\be \label{eq4.2}
\begin{split}
D^{\perp}_{dr}u(s,\eta)
&
=\mathbb{E}\left[  D_{s+dr}\Phi \apt Y^{s,\eta}\cpt  1_{[-s,0[}(r) \right] \quad \textrm{and} \\
D^{\delta_0}u(s,\eta)
&
= 
\mathbb{E}\left[ \int_{[s,T]}   D_{d\rho} \Phi\apt Y^{s,\eta}\cpt \ 
\mathcal{E}\left\{  \int_s^\rho \partial_x\sigma \apt \xi,Y_\xi^{s,\eta} \cpt  dW_\xi +\int_s ^\rho \partial_{x} b \apt  \xi , Y_\xi^{s,\eta} \cpt d\xi   \right\}  \right]
\end{split} \ . 
\ee
In particular item a) of Corollary \ref{cor CCS} is verified.
\end{thm}
\begin{proof}
We recall that $Y:[0,T]\times C([-T,0])\times \Omega\lra C([0,T])$ is a.s. continuous. 
We suppose for a moment that $\partial_x\sigma$ and $\partial_x b$ are H\"older continuous. 
In this case it is possible to show that $\eta\mapsto Y^{s,\eta}$ is of class $C^{1}\apt C([-T,0]) \ ; C([0,T]) \cpt$ a.s.
Let $Y:[0,T]\times C([-T,0])\times \Omega\lra C([0,T])$, $(s,\eta,\omega)\mapsto \apt Y_{t}^{s,\eta}(\omega)\cpt_{t\in [0,T]}$; 
then the first order Fr\'echet derivative with respect to the second argument $\eta$ will be 
$DY:[0,T]\times C([-T,0])\times \Omega\lra \mathcal{L}\apt C([-T,0]), C([0,T])\cpt $, i.e. 
$DY^{s,\eta}:C([-T,0])\lra C([0,T])$ is a linear functional. 
\vspace{-0.2cm}
\begin{rem}
If we fix $t\in [0,T]$ $\omega$-a.s. then it holds 
$Y_t:[0,T]\times C([-T,0])\times \Omega\lra \R$; in this case the
first order
 Fr\'echet derivative with respect to the second argument $\eta$ will be 
$DY_t:[0,T]\times C([-T,0])\times \Omega\lra \apt C([-T,0]) \cpt ^*=\mathcal{M}([-T,0])$. 
In particular if $f\in C([-T,0])$, 
\[
\prescript{}{\mathcal{M}([-T,0])}{\langle} DY_{t}^{s,\eta}(\omega)\ , \ f \rangle_{C([-T,0])}=\int_{[-T,0]} f(r ) D_{dr} Y^{s,\eta}_{t}(\omega) \hspace{3cm} \textrm{\cvd}
\]
\end{rem}
We go on with the proof of Theorem \ref{thm META}.
Computing the derivative $D Y_{t}^{s,\eta}$ it gives 
\be		\label{eq DerY1}
D_{dr} Y_{t}^{s,\eta}=
\left\{
\ba{ll}
\delta_{t-s}(dr) 				& t\leq s\\
\delta_{0}(dr)+\int_s ^t \partial_{x}\sigma\apt  \xi , Y_\xi^{s,\eta} \cpt D_{dr}Y_\xi^{s,\eta} dW_\xi  +\int_s ^t \partial_{x} b \apt  \xi , Y_\xi^{s,\eta} \cpt D_{dr}Y_\xi^{s,\eta} d\xi
& t\geq s  \ .
\ea
\right.
\ee
Consequently, for $t\geq s$, it follows 
\be				\label{eq DerY}
D_{dr} Y_{t}^{s,\eta}=
\left\{
\ba{ll}
\delta_{t-s}(dr) 	=\delta_{dr}(t-s)			& t\leq s\\
&\\
\delta_{0}(dr)\mathcal{E}\left\{  \int_s^t \partial_x\sigma \apt \xi,Y_\xi^{s,\eta} \cpt  dW_\xi +\int_s ^t \partial_{x} b \apt  \xi , Y_\xi^{s,\eta} \cpt d\xi   \right\}
& t\geq s \ .
\ea
\right.
\ee
We remind from Notation \ref{not6.3} that  
$$
\mathcal{E}\left\{  \int_s^t \partial_x\sigma \apt \xi,Y_\xi^{s,\eta} \cpt  dW_\xi +\int_s ^t \partial_{x} b \apt  \xi , Y_\xi^{s,\eta} \cpt d\xi   \right\}
$$
equals
$$
\exp\left\{ \int_s^t \partial_x\sigma \apt \xi,Y_\xi^{s,\eta} \cpt  dW_\xi - 
\frac{1}{2}\int_s^t \apt \partial_{x}\sigma\cpt^{2} \apt  \xi, Y_\xi^{s,\eta}\cpt d\xi  +\int_s ^t \partial_{x} b \apt  \xi , Y_\xi^{s,\eta} \cpt d\xi  \right\}  \ .
$$
Moreover, by usual integration theory for every $t\in [0,T]$, $u (s, \cdot)$ is of class $C^{1}\apt C([-T,0])\cpt$ and 
\be		\label{eq 4}
D_{dr}u\apt s,\eta \cpt=\mathbb{E}\left[ \int_{[0,T]} D_{d \rho}\Phi \apt Y^{s,\eta} \cpt D_{dr} Y_\rho \ \right]
\ee
Taking into account \eqref{eq DerY}, by composition, one obtains a precise evaluation which can be done again via the usual integration results. 
Omitting the details we have
\be		\label{eq4.1}
\begin{split}
D_{dr}u(s,\eta)&
=
\mathbb{E}\left[ \int_{[0,s[}   D_{d \rho} \Phi\apt Y^{s,\eta}\cpt \ D_{dr}Y^{s,\eta}_\rho \  \right]
+
\mathbb{E}\left[ \int_{[s,T]}   D_{d \rho} \Phi\apt Y^{s,\eta}\cpt \   D_{dr}Y^{s,\eta}_\rho \ \right]\\
&=
\mathbb{E}\left[ \int_{[0,s[}   D_{d \rho} \Phi\apt Y^{s,\eta}\cpt \ \delta_{\rho-s}(dr) \  \right]\\
&+
\mathbb{E}\left[ \int_{[s,T]}   D_{d \rho} \Phi\apt Y^{s,\eta}\cpt \  
\delta_{0}(dr)
\mathcal{E}\left\{  \int_s^\rho \partial_x\sigma \apt \xi,Y_\xi^{s,\eta} \cpt  dW_\xi +\int_s ^\rho \partial_{x} b \apt  \xi , Y_\xi^{s,\eta} \cpt d\xi   \right\}  \right]\\
&=
\mathbb{E}\left[  D_{s+dr}\Phi \apt Y^{s,\eta}\cpt \right] 1_{[-s,0[}(r)+\\
&+
\delta_{0}(dr) \mathbb{E}\left[ \int_{[s,T]}   D_{d \rho} \Phi\apt Y^{s,\eta}\cpt \  
\mathcal{E}\left\{  \int_s^\rho \partial_x\sigma \apt \xi,Y_\xi^{s,\eta} \cpt  dW_\xi +\int_s ^\rho \partial_{x} b \apt  \xi , Y_\xi^{s,\eta} \cpt d\xi   \right\} \right]
\end{split} \ .
\ee
Finally we obtain \eqref{eq 4.1} and \eqref{eq4.2}.\\
If $\partial_{x}\sigma $ and $\partial_{x}b$ are not H\"older continuous, $Y$ will not be a.s. differentiable in $\eta$, but only 
in $L^{2}(\Omega)$, i.e. in quadratic mean. However the two expressions in \eqref{eq4.2} still remain valid. We omit the details.
\end{proof}

\begin{rem}
In fact, the regularity assumption on $\sigma$ and $b: [0,T]\times \R \lra \R$ can be partly reduced, see for instance 
using the technique developed in \cite{flru2}. The expression $DY^{s,\eta}$ 
is the same as in the one of \eqref{eq DerY} under Hypotheses given in Section 1.2 and 6 of \cite{flru2}.
\end{rem}
\subsection{Some representations}
In this section we discuss a basis example with some particular cases of $\Phi$ and we express the consequences on $u$ appearing in Theorem \ref{thm META}.\\
$\nabla$ will denote the derivative with respect to the argument in $L^{2}([-T,0])$.
Symbol $D_{\rho}^{ac} \Phi\apt \gamma \cpt$ denotes the absolute continuous component of the first order Fr\'echet derivative of $\Phi$ and in all cases it coincides with the 
partial derivative of $f$ in $L^{2}$ with respect to the last argument which will be denoted in the sequel by $\nabla_{\rho}f  \apt\gamma(a_1),\ldots, \gamma( a_{N}), \gamma  \cpt $, $\rho\in [0,T]$.\\

\begin{ese}  \label{esempio1}
Let $a_0=0< a_1<\ldots< a_{N}=T$.\\
Let $\Phi(\gamma)=f\apt \gamma(a_1),\ldots,\gamma(a_N) ,\gamma \cpt$ 
with $f: \R^{N}\times L^{2}([0,T])\lra \R$, $f\in C^{1}\apt  \R^{N}\times L^{2}([0,T])\lra \R\cpt$ such that all the derivatives have polynomial growth. 
In this case $D \Phi$ has the following particular form
\be		\label{tre}
D_{dr}\Phi(\gamma)=
\sum_{i=1}^{N}D^{\delta_{a_i}}\Phi(\gamma) \ \delta_{a_i}(dr)+\apt D_r^{ac}\Phi\cpt (\gamma)dr \ ,  \textrm{ with }
\left\{
\begin{array}{l}
D^{\delta_{a_i}}\Phi(\gamma) =\partial_{i } f\apt \gamma(a_1),\ldots,\gamma( a_N),\gamma \cpt \\
\\
\apt D_r^{ac}\Phi\cpt (\gamma)dr=\apt \nabla_r f\cpt \apt  \gamma(a_1),\ldots,\gamma( a_N),\gamma \cpt dr
\end{array}
\right.
\ee
where the $D_r^{ac}u(t,\eta)$ denotes the absolute continuous part of the measure $D\Phi $.
In particular  $D\Phi \in \mathcal{D}_{a}([0,T]) \oplus L^{2}([0,T])$.\\
By \eqref{eq 4.1} and \eqref{eq4.2} in Theorem \ref{thm META}, it yields 
\be		\label{eq5}
\begin{split}
D^{\delta_{0}}u(s,\eta)&=
\E
\apq
\sum_{a_i\geq s}D^{\delta_{a_i}} \Phi \apt Y^{s,\eta}\cpt  
\mathcal{E}\left\{ \int_s^{a_i} \partial_x\sigma \apt \xi,Y_\xi^{s,\eta} \cpt  dW_\xi +
 \int_s^{a_i} \partial_x b \apt \xi,Y_\xi^{s,\eta} \cpt  d\xi   \right\}  
 \cpq\\
&+
\E
 \apq \int_{s}^{T} \nabla_{\rho} \Phi\apt Y^{s,\eta}\cpt 
\mathcal{E}\left\{ \int_s^{\rho} \partial_x\sigma \apt \xi,Y_\xi^{s,\eta} \cpt  dW_\xi +
 \int_s^{\rho} \partial_x b \apt \xi,Y_\xi^{s,\eta} \cpt  d\xi   \right\}  \ d\rho
\cpq\\
&
=\E\apq \sum_{a_i\geq s} \partial_{i}f\apt \eta(a_1-s),\ldots, \eta(a_{i-1}-s), Y^{s,\eta}_{a_i},\ldots, Y^{s,\eta}_{a_N},Y^{s,\eta}  \cpt \right. \\
& \hspace{2cm}\left.
\mathcal{E}\left\{ \int_s^{a_i} \partial_x\sigma \apt \xi,Y_\xi^{s,\eta} \cpt  dW_\xi +
 \int_s^{a_i} \partial_x b \apt \xi,Y_\xi^{s,\eta} \cpt  d\xi   \right\}   \cpq\\
&
+
\E \apq  \int_{s}^T \nabla_{\rho}f  \apt \eta(a_1-s),\ldots, \eta(a_{i-1}-s), Y^{s,\eta}_{a_i},\ldots, Y^{s,\eta}_{a_N},Y^{s,\eta}  \cpt  \right. \\
& \hspace{2cm}\left.
\mathcal{E}\left\{ \int_s^{\rho} \partial_x\sigma \apt \xi,Y_\xi^{s,\eta} \cpt  dW_\xi +
 \int_s^{\rho} \partial_x b \apt \xi,Y_\xi^{s,\eta} \cpt  d\xi   \right\}   d\rho \cpq  \ .
\end{split}  
\ee
and 
\be		\label{eq6}
\begin{split}
D^{\perp}_{dr} u(s,\eta)&=
\E\apq \sum_{a_i< s } D^{\delta_{a_i}} \Phi \apt Y^{s,\eta}\cpt \delta_{a_i-s}(dr)\cpq 
+
 \E\apq \int_{0}^s \nabla_{\rho} \Phi\apt Y^{s,\eta}\cpt  \delta_{\rho-s}(dr) \ d\rho \cpq\\
&
= \E\apq \sum_{a_i < s} \partial_{i}f\apt \eta(a_1-s),\ldots, \eta(a_{i-1}-s), Y^{s,\eta}_{a_i},\ldots, Y^{s,\eta}_{a_N},Y^{s,\eta}  \cpt  \delta_{a_i-s}(dr) \cpq\\
&+
\E\apq 1_{[-s,0]}(r)  \nabla_{r +s }f\apt \eta(a_1-s),\ldots, \eta(a_{i-1}-s), Y^{s,\eta}_{a_i},\ldots, Y^{s,\eta}_{a_N},Y^{s,\eta}  \cpt  dr \cpq
\end{split}
\ee
where $\nabla f$ denotes the Fr\'echet derivative of $f$ with respect
to the last argument which is absolutely continuous. In fact
 $\nabla_\cdot f $ is a function in $L^{2}([0,T])$ (for
any argument of $f$)
defined by $r\mapsto \nabla_r f (\cdot)$. 
\end{ese}

\begin{rem}	\label{rem RES1}
Let $ j \in \{ 1,\ldots, N\}$. 
\begin{enumerate}
\item If $s \in ]a_{j-1}, a_{j}]$ we have 
\be		\label{eq5der}
\begin{split}
D^{\delta_0}u(s,\eta) 
&
= \E \apq \sum_{i=j}^{N} D^{\delta_{a_i}} \Phi \apt Y^{s,\eta}\cpt  
\mathcal{E}\left\{ \int_s^{a_i} \partial_x\sigma \apt \xi,Y_\xi^{s,\eta} \cpt  dW_\xi +
 \int_s^{a_i} \partial_x b \apt \xi,Y_\xi^{s,\eta} \cpt  d\xi   \right\}   \cpq \\
&+\E \apq \int_{s}^{T} \nabla_{\rho} \Phi\apt Y^{s,\eta}\cpt 
\mathcal{E}\left\{ \int_s^{\rho} \partial_x\sigma \apt \xi,Y_\xi^{s,\eta} \cpt  dW_\xi +
 \int_s^{\rho} \partial_x b \apt \xi,Y_\xi^{s,\eta} \cpt  d\xi   \right\}   \ d\rho
\cpq
\end{split}
\ee
and
\be
\begin{split}		\label{eq5der1}
D_{dr}^{\perp}u(s,\eta)  &=
\E\apq \sum_{i=1}^{j-1} D^{\delta_{a_i}} \Phi \apt Y^{s,\eta}\cpt \delta_{a_i-s}(dr)\cpq \\
&+ \E\apq \int_{0}^s \nabla_{\rho} \Phi\apt Y^{s,\eta}\cpt  \delta_{\rho-s}(dr) \ d\rho \cpq \ .
\end{split}
\ee
In particular for fixed $\eta\in C([-T,0])$, $s\mapsto D^{\delta_0}u(s,\eta)$ is continuous on $]a_{j-1}, a_{j}[$, is left-continuous and it admits a continuous extension to $[a_{j-1}, a_j]$.
\item If $s\in [0,a_1]$, the \eqref{eq5der} holds with $j=1$ and $s\mapsto D^{\delta_0}u(s,\eta)$ is continuous on $[0,a_1]$.
\item We remark that $u$ is not necessarily of class $C^{0,1}\apt [0,T]\times C([-T,0])\cpt$ excepted if $N=1$.
\end{enumerate}
\end{rem}
The following is a particular case of Example \ref{esempio1} where $\Phi$ only depends on the maturity 
and on the whole trajectory in $L^2$. 

\begin{ese} \label{ese2}
Suppose that $\Phi \apt \gamma \cpt =f\apt \gamma(T),\gamma \cpt$ for 
$f:\R\times L^{2}([-T,0])\lra \R$ of class $C^1$ with polynomial growth derivatives.  
\eqref{eq5} and \eqref{eq6} give 
\[	
\begin{split}
D^{\delta_0}u(s,\eta)
&
=\E\apq  \partial_{1}f\apt Y^{s,\eta}_{T},Y^{s,\eta}  \cpt
\mathcal{E}\left\{ \int_s^{T} \partial_x\sigma \apt \xi,Y_\xi^{s,\eta} \cpt  dW_\xi +
 \int_s^{T} \partial_x b \apt \xi,Y_\xi^{s,\eta} \cpt  d\xi   \right\}   \cpq\\
&
+
\E \apq  \int_{s}^T \nabla_{\rho}f  \apt Y^{s,\eta}_{T},Y^{s,\eta}  \cpt  
\mathcal{E}\left\{ \int_s^{\rho} \partial_x\sigma \apt \xi,Y_\xi^{s,\eta} \cpt  dW_\xi +
 \int_s^{\rho} \partial_x b \apt \xi,Y_\xi^{s,\eta} \cpt  d\xi   \right\}   d\rho \cpq  \ .
\\
D^{\perp}_{dr}u(s,\eta)
&=
\E\apq 1_{[-s,0]}(r)  \nabla_{r +s }f\apt Y^{s,\eta}_{T},Y^{s,\eta}  \cpt  dr \cpq  \ .
\end{split}  
\]
In this case $u$ is of class $C^1$ for every $(s,\eta)\in [0,T]\times C([-T,0])$, $Du(s,\eta) \in  \mathcal{D}_0\oplus L^{2}([-T,0])$ and 
$Du:[0,T]\times C([-T,0])\lra \mathcal{D}_{0}\oplus L^{2}([-T,0])$ is continuous.
\end{ese}
The following is a particular case of Example \ref{esempio1} where $\Phi$ only depends pointwise on a finite number of points; 
in this case $\nabla_r f \equiv 0$. 
\begin{ese}		\label{ese3}
Let $\Phi(\gamma)=f\apt \gamma(a_1),\ldots,\gamma(a_N) \cpt$ with $f: \R^{N} \lra \R$. 
In this case \eqref{tre}  reduces to
$
D_{dr}\Phi(\gamma)=\sum_{i=1}^{N}\partial_{i } f\apt \gamma(a_1),\ldots,\gamma( a_N) \cpt \delta_{a_i}(dr) 
$. 
\eqref{eq5} reduces to 
\be	\label{eq5.2}
D^{\delta_{0}}u(s,\eta)=
\E
\apq
\sum_{a_i\geq s (i\geq 1)} \partial_{i } f\apt Y^{s,\eta}_{a_1},\ldots, Y^{s,\eta}_{a_{N}}\cpt 
\mathcal{E}\left\{ \int_s^{a_i} \partial_x\sigma \apt \xi,Y_\xi^{s,\eta} \cpt  dW_\xi + \int_s^{a_i} \partial_x b \apt \xi,Y_\xi^{s,\eta} \cpt  d\xi  \right\}
\cpq\ .
\ee
We recall that, if $s\in [a_{N-1},a_N[$ then $s\geq a_i$ for all $i=1,\ldots,(N-1)$ and $Y^{s,\eta}_{a_i}=\eta(a_i-s)$, by definition \eqref{eq FLUSSO}.
\end{ese}
The result below is a fundamental step for obtaining a quasi-explicit representation of integrable random variables. 
In fact we verify the validity of Corollary \ref{cor CCS}. 
\begin{prop}		\label{pr PCS}
Let $\Phi:C([0,T])\lra \R$ of the type 
\[
\Phi(\gamma)=f\apt  \gamma(a_1),\ldots, \gamma( a_{N}), \gamma \cpt
\] 
with $f: \R^{N}\times L^{2}([0,T])\lra \R$ 
of class $C^{1}$ with partial derivatives having polynomial growth. Let $\Phi:C([0,T])\lra \R$ defined by 
$\Phi(\gamma)=f\apt \gamma(a_1),\ldots,\gamma(a_N) ,\gamma \cpt$. 
We set $u(s,\eta)=\E\left[ \Phi \apt Y^{s,\eta}\cpt \right]$ for $s\in [0,T]$ and $\eta\in C([-T,0])$.
Then 
\be		\label{eqFG12}
h=u(0,X_{0}(\cdot))+\int_{0}^{T} D^{\delta_{0}}u \apt s, X_{s}(\cdot) \cpt \sigma\apt s, X_s \cpt dW_s   \ ,
\ee
\end{prop}
where $D^{\delta_0}u(s,\eta)$ is given in \eqref{eq5}.
\begin{proof}
We verify the assumptions 
of Corollary \ref{cor CCS}.
$u$ is of course continuous. e) is fulfilled by construction. b) was the object of Theorem \ref{thm META}. 
In Example \ref{esempio1} we have given the expression of $Du(s,\cdot)$ for any $s\in [0,T]$.
\eqref{eq5} and \eqref{eq6} give the explicit expression respectively for $D^{\delta_0}u(s,\eta)$ 
and $D^{\perp}_{dr}u(s,\eta)$.
c) follows from those explicit expressions. d) follows by the fact that 
$(t,\eta)\mapsto D^{\perp}u(t,\eta)$ is bounded on each compact of $[a_i,a_{i+1}]\times C([-\tau,0])$, $1\leq i\leq N-1$. 
It remains to check a), i.e. if $u$ fulfills the 
support predictability property.
Since 
$$
(r,s,\eta)\mapsto \E \apq 1_{[-s,0]}(r)\nabla_{r+s} f \apt Y^{s,\eta}_{a_1}, \ldots ,\ldots, Y^{s,\eta}_{a_N},Y^{s,\eta}  \cpt   \cpq
$$ 
is bounded on each compact, 
in order to verify the support predictability condition we only need to show that for each compact $K\subset C([-T,0])$, 
\be		\label{eq FF}
\int_{0}^{T} \apq \sup_{\eta\in K}  \frac{1}{\e} \int_{-\e}^{0}
g\apt dr, s, \eta \cpt 
\cpq ds =O(\e)
\ee
where 
\[
g\apt dr, s, \eta \cpt=\sum_{a_j<s}\E\apq \partial_{j}f \apt  Y^{s,\eta}_{a_1}, \ldots ,\ldots, Y^{s,\eta}_{a_N},Y^{s,\eta} \cpt  \cpq \delta_{a_j-s}(dr) \ .
\]
For $\e< \min_{j\in\{ 1,\ldots, N\} }\{a_{j}-a_{j-1} \}$, the left-hand side of \eqref{eq FF} gives 
\be \label{eq RF}
\begin{split}
&
\sum_{i=0}^{N-1}\int_{a_i}^{a_{i+1}}\apq \sup_{\eta\in K}  \frac{1}{\e} \int_{-\e}^{0}
\sum_{a_j<s}\E\apq \partial_{j}f \apt  Y^{s,\eta}_{a_1}, \ldots ,\ldots, Y^{s,\eta}_{a_N},Y^{s,\eta} \cpt  \cpq \delta_{a_j-s}(dr)  
\cpq  ds\\
&
=\sum_{i=0}^{N-1}\int_{a_i}^{a_{i+1}}\apq \sup_{\eta\in K}  \frac{1}{\e} 
\sum_{a_j<s}\E\apq \partial_{j}f \apt  Y^{s,\eta}_{a_1}, \ldots ,\ldots, Y^{s,\eta}_{a_N},Y^{s,\eta} \cpt  \cpq  \int_{-\e}^{0} \delta_{a_j-s}(dr) 
\cpq  ds.
\end{split}
\ee
We remark that 
\be
\int_{-\e}^0\delta_{a_j-s}(dr)=
\left\{
\begin{array}{ll}
1 & \textrm{ if } ( a_j-s ) \in [-\e,0]\Leftrightarrow s\in [a_j,a_j+\e]\\
0 & \textrm{ otherwise}  
\end{array}
\right.  \ .
\ee
So in the second sum of \eqref{eq RF}, for $a_j< s$ only remains the term $a_j=a_i$ and \eqref{eq RF} gives 
\[
\sum_{i=0}^{N-1}\int_{a_i}^{a_i+\e}\apq \sup_{\eta\in K}  \frac{1}{\e} 
\E\apq \partial_{i}f \apt  Y^{s,\eta}_{a_1}, \ldots ,\ldots, Y^{s,\eta}_{a_N},Y^{s,\eta} \cpt  \cpq  \cpq  ds  \ .
\]
Since all the derivatives of $f$ have polynomial growth, previous expression is bounded by 
\[
C(N,f,T)\ \sup_{\eta\in K, s\in [0,T]}\left\{  1+\E\apq \left\| Y^{s,\eta}\right\|_{\infty}^q \cpq  \right\}
\]
where $C(N,f,T)$ is some constant and $q$ is some positive real. The conclusion follows by Proposition \ref{PPP}.
So $u$ fulfills the support predictability condition.
\end{proof}
\subsection{About the representation of non-smooth random variables}		\label{esemSINGOLARE}
The next example is essentially  illustrative. It will be developed
and treated 
in a more general context in a paper in preparation.
It will be possible to represent, still using Corollary \ref{cor CCS}, random variables of the type 
\[
h:=\Phi(X) \ ,\quad \textrm{where } \quad \Phi(\gamma)=f\apt 
\int_0^T \varphi_1(r) d\gamma(r), \ldots, \int_0^T \varphi_N
(r)d\gamma (r) 
\cpt  \ ,
\] 
$f$ continuous with polynomial growth and $\varphi_1, \ldots, \varphi_{N}$ of class $C^1([0,T])$. 
$\int_{0}^T\varphi_{i}(r)d\gamma(r)$ are naturally defined by an
 integration by parts as 
$\varphi(T)\gamma(T)-\varphi(0)\gamma(0)-\int_{0}^T \gamma(r)d\varphi(r)=\varphi(T)\gamma(T)-\int_{0}^T \gamma(r)d\varphi(r)$. 
More general formulations can be performed even with less regularity 
using specific approximation techniques. 
Here we only aim at obtaining a representation directly, without approximations. 
For simplicity we consider $X$ to be a diffusion process of type \eqref{eq SDE} with $\sigma\equiv 1$ and $b\equiv 0$, so that $X$ is a Brownian motion $W$ and 
$Y^{s,\eta}$ introduced in \eqref{eq FLUSSO} will be a Brownian flow. \\
 In this illustrative subsection we only suppose $N=1$. 
Let $h=f\apt \int_0^T\varphi(r)dW_r \cpt$, $\varphi\in C^{1}([0,T])$, $\varphi(T)\neq 0$, $f:\R\lra \R$ measurable with polynomial growth. 
We define $u(s,\eta)=\E\apq \Phi(Y^{s,\eta})\cpq$ as in \eqref{eq U}.
It is possible to show the assumptions of Theorem \ref{thm:thm1B} setting $a=0$ and $b<T$. We suppose for a moment that 
$f\in C^{1}_{b}(\R)$. 
We come back to Theorem \ref{thm META} 
which says that that for every $s\in [0,T]$, 
$u(s,\cdot)$ is of class $C^{1}$ Fr\'echet and
\[
D_{dr}u(s,\eta)= D^{\delta_0}u(s,\eta)\delta_{0}(dr)+D^{\perp}_{dr}u(s,\eta) \ .
\]
Since 
\be
D_{dr}\Phi(\gamma)=f'\apt \int_0^T
\varphi (r)d\gamma(r) \cpt
\apt \varphi (T)\delta_{T}(dr)-\varphi(0)\delta_{0}(dr)
-\dot{\varphi}(r)1_{[0,T]}(r)dr \cpt \ .
\ee
Taking into account \eqref{eq4.2} we get 
\be		\label{eqRT1}
\begin{split}
D^{\delta_0} u(s,\eta)&=\E\apq Z_\varphi\cpq \apq \varphi(s)-\varphi(0)1_{\{0\}}(s) \cpq \\
D^{\perp}_{dr}u(s,\eta)&=\E\apq Z_\varphi\cpq \apq \apt -\varphi(0)\delta_{-s}(dr)-\dot\varphi(s+y) \cpt 1_{[-s,0[}(y) dy  \cpq  \ ,
\end{split}
\ee
where 
\be	\label{eqZ}
Z_{\varphi}=f'\apt \int_0^T
\varphi(r)dY^{s,\eta} (r) \cpt  \ .
\ee
Using Malliavin calculus and the fact that
\[
Z_{\varphi}=f'\apt \eta(0)\varphi(s)-\varphi(0)\eta(-s) -
\int_0^s
\eta(r-s)d\varphi(r)+\int_s^T \varphi(r)dW_r \cpt  \ .
\] 
We set the random variable $h=\Phi(Y^{s,\eta})$.
Denoting $D^m$ the Malliavin derivative we obtain 
\be
D^m_r (h)= D^m_r \apt \Phi(Y^{s,\eta})\cpt
=Z_\varphi \varphi(r)1_{]s,T]}(r)  \  ,
\ee
then 
\be
\langle D^m (h ), \varphi\rangle_{L^{2}([0,T])}=
\int_{0}^{T} D^m_r (h) \varphi(r) dr=Z_\varphi \int_s^T \varphi^{2}(r)dr 
\ee
and so 
\be	\label{eqRT}
\E \apq Z_\varphi \cpq=\frac{1}{\int_s^T \varphi^{2}(r)dr } \E\apq \langle D^m (h) , \varphi \rangle  \cpq.
\ee
Consequently, using \eqref{eqRT} in \eqref{eqRT1} we obtain
\be		\label{eqF10}
D^{\delta_0} u(s,\eta)= \frac{\varphi(s)-
\varphi(0)1_{\{0\}}(s)}{\int_s^T 
\varphi^{2}(r)dr }\E\apq \langle D^m (h) , \varphi \rangle \cpq 
\ee
and 
\be		\label{eqF11}
D^{\perp}_{dy}u(s,\eta)= \frac{\apt -\varphi(0)\delta_{-s}(dr)-
\dot\varphi(s+y) \cpt 1_{[-s,0[}(y) dy}{\int_s^T  \varphi^{2}(r)dr }\E\apq \langle D^m (h) , \varphi \rangle \cpq   \ .
\ee
By the well-known integration by parts on the Wiener space, it follows 
\be
\E \apq \langle D^m (h) , \varphi \rangle  \cpq = \E\apq  h \cdot \delta(\varphi) \cpq =  \E\apq  \Phi(Y^{s,\eta}) \cdot \int_{0}^T\varphi(r) dW_r \cpq  \ .
\ee
Consequently, for $s\in [0,T[$
\be		\label{eqF1}
D^{\delta_0}u(s,\eta)=
 \frac{\varphi(s)-\varphi(0)\1_{\{0\}}(s)} {\int_s^T  \varphi^{2}(r)dr
}
\E\apq  \Phi(Y^{s,\eta}) \cdot \int_{0}^T \varphi(r) dW_r \cpq 
\ee
\be			\label{eqF2}
D^{\perp}_{dy}u(s,\eta)= \frac{\apt -\varphi(0)\delta_{-s}(dr)-
\dot\varphi(s+y) \cpt 1_{[-s,0[}(y) dy} {\int_s^T \varphi^{2}(r)dr }
 \E\apq  \Phi(Y^{s,\eta}) \cdot \int_{0}^T\varphi(r) dW_r \cpq  \ .
\ee
\begin{rem}
We remark that in \eqref{eqF1} and \eqref{eqF2} does not appear the derivative of $f$. At this point we admit a technical point not to overcharge the proof.
Even if $f$ is not of class $C^1$, $u(t,\cdot)$ is still of class $C^1$ for every $s\in [0,T[$ and \eqref{eqF1} and \eqref{eqF2} still hold.
\end{rem}
As promised, we verify the assumptions of Theorem \ref{thm:thm1B}.
 We first observe that 
\be		\label{eqF3}
(s,\eta)\mapsto \E \apq  \Phi(Y^{s,\eta}) \cdot \int_{0}^T\varphi(r) dW_r  \cpq 
\ee
is continuous, therefore bounded on each compact of 
$[0,T]\times C([-T,0])$. Given a compact subset $K$ of
$C([-T,0])$,
we denote
\[
C(K):=\sup_{\eta\in K} \E\apq \Phi(Y^{s,\eta}) \cdot \int_{0}^T\varphi(r) dW_r  \cpq  \ .
\]
Assumptions i), ii) and iv) are clearly verified taking into account \eqref{eqF3}, \eqref{eqF10} and \eqref{eqF11}. 
It remains to check the support predictability property. Let $K$ be a compact of $C([0,T])$. 
Since 
\[
\sup_{y\in [0,b], \eta \in K} \left| \frac{ \dot\varphi(s+y)  
1_{[-s,0[}(y) }{\int_s^T
\varphi^{2}(r)dr }   
\E\apq  \Phi(Y^{s,\eta}) \cdot \int_{0}^T\varphi(r) dW_r \cpq \right|
\]
is finite, Remark \ref{remGBN}, item 2. implies that it will be enough to show that 
\be		\label{eqF4}
\int_{0}^b \sup_{\eta\in K}\frac{1}{\e} \int_{(-\e)\vee (-T)}^{0}
\left| 
  \frac{\varphi(0)\delta_{-s}(dr) 1_{[-s,0[}(y) dy}{\int_s^T
\varphi^{2}(r)dr }
 \E\apq  \Phi(Y^{s,\eta}) \cdot \int_{0}^T\varphi(r) dW_r \cpq \right| ds=O(\e)  \ .
\ee
Let $\e <b$. The left-hand side of \eqref{eqF4} is bounded by 
\be
\frac{C(K)}{\e}
\int_0^\e \frac{\varphi(0)}{ \int_s^T \varphi^{2}(r)dr} ds 
\leq 
\frac{C(K)}{ \int_b^T \varphi^{2}(r)dr}\| \varphi\|_{\infty} <+\infty \ .
\ee
So also assumption iii) of Theorem \ref{thm:thm1B}, i.e. the support predictability property is verified. 
The conclusion follows by Remark \ref{rem5.19}.
%
%
%
%
%
%
\subsection{Link to a finite dimensional PDE}   \label{secPDE}

We come back to the assumptions on coefficients $\sigma$
and $b$ of Section \ref{sec: last}.
They characterize again a diffusion process $X$ as solution
of \eqref{eq SDE}.
We link now Corollary \ref{cor CCS} and Proposition \ref{pr PCS} with a well-known result of representation related to hedging theory in mathematical finance. 
We consider a \emph{contingent claim} defined by 
$h=\Phi(\gamma)=f\apt \gamma(a_1),\ldots,\gamma(a_N) \cpt$, i.e.
\be
h=f\apt X_{a_1}, \ldots, X_{a_N} \cpt  \ , \quad 0< a_1< \ldots < a_N=T
\ee
with the usual convention $a_0=0$.
We consider here the case $f$ of class $C^{2}$ with polynomial growth but $\sigma$ may become degenerate.
\begin{prop}	\label{prop R1}
Let $X$ be a diffusion process of type \eqref{eq SDE}. 
Let $N\geq 2$ and $f:\R^{N}\lra \R$ of class $C^{2}$ with polynomial growth.
We suppose $\sigma :[0,T]\times \R\lra \R$ of class $C^{0,2}$, such that $\partial_{x}\sigma$, $\partial_{xx}^2\sigma$ are bounded.\\
There exist functions 
\[
\nu^{i}:\R^{i-1}\times [a_{i-1},a_{i}]\times \R\lra \R \ , \quad 1\leq i\leq N \ , 
\]
such that for $y_1, \ldots, y_{i-1}\in \R$, 
\[
\nu^{i}(s,y):= \nu^{i}\apt y_1, \ldots, y_{i-1}; s, y\cpt  
\]
solves 
\be		\label{eq R1}
\left\{
\begin{array}{ll}
\partial_{s}\nu^{i}(s,y)+\frac{1}{2}\sigma^{2}(s,y)\ \partial^{2}_{yy}\nu^{i}(s,y)+b(s,y) \partial_{y}\nu^{i}(s,y)=0 & s\in ]a_{i-1},a_i[\\
& \\
\nu^{i}(a_i,y)=\nu^{i+1}\apt y_1, \ldots, y_{i-1}, y; a_i, y \cpt	&i< N\\
& \\
\nu^{N}(T,y)=f \apt  y_1, \ldots, y_{N-1}, y \cpt  & i=N
\end{array}
\right. 
\ee
such that 
\be 		 \label{eq R12}
f\apt X_{a_1}, \ldots, X_{a_N} \cpt =H_0+\int_0^T \xi_s dX_s
\ee
and, for $1\leq i\leq N$, 
\be			\label{eq R13}
\begin{array}{ll}
\xi_s=\partial_y\nu^{i}\apt X_{a_1}, \ldots, X_{a_{i-1}};s,X_s  \cpt		\ ,&  \quad s\in ]a_{i-1},a_i[\\
&\\
H_{0}=\nu^{1}\apt 0,X_0\cpt & \\
\end{array}	\ .
\ee
In particular $(s,y)\mapsto \nu^{i}(s,y)\in C^{1,2}\apt ]a_{i-1},a_i[\times \R\cpt\cap C^{0}\apt [a_{i-1},a_{i}]\times \R\cpt$.
\end{prop}

\begin{rem}	\label{CovDGR}
Let $X$ be a finite quadratic variation process such that $[X]_t=\int_{0}^t \sigma^{2}\apt s,X_s\cpt ds$. An example is of course our basic 
process $X$ introduced in \eqref{eq SDE} but there are plenty of other examples. Let $f:\R^N\lra \R$ be continuous with 
linear growth. If there are functions $\nu^{1},\ldots, \nu^{N}$ as in Proposition \ref{prop R1}, then 
representation \eqref{eq R12}, with \eqref{eq R13}, holds replacing
$dX_s$ with the forward integral $d^-X_s$. 
If $X$ is a martingale we recall that $d^-X_s=dX_s$. This was proved in \cite{crnsm2}, Proposition 4.30. 
The aim of Proposition \ref{prop R1} is to construct effectively such functions $\nu^1,\ldots, \nu^{N}$.
\end{rem}

\begin{proof}[Proof of Proposition \ref{prop R1}]
According to Corollary \ref{cor CCS}, representation \eqref{eq R12} holds with $H_0=u\apt 0,X_0(\cdot) \cpt$ and 
$\xi_{s}=D^{\delta_0}u\apt s,X_s(\cdot)\cpt$ where $u(s,\eta):=\E\left[ f\apt Y^{s,\eta}_{a_1},\ldots, Y^{s,\eta}_{a_N}\cpt \right]$. 
In fact $\Phi(\gamma)=f\apt \gamma(a_{1}),\ldots, \gamma(a_{N})\cpt$ is $C^{1}$ Fr\'echet differentiable with polynomial growth.\\
We denote by $\apt Z^{s,y}_t\cpt$ the flow such that $Z=Z^{s,y}$ verifies 
\[
Z_t=y+\int_{s}^{t} \sigma\apt r,Z_r\cpt dW_r +\int_{s}^{t}b\apt r, Z_r \cpt dr\ , \quad t\geq s \ .
\]
In particular we have $Y^{s,\eta}=Z^{s,\eta(0)}$.\\
Let $1 \le i \le N$.
Clearly, if $s\in [a_{i-1},a_i]$,
\be		\label{eq R14}
\begin{split}
u(s,\eta)&=\E \left[  f \apt \eta(a_1-s),\ldots, \eta(a_{i-1}-s) , Z^{s,\eta(0)}_{a_i},\ldots, Z^{s, \eta(0)}_{a_N} \cpt\right]\\
&
=\nu^{i}\apt \eta (a_{1}-s) ,\ldots, \eta( a_{i-1}-s); s, \eta(0) \cpt
\end{split}
\ee
where 
\be		\label{eq R14bis}
\nu^{i}(s,y)=\E\left[  f\apt y_1,\ldots, y_{i-1},Z^{s,y}_{a_i},\ldots, Z^{s,y}_{a_N}  \cpt \right]   \ .
\ee
Now $\nu^{i}$ is continuous. We keep in mind the representation \eqref{eqFG12} in Proposition \ref{pr PCS}. In particular \eqref{eq R12} holds with 
$H_0=u\apt 0, X_0(\cdot)\cpt$ and $\xi_s=D^{\delta_0}u \apt s,X_s(\cdot)\cpt$.
We recall the expression of $D^{\delta_0}u$ calculated in Example \ref{esempio1} in a more general situation. 
We had 
\[
\begin{split}
D^{\delta_0}u(s,\eta)\ 1_{]a_{i-1},a_i[}(s)
&
=\E\left[ \sum_{a_j\geq s}f \apt \eta(a_1-s),\ldots, \eta(a_{i-1}-s),Z^{s,\eta(0)}_{a_i},\ldots, Z^{s,\eta(0)}_{a_N} \cpt   \right.\\
&\hspace{2cm}\left.
\mathcal{E}\left\{ \int_{s}^{a_j}\partial_x\sigma\apt r, Z^{s,\eta(0)}_{r} \cpt dW_r + \int_{s}^{a_j}\partial_x b \apt r, Z^{s,\eta(0)}_{r} \cpt dr
\right\} \right]\\
\end{split}
\]
where $\mathcal{E}$ denotes the Dol\'eans exponential operator as usual.\\
In fact by usual integration theory it is not difficult to show that
for any
 $s\in ]a_{i-1},a_i[$, $y_1,\ldots, y_{i-1} \in \R$ fixed, 
$y\mapsto \nu^{i}(s,y)$ is of class $C^2$. We observe 
that $u\apt 0, X_0(\cdot)\cpt=\nu^1(0,X_0)$.
Moreover $s\mapsto \partial_{y}\nu^{i}(s,y)$ and $s\mapsto \partial_{yy}^2\nu^{i}(s,y)$ are continuous on $]a_{i-1},a_i[$. 
In particular if $s\in ]a_{i-1},a_i[$ 
\[
\begin{split}
\partial_{y}\nu^{i}(s,y)
&=
\E\left[ \sum_{j\geq i} \partial_{j}f\apt y_{1},\ldots, y_{i-1}, Z^{s,y}_{a_{i}},\ldots, Z^{s,y}_{a_N}\cpt   \partial_{y}Z^{s,y}_{a_j} \right]  \ .
\end{split}
\]
Since, at least in $L^{2}(\Omega)$,
\[
\partial_{y}Z^{s,y}_{a_j}=\mathcal{E} \apt \int_{s}^{a_j} \partial_{x}\sigma \apt r, Z^{s,y}_{r}\cpt dW_r \cpt \ , 
\]
it follows that, for $s\in ]a_{i-1},a_i[$, 
\[
D^{\delta_0} u (s, \eta) =\partial_{y}\nu^{i}\apt \eta (a_1-s),\ldots,\eta(a_{i-1}-s);s,\eta(0) \cpt
\]
and so \eqref{eq R13} is established.\\
It remains to prove \eqref{eq R1}.\\
We remark that we can evaluate the second order derivative with respect to $y$. It gives 
\[
\begin{split}
\partial_{yy}^{2}\nu^{i}(s,y)
&
=
\E\left[ \sum_{k,j\geq i} \partial_{k j}^{2}f \apt y_1,\ldots, y_{i-1},Z^{s,y}_{a_i},\ldots, Z^{s,y}_{a_N}\cpt 
\partial_{y}Z^{s,y}_{a_j}\partial_{y}Z^{s,y}_{a_k}
\right]\\
&
+\E
\left[
\sum_{j\geq i}
\partial_{j}f \apt y_1,\ldots, y_{i-1},Z^{s,y}_{a_i},\ldots, Z^{s,y}_{a_N}\cpt 
\partial_{yy}^2 Z^{s,y}_{a_j}
\right]
\end{split}
\]
where $\partial_{yy}^2 Z^{s,y}_{a_j}$ could be calculated explicitly.\\ 
It remains to provide the partial derivative with respect to $s$. 
Let $s, s+h \in ]a_{i-1},a_i[$ and suppose $h>0$. Since 
$Z^{s,y}_{a_i}=Z^{s+h,Z^{s,y}_{s+h}}_{a_i}$, we easily obtain that 
\[
u(s,y)=\E\left[  u\apt s+h,Z^{s,y}_{s+h}\cpt \right]  \ .
\]
So, by It\^o formula, 
\[
\begin{split}
\frac{\nu^{i}(s+h,y)-\nu^{i}(s,h)}{h}&=
\frac{u(s+h,y)-u(s,y)}{h}\\
&=
\frac{1}{h}\E\left[ u(s+h,y)- u (s+h, Z^{s,y}_{s+h}\right]\\
& 
=\frac{1}{h}\E
\left[
-\int_{s}^{s+h} \partial_{y} u \apt s+h, Z^{s,y}_{r} \cpt dZ_{r}^{s,y}
-\frac{1}{2} \int_{s}^{s+h} \sigma^{2}\apt r, Z^{s,y}_{r }\cpt  \partial_{yy}^{2} u \apt s+h ,  Z^{s,y}_{r} \cpt dr
\right]
\end{split} 
\]
where $dZ_{r}^{s,y}=\sigma\apt r,Z_r^{s,y}\cpt dW_r +b\apt r, Z_r^{s,y} \cpt dr$.\\
Similar arguments allow to discuss the limit when $h\rightarrow 0$. Letting $h$ go to zero we get
\[
\begin{split}
\partial_{s}\nu^{i}(s,y)=\lim_{h\rightarrow 0}
\frac{u(s+h,y)-u(s,y)}{h}&=-\frac{1}{2}
\sigma^{2}\apt s, y\cpt  \partial_{yy}^{2} u \apt s ,  y \cpt   - b(s,y)\partial_y u(s,y)\\
&
=-\frac{1}{2}\sigma^{2}(s,y)\partial_{yy}^{2}\nu^{i}(s,y)-b(s,y)\partial_{y}\nu^{i}(s,y) \ .
\end{split}
\]
We have finally established the first line of 
\eqref{eq R1}. The second and third conditions in \eqref{eq R1} are verified by inspection using \eqref{eq R14bis}.
\end{proof}

\bigskip
{\bf ACKNOWLEDGEMENTS}:

The work of the second named author was partially supported
by the ANR Project MASTERIE 2010 BLAN 0121 01.  


\bibliographystyle{plain}
\bibliography{biblio}

\end{document}